\documentclass[a4paper,10pt]{amsbook}
\usepackage[centertags]{amsmath}
\usepackage{amsfonts}
\usepackage{amssymb}
\usepackage{amsthm}
\usepackage{newlfont}
\usepackage{graphicx}
\usepackage{epstopdf}
\usepackage{cite}
\usepackage{bm}
\newcommand{\Z}{\mathbb{Z}}

\newcommand{\bq }{\begin{equation}}
\newcommand{\eq }{\end{equation}}
\theoremstyle{plain}
\newtheorem{thm}{Theorem}
\newtheorem{lem}[thm]{Lemma}
\newtheorem{prop}[thm]{Proposition}

\newtheorem{cor}[thm]{Corollary}
\newtheorem{rem}[thm]{Remark}
\newtheorem{preg}[thm]{Question}
\theoremstyle{definition}
\newtheorem{defn}{Definition}[section]
\theoremstyle{example}

\newtheorem{claim}{Claim}
\title{On the automorphism group of non-singular plane curves fixing the degree}

\author[E. Badr] {Eslam Badr$^{1,2}$ and Francesc Bars$^{1}$}
\address{\begin{enumerate}
           \item Departament Matem\`atiques, Edif. C, Universitat Aut\`onoma de Barcelona
08193 Bellaterra, Catalonia
\item Department of Mathematics,
Faculty of Science, Cairo University, Giza-Egypt
         \end{enumerate} }
\email{eslam@mat.uab.cat\,(Badr E.) and francesc@mat.uab.cat (Bars F.)}

\thanks{E. Badr and F. Bars are supported by MTM2013-40680-P}


\begin{document}

\begin{abstract} This note is devoted, after the result of Harui
\cite{Harui}, to solve the following questions for non-singular
plane curve of degree $d$ over an algebraically closed field $K$ of
zero characteristic.
\begin{enumerate}
\item Given a finite non-trivial group $G$ that appears as an automorphism group of a plane non-singular curve $C$ inside
$\mathbb{P}^2$ of a fixed degree $d\geq4$. Is the defining equation
of such $C$ unique? i.e.\, is there, up to $K$-equivalence, a unique
homogenous polynomial $F_G(X;Y;Z)$ of degree $d$ that depends on a
set of parameters, as its monomials' coefficients with some
restrictions on these parameters, such that any $C$ as above is
given by the equation $F_G(X;Y;Z)=0$ in $\mathbb{P}^2(K)$ via a
certain specialization of the parameters and viceversa?
\par By the work of Henn \cite{He} and Komiya-Kuribayashi \cite{KK},
we conclude that that the answer of this question is positive for
$d=4$. That is, all non-singular plane quartics $C$ with
$Aut(C)\simeq G$ satisfy a unique equation (endowed with a set of
restrictions on the parameters), up to change of variables.

In Chapter 1, we show that for $d>4$ the answer is no more positive and counter examples are provided.
More precisely, if $G$ be a finite group that appears as the full automorphism group of a non-singular
 projective plane curve $C$ of degree $d>4$ over $K$,
 then it may happen that there are non isomorphic plane homogenious equations of degree $d$ (each of them given by a set of parameters which
 some restrictions in the parameters)
 such that their automorphism groups (when we specialize the parameters) are isomorphic to $G$.
 Moreover, the number of the different equations (up to $K$-isomorphism) is bounded above by the number of different
  conjugacy classes of the injective representations of $G$
inside $PGL_3(K)$, but this bound is usually far to be sharp.
Furthermore, we prove the same results when $K$ has a positive characteristic $p$ with $p> (d-1)(d-2)+1$ where $d$ is the degree of the curve.

\item In Chapter 2, we are concerned with the following classical problem. Fix a degree $d\geq 4$, which kind of groups could appear as the full automorphism groups of non-singular projective plane curves and what are the defining equations that are associated to each such group?

The results of Harui in \cite{Harui} simplify the problem to a certain list of finite groups that one should work with. These groups can be classified up to conjugation as; cyclic groups, an element of $Ext^1(-,-)$ of a cyclic group by a dihedral group, an alternating group $A_4$ (resp. $A_5$), an octahedral group $S_4$, a subgroup of $Aut(F_d)$ (the automorphism group of the Fermat curve), a subgroup of $Aut(K_d)$ (the automorphism group of the Klein curve) or a finite primitive subgroup of $PGL_3(K)$ (for the classification of such groups, we refer to Mitchell \cite{Mit}).

Now, we present in this section the list of the cyclic groups $G$ that appears as an
effective action of plane projective non-singular curves of degree $d$
over an algebraic closed field of zero characteristic and in particular, for any cyclic subgroup of
$Aut(C)$, its order should divide one of the following numbers:
$d-1$,$d$,$d^2-3d+3$, $(d-1)^2$, $d(d-2)$ or $d(d-1)$. Moreover, we attach to each cyclic group an equation $C$ (unique up to
$K$-isomorphism) that admits the cyclic group $G$ as a subgroup of the full
automorphism group. Also, we provide an algorithm of computation of the full list once the degree is fixed together with the implementation of this algorithm in SAGE program.

Furthermore, we determine the full automorphism group and a
characterization of the unique equation for the curve $C$ that admits an automorphism of order $d^2-3d+3$, $(d-1)^2$, $d(d-2)$ or $d(d-1)$ in
$Aut(C)$. We also consider the cases when $Aut(C)$ has a cyclic subgroup of
order $\ell d$ or $\ell (d-1)$ with $1\leq \ell< d-2$.

\item In Chapter 3, we give the analogy of the results of Henn in \cite{He} and Komiya-Kuribayashi in \cite{KK} concerning the full automorphism groups of non-singular plane curves of degree 4. Thus, we determine the automorphism groups that appear for projective non-singular, degree 5 plane curves which are non-hyperelliptic and the list of the defining equations of these curves. Similar arguments can be
applied to deal with higher degrees.
\end{enumerate}

\vspace{.05cm}

It should be noted that, we present the above three chapters as three independent notes in
the context. Hence, they might be read separately without any inconvenience.

The subject of this note is very classical, and we wrote them for a
lack of a precise reference, as far as we know. Therefore, any
further information, missing references, comments or
suggestions are welcomed.

Moreover, the above results can be reformulated in the more precise language of the moduli space (which
will be a forthcoming updated version of chapters 2 and 3), as follows.

Let $M_g$ be the moduli space of smooth, genus $g$ curves over an
algebraically closed field $K$ of zero characteristic. Denote by
${M_g(G)}$ the subset of $M_g$ of curves $\delta$ such that $G$ (as
a finite non-trivial group) is isomorphic to a subgroup of
$Aut(\delta)$ and let $\widetilde{M_g(G)}$ be the subset of curves
$\delta$ such that $G\cong Aut(\delta)$. Now, for an integer $d\geq
4$, let $M_g^{Pl}$ be the subset of $M_g$ representing smooth, genus
$g$ plane curves of degree $d$ (in such case, $g=(d-1)(d-2)/2$) and
consider the sets $M_g^{Pl}(G):=M_g^{Pl}\cap M_g(G)$ and
$\widetilde{M_g^{Pl}(G)}:=\widetilde{M_g(G)}\cap M_g^{Pl}$. We
denote by $\rho(M_g^{Pl}(G))$ the elements of $\delta\in
M_g^{Pl}(G)$ such that $G$ acts on a plane model associated to
$\delta$ by $P\rho(G)P^{-1}$ for certain $P$. We have,
$M_g^{Pl}(G)=\cup_{[\rho]\in A}\rho(M_g^{Pl}(G))$ where
$A:=\{\rho|\rho:G\hookrightarrow PGL_3(K)\}/\sim$ and
$\rho_a\sim\rho_b$ if and only if $\rho_a(G)=P\rho_b(G)P^{-1}$ for
some $P\in PGL_3(K)$. A similar decomposition follows for
$\widetilde{M_g^{Pl}(G)}$.

Henn and Komiya-Kuribayashi proved that $\widetilde{M_3^{Pl}(G)}$ is
always equal to $\rho(\widetilde{M_3^{Pl}(G)})$ and in chapter 1, we prove that this does not happen for higher degrees $d\geq 5$ in
particular, we prove that the loci
$\widetilde{M_g^{Pl}(\Z/(d-1))}$ where $d\geq5$ is an odd integer and
also the locus $\widetilde{M_{10}^{Pl}(\Z/3)}$ are not ES-Irreducible hence they are not
irreducible.
In chapter 2, we give an algorithm (once $g$ is fixed) to decide the
list of values that $m$ assumes so that $M_g^{Pl}(\Z/m)$ is not empty, and how many
elements are in the set $A$. Moreover, we prove that if the group
$G$ has an element of order $(d-1)^2$, $d(d-1)$, $d(d-2)$ or
$d^2-3d+3$ then $M_g^{Pl}(G)$ is exactly a set with a unique element
therefore, is irreducible. Also, we study the situations when $G$ has
an element of order $\ell d$ or $\ell (d-1)$.
In chapter 3, we determine the exact locus where
$\rho(M_{6}^{Pl}(G))$ is not trivial, which covers the finite groups
that appear for non-singular plane curves of degree $5$ in $\mathbb{P}^2(K)$.
\end{abstract}

\maketitle


\chapter{On the locus of smooth plane curves with a fixed automorphism group}

\section{Abstract}

Let $M_g$ be the moduli space of smooth, genus $g$ curves over an
algebraically closed field $K$ of zero characteristic. Denote by
${M_g(G)}$ the subset of $M_g$ of curves $\delta$ such that $G$ (as
a finite non-trivial group) is isomorphic to a subgroup of
$Aut(\delta)$ and let $\widetilde{M_g(G)}$ be the subset of curves
$\delta$ such that $G\cong Aut(\delta)$ where
$Aut(\delta)$ is the full automorphism group of $\delta$. Now, for
an integer $d\geq 4$, let $M_g^{Pl}$ be the subset of $M_g$
representing smooth, genus $g$ plane curves of degree $d$ (in such
case, $g=(d-1)(d-2)/2$) and consider the sets
$M_g^{Pl}(G):=M_g^{Pl}\cap M_g(G)$ and
$\widetilde{M_g^{Pl}(G)}:=\widetilde{M_g(G)}\cap M_g^{Pl}$.
\par In this paper, we study some aspects of the irreducibility of
$\widetilde{M_g^{Pl}(G)}$ and the interrelations with the uniqueness
of plane non-singular equations (depending on a set of parameters)
for the elements inside $\widetilde{M_g^{Pl}(G)}$. Also, we
introduce the concept of being equation strongly irreducible
(ES-Irreducible) for which the locus $\widetilde{M_g^{Pl}(G)}$ is
represented by a unique plane equation associated with certain
parameters. Henn, in \cite{He}, and Komiya-Kuribayashi, in
\cite{KK}, observed that $\widetilde{M_3^{Pl}(G)}$ is ES-Irreducible
and in this note we prove that this phenomena does not occur for any
odd $d>4$. More precisely, let $\Z/m\Z$ be the cyclic group of order
$m$, we prove that $\widetilde{M_g^{Pl}(\Z/(d-1)\Z)}$ is not ES-Irreducible for any odd integer $d\geq5$ and the
number of the irreducible components of such loci is at least two.
Furthermore, we conclude the previous result when $d=6$ for the
locus $\widetilde{M_{10}^{Pl}(\Z/3\Z)}$.
\par Lastly, we prove an analogy of these statements when $K$ is any algebraically closed field of positive
characteristic $p$ such that $p>(d-1)(d-2)+1$.


\section{Introduction}

Let $K$ be an algebraically closed field of zero characteristic and
fix an integer $d\geq 4$. We consider, up to $K$-isomorphism, a
projective non-singular curve $\delta$ of genus $g=(d-1)(d-2)/2$ and
assume that $\delta$ has a non-singular plane model, say
$C\subset\mathbb{P}^2$ of degree $d$ whose defining equation is
given by a homogenous polynomial $F(X;Y;Z)=0$. Recall that, all such
plane models are $K$-equivalents that is, we can obtain all the
others by applying a change of variables $P\in PGL_3(K)$ on the
defining equation $F(X;Y;Z)=0$. Hence, where is no abuse of
notations, we may replace $\delta$ by $C$ also, the full
automorphism group of $\delta$ as a finite subgroup of $PGL_3(K)$
can be identified with $Aut(C)$, the set of all elements $\sigma\in
PGL_3(K)$ that retain the defining equation of $C$ (i.e.\,
$F(\sigma(X;Y;Z))=\lambda F(X;Y;Z)$ for some $\lambda\in K^*$).
\par A classical question is to determine,
up to $K$-equivalence, the different plane equations of non-singular
plane curves over $K$ of degree $d$ with automorphism group
isomorphic to $Aut(C)$. Another question, is to classify the groups
that appear as exact automorphism groups of algebraic curves of a
certain degree $d$. For $d=4$, Henn in \cite{He} and
Komiya-Kuribayashi \cite{KK}, answered the above natural questions
(see also Lorenzo's thesis \cite{Lo} \S\,2.1 and \S\,2.2,  in order
to fix some minor details). It appears for $d=4$ the following
phenomena: consider a finite group $G$ where $G\cong Aut(\delta)$
for certain $\delta$, then there is exactly a unique homogenous
polynomial of degree $4$ in $X,Y$ and $Z$, where the coefficients
depend on a set of parameters (with certain algebraic restrictions
on these parameters). Moreover, any specialization of the parameters
in $K$ corresponds to a unique $\delta\in \widetilde{M_3^{Pl}(G)}$
and any $\delta\in \widetilde{M_3^{Pl}(G)}$ has at least one
non-singular plane model with a certain specialization of the
parameters in $K$.
If this phenomena appears for some $g$, we say that the locus
$\widetilde{M_g^{Pl}(G)}$ is ES-Irreducible (see \S2 for the precise
definition) which is a weaker condition than the irreducibility of
this locus in the moduli space $M_g$. In particular, it follows by
Henn \cite{He} and Komiya-Kuribayashi \cite{KK}, that
$\widetilde{M_3^{Pl}(G)}$ is always ES-Irreducible.

The motivation of this work is that we did not expect
$\widetilde{M_g^{Pl}(G)}$ to be ES-Irreducible in general. In order
to construct counter examples where $\widetilde{M_g^{Pl}(G)}$ is not
ES-Irreducible: we need first, a group $G$ such that there exist at
least two injective representations $\rho_i:G\hookrightarrow
PGL_3(K)$ with $i=1,2$ which are not conjugate (i.e\,\,there is no
transformation $P\in PGL_3(K)$ with $P\rho_1(G)P^{-1}=\rho_2(G)$,
more details are included in \S2) and for the zoo of groups that
could appear for non-singular plane curves \cite{Harui}, we consider
$G$ a cyclic group of order $m$. Secondly, one needs to prove the
existence of non-singular plane curves with automorphism group
conjugate to $\rho_i(G)$ for each $i=1,2$.
\par The main result of the paper is that for any odd degree $d(\geq 5)$, the locus
$\widetilde{M_g^{Pl}(\Z/(d-1)\Z)}$ is not ES-irreducible and it has
at least two irreducible components moreover, for degree $d=5$ with
$m\neq 4$, we observed that the locus
$\widetilde{M_{6}^{Pl}(\Z/m\Z)}$ is ES-Irreducible whenever it is
not empty. Furthermore, by \cite{BaBaaut}, we know that the only
group $G$ for which $\widetilde{M_{6}^{Pl}(G)}$ is not
ES-Irreducible is that when $G\cong\Z/4\Z$, and by \cite{BaBacyc} we
suspect that the only case for which $\widetilde{M_g(\Z/m\Z)}$ is
not ES-Irreducible is that when $m$ divides $d$ or $d-1$ (this is
true until degree $9$ and the tables for possible cyclic groups are
listed in \cite{BaBacyc}). In section \S\,5, we give an example when
$d$ is even, concretely that $\widetilde{M_{10}^{Pl}(\Z/3\Z)}$ is
not ES-irreducible. At the end (\S\,6) of this paper we prove that
the above examples of non-irreducible loci are also valid when $K$
is an algebraically closed of positive characteristic $p$, provided
that the characteristic $p$ is big enough once we fixed the degree
$d$.



To conclude the introduction, we relate the above locus with the
classical one of the moduli space of curves and the problem of
irreducibility versus the problem of being ES-Irreducible.

%
%

We denote by $\rho(M_g^{Pl}(G))$ the set of all elements $\delta\in
M_g^{Pl}(G)$ such that $G$ acts on a plane model associated to
$\delta$ by $P\rho(G)P^{-1}$ for certain $P$ (because the linear
systems $g^2_d$ are unique, the same property follows for any plane
model varying the conjugation matrix $P$). Therefore,
$M_g^{Pl}(G)=\cup_{[\rho]\in A}\rho(M_g^{Pl}(G))$ where
$$A:=\{\rho\,\,|\,\,\rho:G\hookrightarrow PGL_3(K)\}/\sim$$ such
that $\rho_a\sim\rho_b$ if and only if $\rho_a(G)=P\rho_b(G)P^{-1}$
for some $P\in PGL_3(K)$. A similar decomposition follows for
$\widetilde{M_g^{Pl}(G)}$.

Now, for being not ES-Irreducible, we are interested in the
situations for which the locus $\widetilde{\rho(M_g^{Pl}(G))}$ is
non-empty for at least two elements of $A$. Also, a much deeper
question is when does it happen that the locus
$\widetilde{\rho(M_g^{Pl}(G))}$ or $\rho(M_g^{Pl}(G))$ is
irreducible? The locus ${\rho(M_g^{Pl}(G))}$ is related, as a subset
of the locus of genus $g$ curves which have a Galois subcover of
Galois group $G$, with some prescribed ramification from the group
$\rho(G)$ (see Remark \ref{rem2}, for some examples) and by the
works of Cornalba \cite{Cor} and Catanese \cite{Ca}, we know that it
is irreducible when $G$ is cyclic. Therefore,
$\rho(M_g^{Pl}(\Z/m\Z)$ (resp. $\widetilde{\rho(M_g^{Pl}(\Z/m\Z))}$)
are subsets in an irreducible locus and moreover, there is a
topological invariant which computes the irreducible components of
$M_g(G)$ (see \cite[\S2]{CLP}). In particular, we can decide when
$M_g(G)$ is irreducible or not. Unfortunately, we do not have any
insight for the irreducibility of the locus $\rho(M_g^{Pl}(G))$
except possibly the computations that is given in \S2 for
$\rho(M_6(\Z/8\Z))$ and the result: $M_g^{Pl}(G)$ is irreducible
when $G$ has an element of order $(d-1)^2$, $d(d-1)$, $d(d-2)$ or
$d^2-3d+3$ since this locus has only one element by \cite{BaBacyc}.

\subsection*{Acknowledgments} It is our pleasure to express our sincere gratitude to Xavier
Xarles for his suggestions and comments. We also thank Massimo
Giulietti and Elisa Lorenzo for noticing us about some bibliography
on automorphism of curves.

\section{On the locus $M_g^{Pl}(G)$ and its relation with a set
of equations via certain parameters}

Consider a projective non-singular curve $\delta$ of genus
$g:=\frac{(d-1)(d-2)}{2}\geq 2$ over $K$ with $G$, a finite
non-trivial group, inside $Aut(\delta)$. We always assume that
$\delta$ admits a non-singular plane equation, and we consider
$\delta$ up to $K$-isomorphism, as a point in $M_g^{Pl}(G)$. A set
of non-singular plane equations can be associated to $\delta$ in
$\mathbb{P}^2(K)$, let $C:\,F(X;Y;Z)=0$ be a particular equation
where $F(X;Y;Z)$ is a homogenous polynomial of degree $d$ and $G$
acts as $\rho(G)$ (retaining invariant the equation $C$) for some
$\rho:G\hookrightarrow PGL_3(K)$ (a fixed representation of $G$
inside $PGL_3(K)$). If $\tilde{C}:\,\tilde{F}(X;Y;Z)=0$ is any other
plane equation of $\delta$ then it is projectively equivalent to $C$
through a change of variables $P\in PGL_3(K)$ applied to $\{X,Y,Z\}$
and the automorphism group $G$. In other words, $F(P(X;Y;Z))=\lambda
F_2(X;Y;Z)$ and $P\rho(G)P^{-1}$ is a subgroup of $Aut(\tilde{C})$.
Now, fix $\rho:G\hookrightarrow PGL_3(K)$ and let us denote by
$\rho(M_g^{Pl}(G))$ the subset of $M_g^{Pl}(G)$ such that each
$\delta\in \rho(M_g^{Pl}(G))$ has a non-singular plane equation with
an effective action by $P^{-1}\rho(G)P$ on its automorphism group
for certain $P\in PGL_3(K)$ or equivalently, $\delta$ has a plane
non-singular equation with $\rho(G)$ preserves such equation in
particular, $\rho(G)$ is a subgroup of its full automorphism group.

\begin{lem} Consider the set $M_g^{Pl}(G)$ for non-trivial $G$ and the
set $A=\{\rho:G\hookrightarrow PGL_3(K)\}/\sim$ where
$\rho_1\sim\rho_2$ if and only if $\exists P\in PGL_3(K)$ such that
$\rho_1(G)=P\rho_2(G)P^{-1}$. Write $A=\{[\rho_1],\ldots,[\rho_n]\}$
then $M_g(G)$ is the disjoint union of $\rho_i(M_g(G))$ where
$i=1,\ldots,n$ and $\rho_i$ is a representative of the class
$[\rho_i]$.
\end{lem}

Fix $[\rho]\in A$ then for $\delta \in \rho(M_g^{Pl}(G))$, we can
associate infinitely many non-singular plane curves which are
$K$-isomorphic through a change of variables $P\in PGL_3(K)$, but we
can consider only the family of such curves with $G$ acts exactly as
$\rho(G)$ for some fixed $\rho$ in $[\rho]$. Under this restriction,
$\delta$ can be associated with a non-empty family of non-singular
plane curves such that any two models are isomorphic, through a
projective transformation $P$ satisfying $P^{-1}\rho(G)P=\rho(G)$.
Recall that, it is a necessary condition for a projective plane
curve of degree $d$ to be non-singular that the defining equation
has degree $\geq d-1$ in each variable, and by a diagonal change of
variables $P$, we can assume that the monomials with the maximal
exponent have coefficients equal to $1$. Consequently, we reduce the
case to the set of $K$-isomorphic non-singular plane curves
$F(X;Y;Z)=0$ associated to $\delta$ with $\rho(G)$ fixes the
equation and each term of the core of $F(X;Y;Z)$ (i.e\, the sum of
all monomials with the highest exponent) is monic.

\begin{lem} Let $G$ be a non-trivial finite group and consider $\rho:G\hookrightarrow
PGL_3(K)$ such that $\rho(M_g^{Pl}(G))$ is non-empty. There exists a
unique homogenous polynomial $F(X;Y;Z)=0$ of degree $d$ of the above
prescribed form and is endowed with certain parameters as the
coefficients of the lower order terms (some restrictions should be
imposed so that $F(X;Y;Z)$ is non-singular). Moreover, every element
of $\rho(M_g^{Pl}(G))$ corresponds to a set of equations obtained by
assigning certain values to the parameters in the given equation
$F(X;Y;Z)=0$ and vice versa, that is, every concrete values of the
parameters corresponds to a point of $\rho(M_g^{Pl}(G))$.
\end{lem}
\begin{proof} Let $\sigma\in G$ be an automorphism of maximal order $m>1$ and choose an element $\rho$ in $[\rho]\in A$ such that, $\rho(\sigma)$ is diagonal of the form $diag(1,\xi_m^a,\xi_m^b)$ where $\xi_m$ a primitive $m$-th root of
unity in $K$. Following the same technique in \cite{Dol} or
\cite{BaBacyc} (for a general discussion), we can associate to the
set $\rho(M_g^{Pl}(<\sigma>)$ a non-singular plane equation $F_{m,
(a,b)}(X;Y;Z)$ with a certain set of parameters (which may have some
restrictions in order to ensure the non-singularity). Now imposing
that $\rho(G)$ should fix the equation $F_{m, (a,b)}$, we obtain
some specific algebraic relations between the parameters that should
satisfy in general the final equation $F_{m, (a,b)}$ which we want
to reach. The remaining part of the lemma is straightforward.
\end{proof}

It is difficult to determine the groups $G$ and $\rho$ such that
$\rho(M_g^{Pl}(G))$ is non-empty for some fixed $g$. Henn \cite{He}
obtained this determination for $g=3$, Badr-Bars \cite{BaBaaut} for
$g=6$ and for a general implementation of any degree, we refer to
\cite{BaBacyc} in which we formulate an algorithm to determine the
$\rho$'s when $G$ is cyclic.

\begin{defn} Write $M_g^{Pl}(G)$ as $\cup_{[\rho]\in A}\rho(M_g^{Pl}(G))$, we define the number
of the equation components of $M_g^{Pl}(G)$ to be the number of
elements $[\rho]\in A$ such that $\rho(M_g^{Pl}(G))$ is not empty.
We say that $M_g^{Pl}(G)$ is equation irreducible if
$M_g^{Pl}(G)=\rho(M_g^{Pl}(G))$ for a certain $\rho:G\hookrightarrow
PGL_3(K)$.
\end{defn}

Of course, if $M_g^{Pl}(G)$ is not equation irreducible then it is
not irreducible and the number of the irreducible equation
components of $M_g^{Pl}(G)$ is a lower bound for the number of
irreducible components.

Here, we state some questions that appear naturally concerning the
locus $\rho(M_g^{Pl}(G))$:
\begin{preg} Is $\rho(M_g^{Pl}(G))$ a subset of the locus of
smooth projective, genus $g$ curves with a Galois subcover of group
isomorphic to $G$ with a prescribed ramification?
\end{preg}
We believe that the answer to this question for $K=\mathbb{C}$ (i.e.
Riemann surfaces) should be always true from the work of Breuer
\cite{Bre}.

\begin{preg} Is $\rho(M_g^{Pl}(G))$ an irreducible set when $G$ is
a cyclic group?
\end{preg}
 It is to be noted that when $K=\mathbb{C}$, Cornalba \cite{Cor}, with $G$ cyclic of prime order,
 and Catanese \cite{Ca}, for general order, obtained that the locus
 of smooth projective curves of genus $g$ with a cyclic Galois
 subcover of group isomorphic to $G$ with a prescribed ramification is
 irreducible.

Concerning the irreducibility question, we prove in \cite{BaBacyc}
that if $G$ has an element of large order $(d-1)^2$, $d(d-1)$,
$d(d-2)$ or $d^2-3d+3$ then $\rho(M_g^{Pl}(G))$ has at most one
element therefore, is irreducible.

At the end of this section, we give positive answers to the above
questions for certain $\rho$ for $\rho(M_6^P(\Z/8))$. In this
example, we obtain that $\rho(M_6^P(\Z/8))=M_6^P(\Z/8)$ (See also
Remark \ref{rem2} for the explicit Galois subcover and the
ramification data for the locus $\rho{(M_6(\Z/4\Z))}$).

\begin{defn} Given $M_g^{Pl}(G)$, we define $\widetilde{M_g^{Pl}(G)}$ to be the
subset of the elements inside $M_g^{Pl}(G)$ whose automorphism group
is isomorphic to $G$. Similarly for $\widetilde{\rho(M_g^{Pl}(G))}$
and we have $\widetilde{M_g^{Pl}(G)}=\cup_{[\rho]\in
A}\widetilde{\rho(M_g^{Pl}(G))}.$ On the other hand, the number of
strongly equation irreducible components of
$\widetilde{M_g^{Pl}(G)}$ is defined to be the number of the
elements $[\rho]\in A$ such that $\widetilde{\rho(M_g^{Pl}(G))}$ is
not empty.

 We say that
$\widetilde{M_g^{Pl}(G)}$ is equation strongly irreducible (or
simply, ES-irreducible) if it is not empty and
$\widetilde{M_g^{Pl}(G)}=\widetilde{\rho(M_g^{Pl}(G))}$ for some
$\rho:G\hookrightarrow PGL_3(K)$
\end{defn}
It is a direct consequence to ask the same questions above for the
locus $\widetilde{\rho(M_g^{Pl}(G))}$ instead of
$\rho(M_g^{P\l}(G))$. Furthermore, in this language we can formulate
the main result in \cite{He} as follows
\begin{thm}[Henn, Komiya-Kuribayashi] If $G$ is a non-trivial group that appears as the full
automorphism group of a non-singular plane  curve of degree $4$,
then $\widetilde{M_3^P(G)}$ is ES-Irreducible.
\end{thm}
\begin{rem} Henn in \cite{He}, observed that $M_3^{Pl}(\Z/3)$ already has two
irreducible equation components, and therefore has at least two
irreducible components, but one of such components has a bigger
automorphism group namely, $S_3$ the symmetry group of of order 3.
\end{rem}

\subsection{The loci $M_6^{Pl}(\Z/8)$ and
$\widetilde{M_6^{Pl}(\Z/8\Z)}$} \mbox{}\\

Consider in $M_6$ an element $\delta$ which has a smooth
non-singular plane model with an effective action of the cyclic
group of order 8 in particular, $\delta\in M_6^P(\Z/8\Z)$. Following
\cite{BaBacyc}, \cite{Dol} or the table \S4 in this note, we can
associate to $\delta$ an equation of the form $X^5+Y^4Z+XZ^4+\beta
X^3Z^2=0$ for certain(/s) $\beta$ which depend/(s) on $\delta$, and
there is only an element in $A$ in the case of degree 5 with cyclic
group of order 8. In particular, we fix an equation for $\delta$
such that $\rho(\Z/8\Z)=<diag(1,\xi_8,\xi_8^4)>$ where $\xi_8$ is a
$8$-th primitive root of unity in $K$, and because $\delta$ is
non-singular then we should have $\beta\neq \pm2$. In this locus we
have $\rho(M_6^P(\Z/8\Z))=M_6^P(\Z/8\Z).$

\par Now, let us compute all the equations of the form
$X^5+Y^4Z+XZ^4+\beta X^3Z^2=0$ that can be associated to the fixed
curve $\delta$. This corresponds to equations obtained by a change
of variables through a transformation $P\in PGL_3(K)$ such that
$P<(diag(1,\xi_8,\xi_8^4)>P^{-1}=<diag(1,\xi_8,\xi_8^4)>$ and the
new eqnution has a similar form $X^5+Y^4Z+XZ^4+\beta'X^3Z^2=0$.
\par Without any loss of generality, we can suppose that
$Pdiag(1,\xi_8,\xi_8^4)P^{-1}=diag(1,\xi_8,\xi_8^4)$ hence in order
to have the same eigenvalues which are pairwise distinct, we may
assume that $P$ is a diagonal matrix, say
$P=diag(1,\lambda_2,\lambda_3)$. Therefore, we get an equation of
the form: $X^5+\lambda_2^4\lambda_3 Y^4Z+\lambda_3^4
XZ^4+\beta\lambda_3^2X^3Z^2=0$. From which we must have
$\lambda_2^4\lambda_3=\lambda_3^4=1$, thus $\lambda_3^2$ is 1 or -1.
Hence, we obtain a bijection map
$$\varphi:M_6^P(\Z/8\Z)\rightarrow \mathbb{A}^1(K)\setminus\{-2,2\}/\sim$$
$$\alpha\mapsto [\beta]=\{\beta,-\beta\}$$
where $a\sim b\Leftrightarrow b=a\ or\ a=-b$. Moreover, by the work
that we did in \cite{BaBaaut}, we know that $X^5+Y^4Z+XZ^4+\beta
X^3Z^2=0$ has a bigger automorphism group than $\Z/8\Z$ if and only
if $\beta=0$, therefore, we have a bijection map
$$\tilde{\varphi}:\widetilde{M_6^P(\Z/8\Z)}\rightarrow \mathbb{A}^1(K)\setminus\{-2,0,2\}/\sim$$
$$\alpha\mapsto [\beta]=\{\beta,-\beta\}$$
and observe that $0\in \mathbb{A}^1(K)$ is the only point which had
no identification by the relation rule $\sim$. The above sets, when
$K$ is the complex field, are irreducible.

Moreover, if we consider the Galois cyclic cover of degree 8 given
by the action of the automorphism of order 8 on $X^5+Y^4Z+XZ^4+\beta
X^3Z^2=0$, we obtain that it ramifies at the points
$(0:1:0),(0:0:1)$ with ramification index 8 as well as the four
points $(1:0:a)$ where $1+a^4+\beta a^2=0$ with ramification index 2
if $\beta\neq \pm 2$. That is, $M_6^P(\Z/8\Z)$ is inside the locus
of curves in $M_6$ which have a cyclic Galois cover of degree 8 of a
genus zero curve and which ramifies at 6 points, 2 points with
ramification index 8 and the other 4 points are with ramification
index 4.

\section{Preliminaries on automorphism on plane curves}
We recall that, given a non-singular plane curve $C$ of degree
$d\geq 4$ over an algebraic closed field $K$ of genus $g\geq2$ then
$Aut(C)$ is a finite subgroup of $PGL_3(K)$ and it satisfies one of
the following situations (for more details, see Mitchell
\cite{Mit}):
\begin{enumerate}
\item fixes a point $P$ and a line $L$ with $P\notin L$ in
$\mathbb{P}^2(K)$,
\item fixes a triangle, i.e. exists 3 points $S:=\{P_1,P_2,P_3\}$ of
$\mathbb{P}^2(K)$, such that is fixed as a set,
\item $Aut(C)$ is conjugate of a representation inside $PGL_3(K)$ of
one of the finite primitive group namely, the Klein group
$PSL(2,7)$, the icosahedral group $A_5$, the alternating group
$A_6$, the Hessian groups $Hess_{216}$, $Hess_{72}$ or $Hess_{36}$.
\end{enumerate}

For a non-zero monomial $cX^iY^jZ^k$ we define its exponent as
$max\{i,j,k\}$. For a homogeneous polynomial $F$, the core of $F$ is
defined as the sum of all terms of $F$ with the greatest exponent.
Let $C_0$ be a smooth plane curve, a pair $(C,G)$ with $G\leq
Aut(C)$ is said to be a descendant of $C_0$ if $C$ is defined by a
homogeneous polynomial whose core is a defining polynomial of $C_0$
and $G$ acts on $C_0$ under a suitable coordinate system.

\begin{thm}[Harui] (see \cite{Harui} $\S 2$) \label{teo1}
Let $G$ be a subgroup of $Aut(C)$. Then $G$ satisfies one of the
following statements:
\begin{enumerate}
  \item $G$ fixes a point on $C$ and then it is cyclic.
  \item $G$ fixes a point not lying on $C$ and it satisfies a short exact sequence of the form
  $$1\rightarrow N\rightarrow G\rightarrow G'\rightarrow 1,$$
with $N$ a cyclic group of order dividing $d$ and $G''$ is isomorph
to a cyclic group $C_m$ of order $m$, a Diedral group $D_{2m}$,
$A_4$ , $A_5$ or $S_4$, where $m$ is an integer $\leq d-1$.
Moreover, if $G'\cong D_{2m}$, then $m|(d-2)$ or $N$ is trivial.
\item $G$ is conjugate (by certain $P\in PGL_3(K)$) to a subgroup of $Aut(F_d)$ where $F_d$ is the Fermat curve $X^d+Y^d+Z^d$
and $(G,C)$ is a descendant of $F_d$. In particular, $|G|\,|\,6d^2$.
\item $G$ is conjugate to a subgroup of $Aut(K_d)$
where $K_d$ is the Klein curve curve $X^{d-1}Y+Y^{d-1}Z+Z^{d-1}X$
and $(G,C)$ is a descendant of $K_d$. Therefore
$|G|\,|\,3(d^2-3d+3)$.
\item $G$ is conjugate to a finite primitive subgroup$PGL_3(K)$ namely, the Klein group $PSL(2,7)$, the icosahedral group $A_5$, the alternating group $A_6$, or the Hessian groups $Hess_{216}$, $Hess_{72}$, $Hess_{36}$.

\end{enumerate}
\end{thm}

\textbf{The Hessian group:} A representation of the Hessian group of
order $216$ inside $PGL_3(K)$ is given by $Hess_{216}=<S,T,U,V>$
with,
$$S=\left(
\begin{array}{ccc}
1&0&0\\
0&\omega&0\\
0&0&\omega^2\\
\end{array}
\right);\ \ U=\left(
\begin{array}{ccc}
1&0&0\\
0&1&0\\
0&0&\omega\\
\end{array}
\right);\ V=\frac{1}{\omega-\omega^2}\left(
\begin{array}{ccc}
1&1&1\\
1&\omega&\omega^2\\
1&\omega^2&\omega\\
\end{array}
\right);\ \ T=\left(
\begin{array}{ccc}
0&1&0\\
0&0&1\\
1&0&0\\
\end{array}
\right);$$ always $\omega$ means a primitive 3rd root of unity.
Also, we consider the primitive subgroups $Hess_{36}=<S,T,V>$ and
$Hess_{72}=<S,T,V,UVU^{-1}>$. It should be noted that,
representations of $Hess_{216}$ inside $PGL_3(K)$ forms a unique set
up to conjugation (see Mitchell \cite{Mit} page $217$), and because
$Hess_{72}$ and $Hess_{36}$ are defined geometrically as subgroups
of $Hess_{216}$ in $PGL_3(K)$ see \cite[\S 23,p.25]{Grove} their
representations inside $PGL_3(K)$ are also unique up to conjugation.
\begin{rem} \label{rem1} In particular, for the Hessian groups $Hess_{216}$, $Hess_{72}$
and $Hess_{36}$, the locus $\widetilde{M_g^{Pl}(Hess_*)}$ is
ES-Irreducible as long as is not empty (where $*\in\{36,72,216\}$)
because the set $A$ is trivial (with the notation of \S2).
\end{rem}
Now, we may ask the following.
\begin{preg} Consider $G$, a non-trivial group, where the set $A$ is
giving by one element (see the notion of $A$ in previous section).
 Is it true that the
topological invariant in \cite[\S2]{CLP} is trivial for $M_g(G)$ in
order to be irreducible? Is it true that $M_g^{Pl}(G)$ are
irreducible?
\end{preg}
From the above and the interplay with the notion of being
ES-irreducible or not of the locus $\widetilde{M_g^{Pl}(G)}$, one
also could ask the following question in group theory,
\begin{preg} Fix a finite non-cyclic group $G$ of $PGL_3(K)$. How many elements are inside the set
$$A=\{\rho| \rho:G\hookrightarrow PGL_3(K)\}/ \sim$$ where
$\rho_1\sim\rho_2$ if $\rho_1(G)=P^{-1}\rho_2(G)P$ for some $P\in
PGL_3(K)$?
\end{preg}

\section{Cyclic groups in smooth plane curves of degree 5 and $\widetilde{M_6^{Pl}(\Z/m\Z)}$.}
Note that, given a smooth plane curve $C:\,F_C(X;Y;Z)=0$ of degree
$d$ such that $Aut(C)$ is non-trivial, we study it up to
$K$-isomorphism, that is, two of them are $K$-isomorphic if one
transforms to the other by a change of variables $P\in PGL_3(K)$.

By a change of variables, we can suppose that the cyclic group of
order $m$ acting on a smooth plane curve of degree $5$ is given in
$PGL_3(K)$ by a diagonal matrix $diag(1;\xi_m^a;\xi_m^b)$ ( where
$\xi_m$ an $m$-th primitive root of unity and $a\leq b$ are positive
integers) and we call this element by type $m,(a,b)$. Following the
same proof of \cite[\S 6.5]{Dol} (or see \cite{BaBacyc} for general
treatment with an algorithm of computation for any $d$), we have the
following table with the corresponding equations (depending of
certain parameters) that have an effective cyclic group of certain
orders. Such groups are the possible situations for which
$\rho(M_6^{Pl}(\Z/m\Z))$ is non-empty:
\begin{center}
\begin{tabular}{|c|c|}
  \hline
  Type: $m, (a,\,b)$ & $F(X;Y;Z)$ \\\hline\hline
  $20,(4,5)$& $X^5+Y^5+ XZ^4$ \\\hline
  $16,(1,12)$& $X^5+Y^4Z+ XZ^4$\\\hline
  $15,(1,11)$& $X^5+Y^4Z+ YZ^4$    \\\hline
$13,(1,10)$& $X^4Y+Y^4Z+ Z^4X$    \\\hline
 $10,(2,5)$& $X^5+Y^5+ XZ^4+\beta_{2,0}X^3Z^2$ \\\hline
  $8,(1,4)$& $X^5+Y^4Z+ XZ^4+\beta_{2,0}X^3Z^2$ \\\hline
 $5,(1,2)$& $X^5+Y^5+Z^5+\beta_{3,1}X^2YZ^2+\beta_{4,3}XY^3Z$    \\\hline
  $5,(0,1)$& $Z^5+L_{5,Z}$    \\\hline

  $4,(1,3)$& $X^5+X\big(Z^4+ Y^4+\beta_{4,2}Y^2Z^2\big)+\beta_{2,1}X^3YZ$\\\hline
 $4,(1,2)$& $X^5+X\big(Z^4+ Y^4\big)+\beta_{2,0}X^3Z^2+\beta_{3,2}X^2Y^2Z+\beta_{5,2}Y^2Z^3$\\\hline
  $4,(0,1)$& $Z^4L_{1,Z}+L_{5,Z}$    \\\hline
  $3,(1,2)$& $X^5+Y^4Z+ YZ^4+\beta_{2,1}X^3YZ+X^2\big(\beta_{3,0}Z^3+\beta_{3,3}Y^3\big)+\beta_{4,2}XY^2Z^2$ \\\hline
 $2,(0,1)$& $Z^4L_{1,Z}+Z^2L_{3,Z}+L_{5,Z}$    \\\hline
  \end{tabular}
\end{center}
where $L_{i,U}$ means a homogeneous polynomial of degree $i$ which
does not contain the variable $U$ and $\beta_{i,j},\in K$, and
remains in the table above to introduce restriction in the
parameters $\beta_{i,j}$ so that $F(X,Y,Z)=0$ is non-singular (which
will be omitted).

By the above table we find that $M_6^{Pl}(\Z/m\Z)$ is not empty,
only for the values $m$ that are included in the previous list
moreover, for $m\neq 4,5$, we have
$M_6^{Pl}(\Z/m\Z)=\rho(M_6^{Pl}(\Z/m\Z))$  where $\rho$ is obtained
such that $\rho(\Z/m\Z)=<diag(1,\xi_m^a,\xi_m^b)>$. Thus, the
corresponding loci $\widetilde{M_6^{Pl}(\Z/m\Z)}$ where $m\neq4,5$
are ES-Irreducible if they are non-empty.

Now, we consider the remaining cases $\widetilde{M_6^{Pl}(\Z/m\Z)}$
with $m=4$ or 5.

Clearly, the equations of type $5, (1,2)$ always have a bigger
automorphism group by permuting $X$ and $Z$. Therefore, there is at
most one equation that define curves of degree $5$ with the full
automorphism group is isomorphic to $C_5$ (Observe that the number
of
  conjugacy classes of representations of $C_5$ inside $PGL_3(K)$ is
  three). In particular, $\widetilde{M_6^P(\Z/5\Z)}=\rho(M_6^P(\Z/5\Z))$
where $\rho(\Z/5\Z)=<diag(1,1,\xi_5)>$ and hence
$\widetilde{M_6^P(\Z/5\Z)}$ is ES-Irreducible if it is non-empty.

  Now, we work for the cyclic groups of order 4. The
  type $4, (1,3)$ has always a bigger automorphism group by permuting
  $Y$ and $Z$ conveniently (after a change of variable on $Y$ to make $\alpha=1$).
  In fact, this type is not irreducible (being of the form $X.G$) hence is singular and will be out of our scope in this note.

  Therefore, we have $M_6^{Pl}(\Z/4\Z)=\rho_1(M_6^{Pl}(\Z/4\Z))\cup
  \rho_2(M_6^{Pl}(\Z/4\Z))$ where $\rho_1$ corresponds to type
  $4, (0,1)$ and $\rho_2$ to type $4, (1,2)$.
%
\subsection{On type $4,(0,1)$} Consider the non-singular plane curve defined
by the equation $\tilde{C}:\,\,X^5+Y^5+Z^4X+\beta X^3Y^2$ where
$\beta\neq0.$ This curve admits a cyclic element of order $4$
namely, $\sigma:=[X;Y;\xi_4Z]$ with axis $Z=0$ and center $(0;0;1)$
(we say that is an homology when cyclic elements fix such geometric
construction). It follows by Mitchell \cite{Mit} $\S 5$, that
$Aut(\tilde{C})$ should fix a point, a line or a triangle. If
$Aut(\tilde{C})$ fixes a triangle and neither a line nor a point is
leaved invariant then, by Harui \cite{Harui} $\S 5$, $\tilde{C}$ is
a descendant of the Fermat curve $F_5$ or the Klein curve $K_5$ but
this is impossible because $4\nmid|Aut(F_5)|(=150)$ and
$4\nmid|Aut(K_5)|(=39)$. Therefore, $Aut(\tilde{C})$ should fix a
line and a point off that line. It follows by \cite{Harui} [Lemma
$3.7$] that the point $(0;0;1)$ is an inner Galois point of
$\tilde{C}$ (and it is unique by Yoshihara \cite{Yoshihara} $\S 2$
[Theorem $4$]). 
Therefore, the unique inner Galois point is fixed by the full
$Aut(\tilde{C})$ in particular, $Aut(\tilde{C})$ is cyclic
(\cite[Lemma 11.44]{Book}) and automorphisms of $\tilde{C}$ are of
the form
$[X+\alpha_2Y;\beta_1X+\beta_2Y;\gamma_1X+\gamma_2Y+\gamma_3Z]$ From
the coefficient of $YZ^4$ we must have $\alpha_2=0$ consequently,
$\gamma_1=\gamma_2=0$ (Coefficients of $Z^3X^2$ and $XZ^3Y$) then
$\beta_1=0$ (Coefficient of $XY^4$). That is, $Aut(\tilde{C})$ is
cyclic and is generated by a diagonal element $[X;vY;tZ]$ which in
turns implies that $v^5=v^2=t^4=1$ hence $v=1$ and $t$ is a $4$-th
root of unity and the result follows.

Therefore, with the above argument we coclude the following result.
\begin{prop} The locus set $\widetilde{\rho_1(M_6^{Pl}(\Z/4\Z)}$ is
non-empty.
\end{prop}

\subsection{On type $4, (1,2)$} Consider the non-singular plane curve
defined by the equation $\tilde{\tilde{C}}:\,\,X^5+X(Z^4+\alpha
Y^4)+\beta Y^2Z^3$ where $\alpha\beta\neq0$. This curve admits a
cyclic group generated by the automorphism
$\tau:=[X;\xi_4Y;\xi_4^2Z]$. For the same reason as above (i.e
$4\nmid|Aut(K_5)|, |Aut(F_5)|$), $\tilde{\tilde{C}}$ is not a
descendant of the Fermat curve $F_5$ or the Klein curve $K_5$.
Moreover, $Aut(\tilde{\tilde{C}})$ is not conjugate to an
icosahedral group $A_5$ (no elements of order 4), the Klein group
$PSL(2,7)$, the Hessian group $Hess_{216}$ or the alternating group
$A_6$ (since by \cite{Harui} [Theorem $2.3$],
$|Aut(\tilde{\tilde{C}})|\leq150$). Moreover, we claim to prove the
following.
\begin{claim}
$Aut(\tilde{\tilde{C}})$ is not conjugate to any of the Hessian
subgroups namely, $Hess_{36}$ or $Hess_{72}$.
\end{claim}
Because both groups contains reflections (but no four groups) then
all reflections in the group will be conjugate (\cite{Mit} Theorem
$11$). Therefore, it suffices to consider the case
$P\,\tau^2\,P^{-1}=\lambda [Z;Y;X]$ this in turns gives $P$ of the
forms $\left(
  \begin{array}{ccc}
    1 & \alpha_2 & \alpha_3 \\
    \beta_1 & 0 & \beta_3 \\
    1 & -\alpha_2 & \alpha_3 \\
  \end{array}
\right)$ or $\left(
              \begin{array}{ccc}
                1 & \alpha_2 & \alpha_3 \\
                0 & \beta_2 & 0 \\
                -1 & \alpha_2 & -\alpha_3 \\
              \end{array}
            \right)$. For both situations, by comparing the coefficients of $XY^4$ and $Y^4Z$ (resp. $L_{5,Y}$), we get $\alpha_3=1$ (since $[Z;Y;X]\in Aut(P\tilde{\tilde{C}})$) in particular, the second transformation is not an option. Furthermore, from the coefficients of $X^2Y^3$ and $Y^3Z^2$ we should have $\beta_1=\beta_3$ a contradiction and we conclude the result.

Consequently, $Aut(\tilde{\tilde{C}})$ should fix a line and a point
off that line. Now, because $\tau\in Aut(\tilde{\tilde{C}})$ is of
the form $diag(1;a;b)$ such that $1,a,b$ (resp. $1,a^3,b^3$) are
pairwise distinct then automorphisms of $\tilde{\tilde{C}}$ are of
the forms $\tau_1:=[X;vY+wZ;sY+tZ],\,\,\tau_2:=[vX+wZ;Y;sX+tZ]$ or
$\tau_3:=[vX+wY;sX+tY;Z]$ (because the fixed point only could be
$(1:0:0)$ or $(0:1:0)$ or $(0:0:1)$ and the line that is leaved
invariant is one of the reference lines).
\par Now, if $\tau_1\in Aut(\tilde{\tilde{C}})$ then $s=0=w$ (Coefficient of $Y^5$ and $Z^5$) and the same conclusion if $\tau_2$ (resp. $\tau_3$) $\in Aut(\tilde{\tilde{C}})$ from the coefficients of $X^3Y^2$ and $Y^4Z$ (resp. $Z^3X$ and $YZ^4$). Hence, automorphisms of $\tilde{\tilde{C}}$ are all diagonal and moreover,
if $[X;vY;sZ]\in Aut(\tilde{\tilde{C}})$ then we must have
$v^4=s^4=v^2s^3=1$ that is $Aut(\tilde{\tilde{C}})$ is cyclic of
order 4.

Consequently, the following result follows.
\begin{prop} The locus set $\widetilde{\rho_2(M_6^{Pl}(\Z/4\Z)}$ is
non-empty.
\end{prop}

\begin{cor} The locus set $\widetilde{M_6^{Pl}(\Z/m\Z)}$ is
ES-Irreducible if and only if $m\neq4$. If $m=4$ then
$\widetilde{M_6^{Pl}(\Z/m\Z)}$ has exactly two irreducible equation
components and hence the number of its irreducible components is at
least two.
\end{cor}

\begin{rem}\label{rem2} Observe that the Galois cover of degree 4 corresponding to $\rho_1(M_6^{Pl}(\Z/4\Z))$:
$${C}_1:=Z^4L_{1,Z}+L_{5,Z}=0\rightarrow {C}_1/<[X;Y;\xi_4Z]>$$
is ramified exactly at six points with ramification index 4. Indeed,
the fixed points of $\sigma^i$ for $i=1,2,3,4$ are all the same
where $\sigma=diag(1,1,\xi_4)$ therefore, we only need to consider
the ramification points of $\sigma$ in particular, the ramification
index is 4. Now, by the Hurwitz formula we get $10=4(2 g_0-2)+3k$
where $g_0$ is the genus of ${C}_1/<[X,Y,\xi_4 Z]>$ hence we are
forced to $g_0=0$ and $k=6$. On the other hand, the Galois cover of
degree 4 corresponding to $\rho_2(M_6^{Pl}(\Z/4\Z))$, gives that
$${{C}}_2:=X^5+X(Z^4+
Y^4)+\beta_{2,0}X^3Z^2+\beta_{3,2}X^2Y^2Z+\beta_{5,2} Y^2
Z^3=0\rightarrow{C}_2/<[X;\xi_4 Y;\xi_4^2 Z]$$ is ramified at the
points $(0:1:0),(0:0:1)$ with ramification index 4 and at the 4
points namely, $(1:0:a)$ where $1+a^4+\beta_{2,0}a^2=0$ with
ramification index 2 provided that $\beta_{2,0}\neq \pm 2$. The
situation with $\beta_{2,0}=\pm 2$ is that the equation is singular
or non-geometrically irreducible, which
is not of our concern in this work.

\end{rem}

\begin{rem} Given $G$ a non-trivial finite group such that $\widetilde{M_6^{Pl}(G)}$ is
non-empty, by a tedious work \cite{BaBaaut}, one can show that
$\widetilde{M_6^{Pl}(G)}$ is ES-Irreducible except the case
$G\cong\Z/4\Z$.
\end{rem}

\begin{thm} Let $d\geq 5$ be an odd integer and consider $g=(d-1)(d-2)/2$ as usual. Then $\widetilde{M_g^{Pl}(\Z/(d-1)\Z)}$ is not
ES-Irreducible and it has at least two irreducible components.
\end{thm}
\begin{proof}
The above argument for concrete curves of Type $4, (0,1)$ and Type
$4, (1,2)$ is valid for any odd degree $d\geq5$ and the proof is
quite similar. In other words, let $\tilde{C}$ and
$\tilde{\tilde{C}}$ be the non-singular plane curves of types $d-1,
(0,1)$ and $d-1, (1,2)$ defined by the equations
$X^d+Y^d+Z^{d-1}X+\beta X^{d-2}Y^2=0$ and
$X^d+X(Z^{d-1}+Y^{d-1})+\beta Y^2Z^{d-2}=0$ where $\beta\neq0$.
Then, $Aut(\tilde{C})$ and $Aut(\tilde{\tilde{C}})$ are cyclic of
order $d-1$ generated by $[X;Y;\xi_{d-1}Z]$ and
$[X;\xi_{d-1}Y;\xi_{d-1}^2Z]$ respectively which are cyclic groups
which are not conjugate, therefore belongs to two different
$[\rho]'s$.\\\\
\noindent\textbf{On type $d-1, (0,1)$:} With a homology of order
$d-1\geq 4$ inside $Aut(\tilde{C})$ we conclude that
$Aut(\tilde{C})$ should fix a point, a line or a triangle (See
\cite{Mit} \S 5). Moreover, the center $(0;0;1)$ of this homology is
an inner Galois point ( \cite{Harui} [Lemma $3.7$]) and then it is
unique by Yoshihara (\cite{Yoshihara} $\S 2$ [Theorem $4$]).
Therefore, it should be fixed by the $Aut(\tilde{C})$ (in
particular, the axis $Z=0$ is leaved invariant by \cite{Mit}
[Theorem 4]) hence $Aut(\tilde{C})$ is cyclic (\cite[Lemma
11.44]{Book}) and automorphisms of $\tilde{C}$ are of the form
$diag(1;v;t)$ such that $v^d=t^{d-1}=v=1$. That is,
$|Aut(\tilde{C})|=d-1$.
\\\\
\noindent\textbf{On type $d-1, (1,2)$:} First, we proof the
following claim.
\begin{claim}
The full automorphism group of non-singular plane curves of type
$d-1, (1,2)$ is not conjugate to any of the primitive groups
mentioned in Theorem \ref{teo1} (5). Moreover, $\tilde{\tilde{C}}$
is not a descendant of the Fermat curve $F_d:\,X^d+Y^d+Z^d$ or Klein
curve $K_d:\,X^{d-1}Y+Y^{d-1}Z+Z^{d-1}X$  in particular,
$Aut(\tilde{\tilde{C}})$ fixes a line and a point off this line.
\end{claim}
We consider the case $d\geq7$ (for $d=5$ we refer to the above
results) and the claim is straightforward. Indeed, the alternating
group $A_6$ has no elements of order $\geq6$, the Klein group
$PSL(2,7)$ which is the only simple group of order $168$ has no
elements of order $\geq8$ (also, no elements of order $6$ are
present inside \cite{Fano}) therefore the groups $A_6, A_5$ and the
Klein group do not appear as the full automorphism group. Moreover,
elements inside the Hessian group $SmallGroup(216,153)$ have orders
$1,2,3,4$ or $6$ then $Hess_{216}$ (consequently, $Hess_{36}$ and
$Hess_{72}$) do not appear except possibly for $d=7$. On the other
hand, $d-1\nmid 3(d^2-3d+3)$ that is $\tilde{\tilde{C}}$ is not a
descendant of the Klein curve $K_d$ also, is not a descendant of the
Fermat curve because $d-1\nmid 6d^2$ (except possibly for $d=7$).
Finally, it remains to treat the case $d=7$ for the Hessian groups
or being a Fermat's descendant. In this case, the Fermat curve has
automorphism group generated by the four transformations
$[\xi_{7}X;Y;Z],\,[X;\xi_{7}Y;Z],\,[X;Z;Y]$ and $[Y;Z;X]$
consequently elements of $Aut(F_7)$ are of the forms
$[X;\xi_{7}^aY;\xi_{7}^bZ],\,[\xi_{7}^bZ;\xi_{7}^aY;X],\,[X;\xi_{7}^bZ;\xi_{7}^aY],\,[\xi_{7}^aY;X;\xi_{7}^bZ],\,[\xi_{7}^aY;\xi_{7}^bZ;X]$
or $[\xi_{7}^bZ;X;\xi_{7}^aY]$ and one can easily verify that non of
them has order $6$ consequently, being a descendant of the Fermat
curve $F_7$ does not occur. Also, by the same argument as Claim 1 on
Type $4, (1,2)$ we exclude the Hessian groups and hence the claim is
proved.
\par Now, the full automorphism group should fix a line and a point which does not belong to this line thus all automorphisms of $\tilde{\tilde{C}}$ satisfy one of the forms $[X;vY+wZ;sY+tZ],\,[vX+wZ;Y;sX+tZ]$ or
$[vX+wY;sX+tY;Z]$ (because $\tilde{\tilde{C}}$ has an automorphism
$[X;\xi_{d-1}Y;\xi_{d-1}^2Z]$ of the form $diag(1;v;s)$ with $1,v,s$
(resp. $1,v^3,s^3$) are pairwise distinct). If $[X;vY+wZ;sY+tZ]\in
Aut(\tilde{\tilde{C}})$ then $s=0=w$ (Coefficient of $Y^d$ and
$Z^d$) and the same conclusion if $[vX+wZ;Y;sX+tZ]$ (resp.
$[vX+wY;sX+tY;Z]$) $\in Aut(\tilde{\tilde{C}})$ from the
coefficients of $X^{d-2}Y^2$ and $Y^{d-1}Z$ (resp. $Z^{d-2}X^2$ and
$YZ^{d-1}$). Hence, automorphisms of $\tilde{\tilde{C}}$ are all
diagonal and moreover, if $diag(1;v;s)\in Aut(\tilde{\tilde{C}})$
then we must have $v^{d-1}=s^{d-1}=v^2s^{d-2}=1$ that is
$v=\xi_{d-1}^r$ and $s=\xi_{d-1}^{r'}$ such that $d-1|2r-r'$ hence
$Aut(\tilde{\tilde{C}})$ is cyclic of order $d-1$.

\end{proof}

\section{On the locus $\widetilde{M_{10}^{Pl}(\Z/3\Z)}$.}
By a similar argument as this one that has been done for degree 5,
we obtain the following smooth plane equations of degree 6 with an
effective cyclic group of order 3 in the automorphism group.

\begin{center}

\begin{tabular}{|c|c|}
  \hline
  Type: $m, (a,\,b)$ & $F(X;Y;Z)$ \\\hline\hline
$3,(0,1)$& $Z^6+Z^3 L_{3,Z}+L_{6,Z}$\\ \hline $3,(1,2)$ &
$X^5Y+Y^5Z+Z^5X+\alpha Z^2X^4+\beta X^2 Y^4+\gamma Y^2Z^4+\delta
X^3Y^2Z+\mu XY^3Z^2+\eta X^2YZ^3$\\ \hline
\end{tabular}

\end{center}

\subsection{On type $3,(1,2)$}

\begin{prop}\label{nonprimitive}
Consider the plane curve $\tilde{C}$ defined by the equation
\[X^5Y+Y^5Z+Z^5X+\mu_1 Z^2X^4+\mu_2 X^2Y^4+\mu_3 Y^2Z^4+\delta_1 X^3Y^2Z+\delta_2 XY^3Z^2+\delta_3 X^2YZ^3=0.\]
Then, $Aut(\tilde{C})$ is not conjugate to any of the finite
primitive groups inside $PGL_3(K)$ namely, the Klein group
$PSL(2,7)$, the icosahedral group $A_5$, the alternating group
$A_6$, the Hessian group $Hess_{216}$ or to any of its subgroups
$Hess_{72}$ or $Hess_{36}$.
\end{prop}

\begin{proof}
Let $\tau\in Aut(\tilde{C})$ be an element of order $2$ such that
$\tau\sigma\tau=\sigma^{-1}$ where $\sigma:=[X;\omega Y;\omega^2Z]$
then $\tau$ has one of the forms $[X;\beta Z;\beta^{-1}Y],\,[\beta
Y;\beta^{-1}X;Z]$ or $[\beta Z;Y;\beta^{-1}X]$ a contradiction
because non of these forms retain $\tilde{C}$. Therefore,
$Aut(\tilde{C})$ does not contain an $S_3$ as a subgroup hence is
not conjugate to $A_5$ or $A_6$. Moreover, it is well known that
$PSL(2,7)$ contains an octahedral group of order $24$ (but not an
isocahedral group of order $60$) and since all elements of order 3
in $PSL(2,7)$ are conjugate (see \cite{Fano}) then by the same
argument as above we conclude that $Aut(\tilde{C})$ is not conjugate
$PSL(2,7)$.
\par Now, if $Aut(\tilde{C})$ is conjugate (through a transformation $P$) to any of the Hessian groups then we can assume that $PSP^{-1}=\lambda S$ (because we did not fixed the equation model for a curve having automorphism group any of the Hessian groups). In particular, $P$ has the form $[Y;\alpha Z;\beta X],\,\,[Z;\alpha X;\beta Y]$ or $[X;\alpha Y;\beta Z]$ but non of these transformation satisfy $\{[X;Z;Y],\,\,[Y;X;Z]\,\,[Z;Y;X]\}\subseteq Aut(P\tilde{C})$ thus $Aut(\tilde{C})$ is not conjugate to any of the Hessian groups.
\end{proof}

\begin{lem}\label{lem2}
If a finite group $G$ inside $PGL_3(K)$ fixes a line and a point off
that line moreover if it contains a diagonal element which is not a
homology then the fixed line must be one of the reference lines
$X=0,\,Y=0$ or $Z=0$ and the point that is leaved invariant is
$(1;0;0),\,(0;1;0)$ or $(0;0;1)$. In particular, all elements of $G$
satisfy a form of the following $[X;vY+wZ;sY+tZ],\,[vX+wZ;Y;sX+tZ]$
or $[vX+wY;sX+tY;Z]$.
\end{lem}

\noindent\textbf{Notations.} Let $\Upsilon:=\Gamma\cup\{-3\}$ where
$\Gamma$ is the set of all roots of the polynomial
$g_1(\upsilon)g_2(\upsilon)g_3(\upsilon)$ such that
\begin{eqnarray*}
g_1(\upsilon)&:=&\upsilon\times\left(\upsilon^2+\upsilon+7\right)(\upsilon (\upsilon+7)+49)\times(\upsilon^4-39 \upsilon^2+441)\\
g_2(\upsilon)&:=&(\upsilon \left(\upsilon \left(\upsilon (\upsilon+2) \left(\upsilon^2+\upsilon+4\right)+3\right)-3\right)+1)\\
g_3(\upsilon)&:=&(19 + (\upsilon-3)\upsilon (7 + (\upsilon-4)
\upsilon) (3 + (\upsilon-2) \upsilon))
\end{eqnarray*}
Now, we can prove the main result for this section.
\begin{thm}
Consider the curve defined by the equation
$\tilde{C}:\,\,\,X^5Y+Y^5Z+Z^5X+\delta_3 X^2YZ^3$ and assume for
simplicity that $\delta_3\notin\Upsilon$. Then $Aut(\tilde{C})$ is
cyclic of order $3$ and is generated by the transformation
$diag(1;\omega;\omega^2)$.
\end{thm}

\begin{proof}
It follows, by Proposition \ref{nonprimitive}, that $Aut(\tilde{C})$
is not conjugate to any of the finite primitive groups that have
been mentioned in Theorem \ref{teo1}. Therefore, the automorphisms
group (being a finite group in $\mathbb{P}^2(K)$ and by Mitchell in
\cite{Mit}) either fixes a line and a point off that line or it
fixes a triangle. In what follows, we treat each of these two cases.
\begin{enumerate}
  \item If $Aut(\tilde{C})$ fixes a point, it must be one of the reference points (by lemma \ref{lem2}).
  Consequently, $Aut(\tilde{C})$ is cyclic since all the reference points lie on $\tilde{C}$. Furthermore, elements of $Aut(\tilde{C})$
   are of the forms $$\tau_1:=[X;vY+wZ;sY+tZ],\,\tau_2:=[vX+wZ;Y;sX+tZ]\,\,\,\, or\,\,\, \tau_3:=[vX+wY;sX+tY;Z]$$
For $\tau_1$ to be in $Aut(\tilde{C})$, we must have $w=0=s$
(coefficients of $X^5Z$ and $XY^5$) and similarly, for $\tau_2$
(resp. $\tau_3$) through  the coefficients of $Y^5X$ and $Z^6$
(resp. coefficients of $YZ^5$ and $X^5Z$). That is, elements of
$Aut(\tilde{C})$ are all diagonal of the form $diag(1;v;t)$
moreover, $tv^4=1=t^3$ and $t^5=v$ therefore, $t=\xi_{3}^a$ and
$v=\xi_{3}^{2a}$ where $\xi_{3}$ is a primitive $3$-rd root of
unity. In particular, $|Aut(C)|=3$.
  \item If $Aut(\tilde{C})$ fixes a triangle and there exist neither a line or a point fixed by it, then by Harui \cite{Harui}, $\tilde{C}$
   is a descendant of the Fermat curve $F_6:\,\, X^6+Y^6+Z^6$ or the Klein curve $K_6:\,\,X^5Y+Y^5Z+Z^5X$. Hence, $Aut(\tilde{C})$ is conjugate to a subgroup of $Aut(F_6):=<\,[\xi_6X;Y;Z],\,[X;\xi_6Y;Z],\,[Y;Z;X],\,[X;Z;Y]\,>$ or to a subgroup of $Aut(K_6):=<\,[Z;X;Y],\,\,[X;\xi_{21}Y;\xi_{21}^{-4}Z]\,>.$
      \begin{itemize}
        \item Suppose first that $Aut(\tilde{C})$ is conjugate (through $P$) to a subgroup of $Aut(F_6)$ then $PSP^{-1}$ should be one of the automorphisms $S,S^{-1},[\xi_6^aY;\xi_6^bZ;X],\,[\xi_6^bZ;X;\xi_6^aY]$ (these are all the automorphisms of order 3 in $Aut(F_6)$ which are not homologies) and it suffices to consider the situations $P S P^{-1}=\lambda M$ where $M=S,[Y;Z;X],[Y;\xi_6Z;X]$ or $[Y;\xi_6^2Z;X]$ since any other $M$ is conjugate to one of those inside $Aut(F_6)$. Now, if $M=S$ then $P\in PGL_3(K)$ is of the form $[Y;\alpha Z;\beta X],\,\,[Z;\alpha X;\beta Y]$ or $[X;\alpha Y;\beta Z]$ and obviously, non of these transformations gives $X^6+Y^6+Z^6\in P\tilde{C}$ a contradiction. On the other hand, if $M=[Y;Z;X]$ (resp. $[Y;\xi_6Z;X]$ or $[Y;\xi_6^2Z;X]$) then
            $P$ has the form $\left(
      \begin{array}{ccc}
      \lambda & \lambda\omega^2\beta_2 & \lambda\omega\beta_3 \\
      1 & \beta_2 & \beta_3 \\
      \lambda^{2} & \lambda^{2}\omega\beta_2 & \lambda^{2}\omega^2\beta_3 \\
       \end{array}
        \right)$ where $\lambda^3=1$ (resp. $\lambda^3\xi_6=1$ or $\lambda^3\xi_6^2=1$) moreover $X^5Z\in P\tilde{C}$ (since $\delta_3$ is not a root of $g_1(\upsilon)$. In particular, $P\tilde{C}$ should has the form  $$P\tilde{C}:\,\,X^6+Y^6+Z^6+k_1(X^5Z+Y^5X+aZ^5Y)+\,\,other\,\,terms$$ where $a=1$ (resp. $\xi_6$ or $\xi_6^2$) and then such transformations exist only if $\delta_3$ is a root of $g_2(\upsilon)g_3(\upsilon).$ Consequently, $\tilde{C}$ is not a descendant of the Fermat sextic curve.
%
%
%
%
%
%

        \item Lastly, assume that $Aut(C)$ is conjugate (through $P$) to a subgroup of $Aut(K_6)$ with $$P\tilde{C}:\,\,X^5Y+Y^5Z+Z^5X+lower\,\,\,terms.$$ Then, we must have $PSP^{-1}=\lambda S$ for some $\lambda\in K^*$. Indeed, elements of order $3$ inside $Aut(K_6)$ (which are not homologies) are $S,S^{-1},[\xi_{21}^aY;\xi_{21}^{-4a}Z;X]$ and $[\xi_{21}^{-4a}Z;X;\xi_{21}^aY]$ and it is enough to consider the cases $PSP^{-1}=\lambda M$ where $M$ is could assume any of the values $S,S^{-1},[\xi_{21}^aY;\xi_{21}^{-4a}Z;X]$ and $[\xi_{21}^{-4a}Z;X;\xi_{21}^aY]$ with $a=0,1,2$ (because any other value is conjugate to one of those inside $Aut(K_6)$). Now, assume that $PSP^{-1}=\lambda [\xi_{21}^aY;\xi_{21}^{-4a}Z;X]$ (respectively, $[\xi_{21}^{-4a}Z;X;\xi_{21}^aY]$) then $P$ has the form
$$\left(
    \begin{array}{ccc}
      \lambda\xi_{21}^a & \lambda\xi_{21}^a\omega^2\beta_2 & \lambda\xi_{21}^a\omega\beta_3 \\
      1 & \beta_2 & \beta_3 \\
      \lambda^{2}\xi_{21}^{a}& \lambda^{2}\xi_{21}^{a}\omega\beta_2 & \lambda^{2}\xi_{21}^{a}\omega^2\beta_3 \\
    \end{array}
  \right)
\,\,\,(respectively,\,\,\,\left(
    \begin{array}{ccc}
      \lambda^{2}\xi_{21}^{18a} & \lambda^{2}\xi_{21}^{18a}\omega\beta_2 & \lambda^{2}\xi_{21}^{18a}\omega^2\beta_3 \\
      1 & \beta_2 & \beta_3 \\
      \lambda\xi_{21}^{a} & \lambda\xi_{21}^{a}\omega^2\beta_2 & \lambda\xi_{21}^{a}\omega\beta_3 \\
    \end{array}
  \right) )$$
where $\lambda^3=\xi_{21}^{3a}$. Since $\delta_3\neq0,-3$ then non
of the above transforms $\tilde{C}$ to a descendant of the Klein
curve (because $X^6$ appears in $P\tilde{C}$). On the other hand, if
$PSP^{-1}=\lambda S^{-1}$ then $P$ fixes one of the variables and
permutes the other two hence the resulting core is different from
$X^5Y+Y^5Z+Z^5X$ consequently $PSP^{-1}=\lambda S$. That is, $P$ has
one of the forms $[Y;\alpha Z;\beta X],\,\,[Z;\alpha X;\beta Y]$ or
$[X;\alpha Y;\beta Z]$ consequently $\tilde{C}$ is transformed to
the form $k_0(X^5Y+Y^5Z+Z^5X)+\delta_3k_1$(monomial). More
precisely, this monomial is one of the monomials $X^2YZ^3,Y^2ZX^3$
or $Z^2XY^3$ consequently, $[Z;X;Y]\notin Aut(P\tilde{C})$. Thus,
$Aut(P\tilde{C})\preceq <\,\tau:=[X;\xi_{21}Y;\xi_{21}^{-4}Z]>$
moreover $\tau^r\in Aut(P\tilde{C})$ if and only if $7|r$ that is,
$Aut(\tilde{C})$ has order 3.
      \end{itemize}

\end{enumerate}
This completes the proof.
\end{proof}

\subsection{On type $3,(0,1)$}
\begin{lem}\label{lem4}
Consider the non-singular plane curve
$\tilde{\tilde{C}}:\,Z^6+Z^3L_{3,Z}+L_{6,Z}$. Then the full
automorphism group of $\tilde{\tilde{C}}$ is conjugate to the
Hessian group $Hess_{216}$ or it leaves invariant a point, a line or
a triangle.
\end{lem}

\begin{proof}
Since $Aut(\tilde{\tilde{C}})$ contains a homology (fix a line and a
point off in $\mathbb{P}^2(K)$) of period $3$ namely,
$\sigma:=[X;Y;\omega Z]$ then the result follows because
$Hess_{216}$ is the only multiplicative group that contains such
homologies and does not leave invariant a point, a line or a
triangle (Theorem $9$ \cite{Mit}).
\end{proof}

Now, we can prove our main result for this section.
\begin{thm}
The automorphisms group of the curve $Z^6+X^5Y+XY^5+\alpha_3Z^3X^3$
such that $\alpha_3\neq0$ is cyclic and has order 3.
\end{thm}

\begin{proof}
Assume that $Aut(\tilde{\tilde{C}})$ is conjugate (through a
transformation $P$) to the Hessian group $Hess_{216}$ then we can
assume that $P\sigma P^{-1}=\lambda \sigma$ for some $\lambda\in
K^*$. Hence, $P=[\alpha_1X+\alpha_2Y;\beta_1X+\beta_2Y;Z]$ and
clearly, $\{[Z;Y;X],[X;Z;Y]\}\nsubseteq Aut(P\tilde{\tilde{C}})$ a
contradiction. Now, it follows by lemma \ref{lem4} that
$Aut(\tilde{\tilde{C}})$ should fix a point, a line or a triangle.
In what follows, we treat each case.
\begin{enumerate}
  \item If $Aut(\tilde{\tilde{C}})$ fixes a line and a point off that line and if $\tilde{\tilde{C}}$ admits a bigger non cyclic automorphism group then $Aut(\tilde{\tilde{C}})$ satisfies a short exact sequence of the form $1\rightarrow C_3\rightarrow Aut(\tilde{\tilde{C}})\rightarrow G'\rightarrow 1$ where $G'$ is conjugate to $C_m$ ($m=2,3$ or $4$), $D_{2m}$ ($m=2$ or $4$), $A_4,S_4$ or $A_5$.
      \par If $G'$ is conjugate to $C_3,A_4,S_4$ or $A_5$ then there exists (by Sylow's theorem) a subgroup $H\preceq Aut(\tilde{\tilde{C}})$ of order $9$. In particular, $H$ is conjugate to $C_9$ or $C_3\times C_3$ but both cases do not occur. Indeed, if $H=C_9$ then $Aut(\tilde{\tilde{C}})$ has an element of order $9$ a contradiction since $9\nmid d-1,d,(d-1)^2,d(d-2),d(d-1),d^2-3d+3$ with $d=6$ (for more details, we refer to \cite{BaBacyc}) moreover if $H=C_3\times C_3$ then there exists $\tau\in Aut(\tilde{\tilde{C}})$ of order $3$ such that $\tau\sigma=\sigma\tau$ hence $\tau=[vX+wY;sX+tY;Z]$. Comparing the coefficients of $Z^3Y^3$ and $X^6$ in $\tau(C)$ we get $w=0=s$ and $v^5t=vt^5=v^3=1$ (thus $\tau\in<\sigma>$) a contradiction. By a similar argument, we exclude the cases $C_4$ and $D_{2m}$ because for each $SmallGroup(6m,ID)$ there must be an element $\tau$ of order $2$ or $4$ that commutes with $\sigma$. Finally, if $G'$ is conjugate to $C_2$ then there exists an element $\tau$ of order $2$ such that $\tau\sigma\tau=\sigma^{-1}$ and one can easily verify that such an element does not exists.\\
      We conclude that $Aut(\tilde{\tilde{C}})$ should be cyclic (in particular, is commutative) hence can not be of order $>3.$ (otherwise; there exists an element $\tau\in Aut(\tilde{\tilde{C}})$ of order $>3$ which commutes with $\sigma$ and by a previous argument such elements do not exist).
  \item If $Aut(\tilde{\tilde{C}})$ fixed a triangle and neither a point nor a line is fixed, then it follows by Harui \cite{Harui} that $\tilde{\tilde{C}}$ is a descendant of the Fermat curve $F_6$ or the Klein curve $K_6$. The last case does not happen because $P\sigma P^{-1}$ should be an automorphism of $K_6$ of order 3 whose Jordan form is given as $\sigma$ (i.e a homology) but there are no such elements in $Aut(K_6)$.
Now, suppose that $\tilde{\tilde{C}}$ is a descendant of $F_6$ that
is, $\tilde{\tilde{C}}$ can be transformed (through $P$) into a
curve $P\tilde{\tilde{C}}$ with core $X^6+Y^6+Z^6$ then
$P=[\alpha_1X+\alpha_2Y;\,\beta_1X+\beta_2Y;\,Z]$ ( because in
$Aut(F_6)$ there are only two sets of homologies of order 3 namely;
$\{[\omega X;Y;Z],\,[X;\omega Y;Z],\,[X;Y;\omega Z]\}$ and
$\{[\omega^2 X;Y;Z],\,[X;\omega^2 Y;Z],\, [X;Y;\omega^2 Z]\}$ and
they are not conjugate in $PGL_{3}(K)$. Moreover, elements of the
first set are all conjugate to $[X;Y;\omega Z]$ inside $Aut(F_6)$ so
it suffices to consider $P\sigma P^{-1}=\lambda\sigma$). Now,
$P\tilde{\tilde{C}}$ has the form
$$
C':\,\,\mu_0X^6+\mu_1Y^6+Z^6+\alpha_3(\alpha_1X+\alpha_2Y)^3Z^3+\mu_2X^5Y+\mu_3X^4Y^2+\mu_4X^3Y^3+\mu_5X^2Y^4+\mu_6XY^5
$$
where $\mu_0:=\alpha _1 \beta _1 \left(\alpha
_1^4+\beta_1^4\right)(=1)$ and $\mu_1:=\alpha _2 \beta _2
\left(\alpha _2^4+\beta_2^4\right)(=1)$. In particular,
$(\alpha_1\beta_1)(\alpha_2\beta_2)\neq0$ therefore,
$[X;vZ;wY],\,[vZ;wY;X],\,[wY;vZ;X]$ and $[vZ;X;wY]\notin
Aut(P\tilde{\tilde{C}})$ (because of the monomial $XY^2Z^3$)
moreover, $[wY;X;vZ]\in Aut(P\tilde{\tilde{C}})$ only if
$\alpha_1=\alpha_2$ and $w=v^3=1$ hence
$$P\tilde{\tilde{C}}:\,Z^6+\alpha_3\alpha_1^3(X+Y)^3Z^3+\alpha
_1(X+Y) \left(\beta_1X+\beta_2Y \right)\left(\alpha_1^4(X+Y)^4
+\left(\beta _1X +\beta_2Y \right)^4\right).$$ Consequently,
$\beta_1=\beta_2$ (because we are assuming $[Y;X;vZ]\in
Aut(P\tilde{\tilde{C}})$) a contradiction. Finally, if
$[X,\xi_6^rY,\xi_6^{r'}Z]\in Aut(P\tilde{\tilde{C}})$ then $r=0$ and
$2|r'$ (since $\alpha_1\alpha_2\neq0$) that is,
$|Aut(P\tilde{\tilde{C}})|=3$ and we are done.
\end{enumerate}
\end{proof}
As a conclusion of the results that are introduced in this section
we get the following result.
\begin{thm} The locus $\widetilde{M_{10}^{Pl}(\Z/3\Z)}$ is not
ES-Irreducible and it has at least two irreducible components.
\end{thm}

\section{Positive characteristic}

Now, suppose that $\mathbb{K}$ is an algebraically closed field of
positive characteristic $p>0$. Consider a non-singular plane curve
$C$ in $\mathbb{P}^2(\mathbb{K})$ of degree $d$ and assume that the
order of $Aut(C)$ is coprime with $p$, $p\nmid d(d-1)$, $p\geq 7$
and the order of $Aut(F_d)$ and $Aut(K_d)$ are coprime with $p$
where $F_d: X^d+Y^d+Z^d=0$ is the Fermat curve and
$K_d:X^{d-1}Y+Y^{d-1}Z+Z^{d-1}X=0$ is the Klein curve. Then, all the
techniques that appeared in Harui \cite{Harui}, can be applied:
Hurwitz bound, Arakawa and Oiakawa inequalities and so on. In
particular, the
arguments of all the previous sections hold.

Consider the $p$-torsion of the degree 0 Picard group of $C$, which
is a finitely generated $\mathbb{Z}/(p)$-module of dimension
$\gamma$(always $\gamma\leq g$ where $g$ is the genus of $C$), we
call $\gamma$ the $p$-rank of $C$.

For a point $P$ of $C$ denote by $Aut(C)_{P}$ the subgroup of
$Aut(C)$ that fixes the place $P$.
\begin{lem} Assume that $Aut(C)_P$ is prime to $p$ for any point $P$
of $C$ and the $p$-rank of $C$ is trivial. Then $Aut(C)$ is prime to
$p$.
\end{lem}
\begin{proof} Consider $\sigma\in Aut(C)$ of order $p$, then the
extension $\mathbb{K}(C)/\mathbb{K}(C)^{\sigma}$ is a finite
extension of degree $p$ and is unramified everywhere (because if it
ramifies at a place $P$ then $\sigma$ will be an element of
$Aut(C)_P$ giving a contradiction). But, if $\gamma=0$ (i.e. the
$p$-rank is trivial for $C$) then, from Deuring-Shafarevich formula
\cite[Theorem11.62]{Book}, we obtain that
$\frac{\gamma-1}{\gamma'-1}=p$ where $\gamma'$ is the $p$-rank for
$C/<\sigma>$ which is impossible. Therefore, such extensions do not
exist.
\end{proof}

\begin{lem}\label{lem6.2} Consider $C$ a plane non-singular curve of degree $d\geq 4$.
If $p> (d-1)(d-2)+1$, then $Aut(C)_P$ is coprime with $p$ for any
point $P$ of the curve $C$.
\end{lem}
\begin{proof} By \cite[Theorem 11.78]{Book} the maximal order of
the $p$-subgroup of $Aut(C)_P$ is at most $\frac{4p}{(p-1)^2} g^2$.
Hence, with $g=\frac{(d-1)(d-2)}{2}$ and assuming that
$p>\frac{4p}{(p-1)^2} g^2$, we obtain the result.
\end{proof}

\begin{lem}\label{lem6.3} Let $C$ be a non-singular curve of genus $g\geq 2$ defined over an algebraic
closed field $\mathbb{K}$ of characteristic $p>0$. Suppose that $C$
has an unramified subcover of degree $p$, i.e. $\Phi:C\rightarrow
C'$ of degree $p$. Then $C'$ has genus $\geq 2$, $g\equiv 1(mod\ p)$
and $\gamma\equiv 1(mod\ p)$. In particular, for the existence of
such subcover, one needs to assume that $p<g$.
\end{lem}
\begin{proof} The Hurwitz formula for $\Phi$ gives the equality
$(2g-2)=p(2g'-2)$ where $g'$ is the genus of $C'$. We have $g'\neq
0$ or $1$ because $g\geq 2$, therefore $g'\geq 2$ and $g-1\equiv
0(mod\ p)$. Now, consider Deuring-Shafaravich formula, which in such
unramified extension could be read as $\gamma-1=p(\gamma'-1)$ where
$\gamma'$ the $p$-rank of $C'$. If $\gamma=1$ then there is nothing
to prove and if $\gamma>1$ then the congruence is clear. Finally, if
$\gamma=0$ then this situation is not possible.
\end{proof}

\begin{cor} Let $C$ be a non-singular plane curve of degree $d$ and genus $g\geq 2$ defined over an algebraic
closed field $\mathbb{K}$ of characteristic $p>0$. Suppose that
$p>(d-1)(d-2)+1>g$. Then the order of $Aut(C)$ is coprime with $p$.
\end{cor}
\begin{proof} Suppose $\sigma\in Aut(C)$ of order $p$, then
$\mathbb{K}(C)/\mathbb{K}(C)^{\sigma}$ is a separable degree $p$
extension, and by Lemma \ref{lem6.2}, it is unramified everywhere.
By
Lemma \ref{lem6.3} we find that such extensions do not exist. 
\end{proof}

And as a direct consequence of the above lemmas and because all
techniques in the previous sections, from \cite{Harui}, are
applicaple when $Aut(C)$ is coprime with $p$, then we obtain:
\begin{cor} Assume $p>13$. The automorphism groups of the curves $\tilde{C}:
X^5+Y^5+Z^4X+\beta X^3Y^2$ and
$\tilde{\tilde{C}}:\,\,\,X^5+X(Z^4+Y^4)+\beta Y^2Z^3$ such that
$\beta\neq0$, are cyclic of order 4. Moreover, $\tilde{C}$ is not
isomorphic to $\tilde{\tilde{C}}$ for any choice of the parameters.
\end{cor}
\begin{proof} Only we need to mention that the linear $g_2$-systems
for the immersion of the curve inside $\mathbb{P}^2$ are unique up
to conjugation in $PGL_3(\mathbb{K})$ see \cite[Lemma 11.28]{Book} (
also the curves $\tilde{C}$ and $\tilde{\tilde{C}}$ have cyclic
covers of degree 4 with different type of the cover, from Hurwitz
equation, therefore they belong to different irreducible components
in the moduli space of genus 6 curves).
\end{proof}

\begin{cor} For $p>13$ we have that the locus $\widetilde{M_6^{Pl}(\Z/4\Z)}$ of the moduli space of
positive characteristic, has at least two irreducible components.
\end{cor}

Similarly we obtain the following result from results in \S4,

\begin{cor} For $p>(d-1)(d-2)+1$ where $d\geq 5$ is an odd integer, the locus $\widetilde{M_g^{Pl}(\Z/(d-1)\Z)}$ of the
moduli space over positive characteristic $p$ is not ES-Irreducible
and it has at least two strongly equation components. In particular,
it has at least two irreducible components.
\end{cor}

\chapter{Plane non-singular curves with large cyclic automorphism
groups}

\section{Abstract} In this elementary note,
we present a result for obtaining the full list of the cyclic groups that could
appear as an effective action for projective non-singular, non-hyperelliptic plane curves $C$ of
a fixed degree $d$ over an algebraic closed field of zero
characteristic. In particular, for any cyclic subgroup of
$Aut(C)$, its order should divide one of the following numbers:
$d-1$,$d$,$d^2-3d+3$, $(d-1)^2$, $d(d-2)$ or $d(d-1)$. Moreover, we attach to each cyclic group an equation $C$ (unique up to
$K$-isomorphism) that admits the cyclic group $G$ as a subgroup of the full
automorphism group. Also, we provide an algorithm of computation of the full list (once the degree is fixed) together with the implementation of this algorithm in SAGE program. On the other hand, we determine the full automorphism group and a
characterization of the unique equation for the curve $C$ that admits an automorphism of order $d^2-3d+3$, $(d-1)^2$, $d(d-2)$ or $d(d-1)$ in
$Aut(C)$. We also consider the cases when $Aut(C)$ has a cyclic subgroup of
order $\ell d$ or $\ell (d-1)$ with $1\leq \ell< d-2$.

\section{Introduction}
Consider $C$, a non-singular projective plane curve of degree $d\geq 4$ over an
algebraically closed field $K$ of zero characteristic. Write by
$F(X,Y,Z)=0$ the defining equation of $C$ in $\mathbb{P}^2(K)$ ( in particular, the polynomial $F$ has degree $\geq d-1$ in each
variable because $C$ is non-singular). Denote by $Aut(C)$, the full
automorphism group of $C$ over $K$, which is the finite subgroup of
$PGL_3(K)$ that preserves the defining equation $F(X,Y,Z)=0$ of $C$.

Harui in \cite{Harui} gave a list of the possible $G\leq Aut(C)$
that could appear as: cyclic groups, an element of $Ext^1(-,-)$ of a cyclic group by a dihedral group, an alternating group $A_4$ (resp. $A_5$), an octahedral group $S_4$, a subgroup of $Aut(F_d)$ (the automorphism group of the Fermat curve), a subgroup of $Aut(K_d)$ (the automorphism group of the Klein curve) or a finite primitive subgroup of $PGL_3(K)$ (for the classification of such groups, we refer to Mitchell \cite{Mit}).

For a fixed degree $d$, the previous result does not give the precise list of groups
that appear and and the corresponding equations with an effective action given by $G$ (up to
$K$-isomorphism). In the literature, as far as we know, this is only given for $d=4$ by Henn
\cite{He} (see also the manuscript \cite{BaBaaut} for $d=5$). Recall
that, two non-singular plane curves, with equations given by $F_1(X,Y,Z)=0$ and
$F_2(X,Y,Z)$ respectively, are $K$-isomorphic if and only if $F_1$ can be transformed through a change of
variables $P\in GL_3(K)$ to $F_2$ and in particular
$Aut(C_2)$ is conjugate inside $PGL_3(K)$ to $Aut(C_1)$ in the sense that $Aut(C_2)=PAut(C_1)P^{-1}$.

Now, let $d\geq 5$ be an arbitrary but fixed, we
present the first clear statements which lists the exact cyclic groups $G$ (of order $m$) associated with equations (equipped with a family of parameters) of the plane curves that have an effective action given by $G$ (up to $K$-isomorphism). In particular, we prove
that the order $m$ of any cyclic group that might appear should divide one of the following integers:
$d-1$, $d$, $d^2-3d+3$, $(d-1)^2$, $d(d-2)$ or $d(d-1)$ [Theorem \ref{thm20}].
On the other hand, we formulate the main result as an algorithm (CAGPC) together with a complete implementation in SAGE program (for more details, one may use the link in Remark \ref{rem21}).
On the other hand, it is very classical to determine those curves with large
automorphism group, see for example \cite{Swi} and the interest of
their equations up to $K$-isomorphism. Consequently, the second facet of this chapter
is to determine (up to $K$-isomorphism) the smooth plane curves of
degree $d$ which have a cyclic subgroup of large order in their
automorphism group (by a large order we mean $d^2-3d+3$,
$(d-1)^2$, $d(d-2)$ or $d(d-1)$). For such cases, we determine the
equation of the curve $F(X;Y;Z)=0$ (up to $K$-isomorphism) and its
full automorphism group. Moreover, we solve the same question for
non-singular plane curves with a cyclic subgroup of $Aut(C)$ of
order $\ell (d-1)$ with $\ell\geq 2$, and we obtain a partial result
for $C$ with a cyclic subgroup of order $\ell d$ with $\ell\geq 3$
(see remark \ref{rem33} and corollary \ref{cor44} for general
$\ell\geq 1$).

\section{Cyclic automorphism group of non-singular plane curves}

Fix an integer $d\geq 4$ and let $C$ be a non-singular projective plane curve of
degree $d$. Our aim here is to determine the cyclic groups that appear
as subgroups inside $Aut(C)$. We follow a similar approach of Dolgachev
\cite{Dol} which treat the same question for $d=4$ (see \cite[\S
2.1]{Bars}).

Let $\sigma\in Aut(C)$ be of maximal order $m$, then we can assume, without any loss of generality, that $\sigma$ is given in its canonical Jordan form by $(x:y:z)\mapsto
(x:\xi_m^a y:\xi_m^b z)$ where $\xi_m$ is a primitive $m$-th root of
unity in $K$ and $a,b$ integers such that $0\leq a\neq b\leq m-1$
with $a\leq b$ and $gcd(a,b)$ coprime with $m$ if $ab\neq 0$(
we can reduce to $gcd(a,b)=1$) and $gcd(b,m)=1$ if $a=0$.

Throughout this paper, we also use also the following notations.
\begin{itemize}
\item Type $m, (a,b)$ is identified with the corresponding automorphism $[X;\zeta_m^aY;\zeta_m^bY]$ where $\zeta_m$ is a primitive $m$-th root of unity. In other words, saying that $m, (a,b)$ is a generator
means that the corresponding automorphism generates any other
automorphism in the solution set given by the composition of this
fix automorphism. \item $L_{i,*}$ denotes a degree $i$ homogeneous
polynomial in $K[X,Y,Z]$ without the variable $*$ where
$*\in\{X,Y,Z\}$.
\item $S(u)_m:=\{j:\,\,u\leq j\leq d-1,\,\,d-j=0\,(mod\,\,m)\}$.
\item $S_u^{d,X}\,\,{m,(a,b)}:=\{i:\,\,u\leq i\leq d-u\,\,\text{and}\,\,ai+(d-i)b=0\,(mod\,\,m)\}$.
\item $S_u^{d-1,X}\,\,{m,(a,b)}:=\{i:\,\,1\leq i\leq d-u\,\,\text{and}\,\,ai+(d-1-i)b=0\,(mod\,\,m)\}$
\item $S(1)^{j,X}_{m,(a,b)}:=\{i:\,\,0\leq i\leq j\,\,\text{and}\,\,ai+(j-i)b=a\,(mod\,\,m)$.
\item $S(2)^{j,X}_{m,(a,b)}:=\{i:\,\,0\leq i\leq j\,\,\text{and}\,\,ai+(j-i)b=0\,(mod\,\,m)$.
\item $S^{j,Y}_{m,(a,b)}:=\{i:\,\,0\leq i\leq j\,\,\text{and}\,\,bi+(d-j)a=a\,(mod\,\,m)\}$.
\item $S^{j,Z}_{m,(a,b)}:=\{i:\,\,0\leq i\leq j\,\,\text{and}\,\,ai+(d-j)b=a\,(mod\,\,m)\}.$
\item $\Gamma_m:=\{(a,b)\in \mathbb{Z}^2:\,\,g.c.d\,(a,b)=1,\,\,\,1\leq a\neq b\leq m-1\}.$
\item the points of $\mathbb{P}^2(K)$: $(1:0:0),(0:1:0)$ and
$(0:0:1)$ are named reference points. \item $\alpha$ always is an
element in $K^*$ which by a change of variables can be assumed equal
to 1.
\end{itemize}
where $u,j,m,d,a$ and $b$ are non-negative integers.

\begin{thm}\label{thm20}
Let $C$ be a non-singular projective plane curve of degree $d$ over
an algebraically closed field ${K}$ of zero characteristic. If $G$
is a non-trivial cyclic subgroup of $Aut(C)$ of order $m$, then the
classification of such curves up to their cyclic subgroups is given
by the following.\,(note that, we attach to each type an equation of
$C$ which is unique up to ${K}$-isomorphism)\footnote{We warm the
reader that may happen that the curve given for a certain type
$m(a,b)$ could be never geometrically irreducible for any parameter
and we should discard from the list}:
\begin{enumerate}
  \item The curve $C$ is of the form $Z^{d-1}L_{1,Z}+\big(\sum_{j\in S(2)_m}Z^{d-j}L_{j,Z}\big)+L_{d,Z},$ which is type $m,\,(0,1)$ for some $m|d-1$.

  \item The curve $C$ has the form $Z^d+\big(\sum_{j\in S(1)_m}Z^{d-j}L_{j,Z}\big)+L_{d,Z},$ and $C$ is of type $m,\,(0,1)$ for some $m|d$.

 \item All reference points lie on $C$ which is of type $m,\,(a,b)$ for some $m\,|\,(d^2-3d+3)$ and $(a,b)\in\Gamma_m$ such that $a=(d-1)a+b=(d-1)b\,(mod\,\,m)$. Furthermore, $C$ is defined by the equation
 \begin{eqnarray*}
X^{d-1}Y&+&Y^{d-1}Z+\alpha Z^{d-1}X+\\
&+&\sum_{j=2}^{\lfloor\frac{d}{2}\rfloor}\,\,X^{d-j}\big(\sum_{i\in S(1)^{j,X}_{m,(a,b)}}\beta_{ji}Y^iZ^{j-i}\big)+Y^{d-j}\big(\sum_{i\in S^{j,Y}_{m,(a,b)}}\alpha_{ji}Z^iX^{j-i}\big)+Z^{d-j}\big(\sum_{i\in S^{j,Z}_{m,(a,b)}}\gamma_{ji}X^{j-i}Y^i\big),
\end{eqnarray*}
\item Two reference points lie on $C$ which is of type $m,\,(a,b)$ for one of the following subcases.
\begin{enumerate}
  \item[(4.1)] First, $m\,|\,d(d-2)$ and $(a,b)\in\Gamma_m$ such that $(d-1)a+b\equiv0\,(mod\,m)$ and $a+(d-1)b\equiv0\,(mod\,m)\,\,.$ Moreover, $C$ is given by
\[
X^d+\big(\sum_{j=2}^{d-1}\,\,X^{d-j}\sum_{i\in S(2)^{j,X}_{m,(a,b)}}\beta_{ji}Y^iZ^{j-i}\big)
+\big(Y^{d-1}Z+\alpha YZ^{d-1}+\sum_{i\in S_2^{d,X}\,\,{m,(a,b)}}\beta_{di}Y^iZ^{d-i}\big)=0,
\]

\item[(4.2)] The second case with $m|(d-1)^2$ and $(a,b)\in\Gamma_m$ such that $(d-1)a+b\equiv0\, (mod\,m)$ and $(d-1)b\equiv0\, (mod\,m)$. Furthermore, the defining equation of $C$ is
\begin{eqnarray*}
X^d&+&\sum_{j=2}^{d-2}\,\,X^{d-j}\big(\sum_{i\in S(2)^{j,X}_{m,(a,b)}}\beta_{ji}Y^iZ^{j-i}\big)+X\big(\alpha Z^{d-1}+\sum_{i\in S_1^{d-1,X}\,\,{m,(a,b)}}\beta_{(d-1)i}Y^iZ^{d-1-i}\big)+\\
&+&\big(Y^{d-1}Z+\sum_{i\in S_2^{d,X}\,\,{m,(a,b)}}\beta_{di}Y^iZ^{d-i}\big)=0
\end{eqnarray*}
\item[(4.3)] Lastly, $m|(d-1)$ and $(a,b)\in\Gamma_m$ such that $(d-1)b\equiv0\,(mod\,m)$ and $(d-1)a\equiv0\,(mod\,m)$. In such case $C$ has the form
\begin{eqnarray*}
X^d&+&\sum_{j=2}^{d-2}\,\,\big(X^{d-j}\sum_{i\in S(2)^{j,X}_{m,(a,b)}}\beta_{ji}Y^iZ^{j-i}\big)+\sum_{i\in S_2^{d,X}\,\,{m,\,(a,b)}}\beta_{di}Y^iZ^{d-i}+\\
&+&X\big(Z^{d-1}+\alpha Y^{d-1}+\sum_{i\in S_2^{d-1,X}\,\,{m,\,(a,b)}}\beta_{(d-1)i}Y^iZ^{d-1-i}\big),
\end{eqnarray*}
\end{enumerate}

\item One reference points lie in $C$ that is of type $m,\,(a,b)$ for some $m|\,d(d-1)$ and $(a,b)\in\Gamma_m$ such that $da\equiv0\,(mod\,m)$ and $(d-1)b\equiv0\,(mod\,m)$. Moreover, $C$ is defined by
\begin{eqnarray*}
C:\,\,X^d&+&Y^d+\sum_{j=2}^{d-2}\,\,\big(X^{d-j}\sum_{i\in S(2)^{j,X}_{m,(a,b)}}\beta_{ji}Y^iZ^{j-i}\big)+\sum_{i\in S_1^{d,X}\,\,{m,\,(a,b)}}\beta_{di}Y^iZ^{d-i}+\\
&+&X\big(\alpha Z^{d-1}+\sum_{i\in S_1^{d-1,X}\,\,{m,\,(a,b)}}\beta_{(d-1)i}Y^iZ^{d-1-i}\big)=0
\end{eqnarray*}
\item None of the reference points lie on $C$ and we obtain type $m, (a,b)$ for $m|d$ and $(a,b)\in\Gamma_m$ such that $da\equiv0\,(mod\,m)$ and $db\equiv0\,(mod\,m)$. Furthermore, we have
   \[
C:\,\,X^d+Y^d+Z^d+\sum_{j=2}^{d-1}\,\,\big(X^{d-j}\sum_{i\in S(2)^{j,X}_{m,(a,b)}}\beta_{ji}Y^iZ^{j-i}\big)
+\sum_{i\in S_1^{d,X}\,\,{m,\,(a,b)}}\beta_{di}Y^iZ^{d-i}=0.
\]
     \end{enumerate}
Here, $\alpha,\,\beta_{ij}, \gamma_{ij}, \alpha_{ij}$ are in ${K}$
and $\alpha\neq0$.

\end{thm}
\begin{rem}\label{rem21}
The above result and its proof give an algorithm to list for
every fix degree $d$, all the cyclic groups that could appear with
an equation (up to $K$-isomorphisms), for the complete algorithm and
its implementation in SAGE see the link
{http://mat.uab.cat/$\sim$eslam/CAGPC.sagews}, and read the last
section of this chapter for a list of cyclic groups that appears for
lower degree $d$ with their equations.
\end{rem}
\begin{rem} The above result is also true when $K$ is an algebraic
closed field of characteristic $p>0$ and we assume from the
beginning that $m$ is always coprime with $p$, covering the cyclic
groups of order coprime with the characteristic.
\end{rem}

\begin{proof}
Let $\varphi$ be a generator of order $m:=|G|$. One can choose
coordinates so that $\varphi$ is represented by
$\big(X;Y,Z\big)\mapsto\big(X;\xi_m^aY,\xi_m^bZ\big)$ where, $\xi_m$
is a primitive $m$-th root of unity in $K$ and $a, b$ are integers
with $0\leq a\neq b\leq m-1$\,$\big($\,without loss of generality,
one can assume that
\, $a\leq b$, with $gcd(b,m)=1$ if $a=0$ and with $gcd(a,b)=1$ if $a\neq0$ $\big)$. Now, we have two cases, either $a=0$ or $a\neq0$. In the following we treat each of these cases.\\
\par\textbf{Case I\,:} Suppose first that $a=0$, write:
$F(X;Y;Z)=\lambda Z^d+\big(\sum_{j=1}^{d-1}Z^{d-j}L_{j,Z}\big)+L_{d,Z}.$
\par If $\lambda=0$, then $(d-1)b\equiv0\,(mod\,m)$. Hence, $m|d-1$  and we can take a generator $(a,b)=(0,1)$. Therefore, by checking each monomial's invaraince, we obtain type $m,\,(0,1)$ which is defined by the equation $Z^{d-1}L_{1,Z}+\big(\sum_{j\in S(2)_m}Z^{d-j}L_{j,Z}\big)+L_{d,Z},$ and $(1)$ arises.
\par If $\lambda\neq0,$ then $db\equiv0(mod\, m)$. From which, we have $m|d$ and $(a,b)=(0,1)$ is a generator for each such $m$. Consequently, we get type $m,\,(0,1)$ of the form $Z^d+\big(\sum_{j\in S(1)_m}Z^{d-j}L_{j,Z}\big)+L_{d,Z},$, which is precisely $(2)$.\\
\par\textbf{Case II\,:} Suppose that $a\neq0$, then, necessarily, $m>2.$ Now, Let $P_1=(1;0;0),\,P_2=(0;1;0)$\,and\, $P_1=(0;0;1)$ be the reference points, therefore, we
have the following four subcases:
\begin{description}
\item[i.] All reference points lie in $C,$
\item[ii.] Two reference points lie in $C,$
\item[iii.] One reference point lies in $C,$
\item[iv.] None of the reference points lie in $C.$
\end{description}
\par We will treat each of these subcases.
\begin{itemize}
\item If all reference points lie on $C$, then the possibilities for the defining equation are
    now:
 \[
C:\,\,\sum_{j=1}^{\lfloor\frac{d}{2}\rfloor}\,\,\big(X^{d-j}L_{j,X}+Y^{d-j}L_{j,Y}+Z^{d-j}L_{j,Z}\big).
\]
It is obvious that $B_i$ can not appear in $L_{1,B_j}$ whenever $j\neq i$, where $B_1:=X,\, B_2:=Y\,\,\, \text{and}\,\,\,
B_3:=Z.$ Hence, by the change of the variables $X,Y$ and $Z$\,(for instance, $L_{1,X}\leftrightarrow Y$,\,$L_{1,Y}\leftrightarrow Z$ and $L_{1,Z}\leftrightarrow X$), we can assume that
\begin{eqnarray*}
C&:& X^{d-1}Y+Y^{d-1}Z+\alpha Z^{d-1}X+\sum_{j=2}^{\lfloor\frac{d}{2}\rfloor}\,\,\big(X^{d-j}L_{j,X}+Y^{d-j}L_{j,Y}+Z^{d-j}L_{j,Z}\big).
\end{eqnarray*}
The first three factors implies that $a=(d-1)a+b=(d-1)b\,(mod\,m).$
In particular, $m|(d-2)^2+(d-1)$. Now, for each $m$, let $L_m$ be
the set $L_m:=\{(a,b)\in\Gamma_m:\,\,a=(d-1)a+b=(d-1)b\,(mod\,m)\}$.
Then, for any $(a,b)\in L_m$, by checking each monomial invariance,
we get type $m,\,(a,b)$ of the form $(3).$ Furthermore, this type is
$K$-equivalent to any type $m, (a',b')\in <m, (a,b)>$. So, to
complete the classification for a certain $m$, it suffices to choose
another $(a_0,b_0)\in L_m-<(a,b)>$ till $L_m=\phi$. This proves
$(3)$.

\item If two reference points lie in the smooth plane curve $C$, then by re-scaling the matrix $\varphi$
and permuting the coordinates, we can assume that $(1;0;0)\notin C.$ The equation is then
\[C: X^d+X^{d-2}L_{2,X}+X^{d-3}L_{3,X}+...+XL_{d-1,X}+L_{d,X}=0,\]
since $L_{1,X}$ is not invariant by $\varphi$ since, $ab\neq0$. Moreover, $Z^d$ and $Y^d$ are not in $L_{d,X},$
by the assumption that only $(1;0;0)\notin C.$\vspace{0.1cm}
\par Assume first that $Y^{d-1}Z$ and $YZ^{d-1}$ are in $L_{d,X}.$ Then, $(d-1)a+b\equiv0\,(mod\,m)$ and $a+(d-1)b\equiv0\,(mod\,m)$. In particular, $m|(d-1)^2-1$ and for each $m\in L_m$ where $L_m$ is defined as follows \[L_m:=\{(a,b)\in\Gamma_m:\,\,\,(d-1)a+b\equiv0\,(mod\,m),\,a+(d-1)b\equiv0\,(mod\,m)\}.\] If $(a,b)\in L_m$, then we obtain type $m,\,(a,b)$ of the form
\[
X^d+\big(\sum_{j=2}^{d-1}\,\,X^{d-j}\sum_{i\in S(2)^{j,X}_{m,(a,b)}}\beta_{ji}Y^iZ^{j-i}\big)
+\big(Y^{d-1}Z+\alpha YZ^{d-1}+\sum_{i\in S_2^{d,X}\,\,{m,(a,b)}}\beta_{di}Y^iZ^{d-i}\big)=0,
\]
where $\alpha\neq0$ and (4.1) follows.
\par Assume second that $Y^{d-1}Z\in L_{d,X}$ and $YZ^{d-1}\notin L_{d,X}.$ Then, by the non-singularity, $Z^{d-1}$ is in $L_{d-1,X}$.
That is, $(d-1)a+b\equiv0\, (mod\,m)$ and $(d-1)b\equiv0\, (mod\,m).$ Hence, $m|(d-1)^2$ and $(a,b)\in L_m$ \,\,\,\big($:=\{(a,b)\in\Gamma:\,\,\,(d-1)a+b\equiv0\, (mod\,m),\,(d-1)b\equiv0\, (mod\,m)\}\big),$
if we check monomials invariance for a certain $(a,b)$, we obtain the type $m,\,(a,b)$ of the form
\begin{eqnarray*}
X^d&+&\sum_{j=2}^{d-2}\,\,X^{d-j}\big(\sum_{i\in S(2)^{j,X}_{m,(a,b)}}\beta_{ji}Y^iZ^{j-i}\big)+X\big(\alpha Z^{d-1}+\sum_{i\in S_1^{d-1,X}\,\,{m,(a,b)}}\beta_{(d-1)i}Y^iZ^{d-1-i}\big)+\\
&+&\big(Y^{d-1}Z+\sum_{i\in S_2^{d,X}\,\,{m,(a,b)}}\beta_{di}Y^iZ^{d-i}\big)=0.
\end{eqnarray*}
That proves the case $4.2$.
\par Up to a permutation of $Y$ and $Z$, we can assume that $Y^{d-1}Z$ and $YZ^{d-1}$ are not in $L_{d,X}.$ By the non-singularity, $Z^{d-1}$ and $Y^{d-1}$ should be in $L_{d-1,X}$ consequently, $(d-1)b\equiv0\,(mod\,m)$ and $(d-1)a\equiv0\,(mod\,m).$
Therefore, $m|(d-1)$ and we get type $m,\,(a,b)$ of (4.3) defined by the equation
\begin{eqnarray*}
X^d&+&\sum_{j=2}^{d-2}\,\,\big(X^{d-j}\sum_{i\in S(2)^{j,X}_{m,(a,b)}}\beta_{ji}Y^iZ^{j-i}\big)+\sum_{i\in S_2^{d,X}\,\,{m,\,(a,b)}}\beta_{di}Y^iZ^{d-i}+\\
&+&X\big(Z^{d-1}+\alpha Y^{d-1}+\sum_{i\in S_2^{d-1,X}\,\,{m,\,(a,b)}}\beta_{(d-1)i}Y^iZ^{d-1-i}\big)=0,
\end{eqnarray*}
\item If one reference points lie in the $C$, then by normalizing the matrix $\varphi$ and
permuting the coordinates, we can assume that $(1;0;0),\,(0;1;0)\notin C.$ We then write
\[C: X^d+Y^d+X^{d-2}L_{2,X}+X^(d-3)L_{3,X}+...+XL_{d-1,X}+L_{d,X}=0,\]
such that $Z^d\notin L_{d,X}.$ Also, by the non-singularity, we have that $Z^{d-1}\in L_{d-1,X},$ then $da\equiv0\,(mod\,m)$ and
$(d-1)b\equiv0\,(mod\,m).$ Hence, $m|\,d(d-1)$ define $L_m$ to be $\{(a,b)\in\Gamma:\,\,\,da\equiv0\,(mod\,m),\,\, (d-1)b\equiv0\,(mod\,m)\}.$
Consequently, for each element in this set we obtain type $m,\,(a,b)$ of the form
\begin{eqnarray*}
C:\,\,X^d&+&Y^d+\sum_{j=2}^{d-2}\,\,\big(X^{d-j}\sum_{i\in S(2)^{j,X}_{m,(a,b)}}\beta_{ji}Y^iZ^{j-i}\big)+\sum_{i\in S_1^{d,X}\,\,{m,\,(a,b)}}\beta_{di}Y^iZ^{d-i}+\\
&+&X\big(\alpha Z^{d-1}+\sum_{i\in S_1^{d-1,X}\,\,{m,\,(a,b)}}\beta_{(d-1)i}Y^iZ^{d-1-i}\big)=0
\end{eqnarray*}
which produces $(5)$

\item If none of the reference points lie in the smooth plane curve $C$, then
\[C: X^d+Y^d+Z^d+\big(\sum_{j=2}^{d-1}X^{d-j}L_{j,X}\big)+L_{d,X}=0,\]
where $L_{1,X}$ does not appear since $ab\neq0.$ Clearly,
$da=db=0\,(mod\,m)$, from which $m|d$ for each $m$, and each
$(a,b)\in
L_m\,\,\big(:=\{(a,b)\in\Gamma_m:\,\,\,da\equiv0\,(mod\,m),\,\,
db\equiv0\,(mod\,m)\}\big)$, we obtain type $m,\,(a,b)$ of the
following form
\[
C:\,\,X^d+Y^d+Z^d+\sum_{j=2}^{d-1}\,\,\big(X^{d-j}\sum_{i\in S(2)^{j,X}_{m,(a,b)}}\beta_{ji}Y^iZ^{j-i}\big)
+\sum_{i\in S_1^{d,X}\,\,{m,\,(a,b)}}\beta_{di}Y^iZ^{d-i}=0.
\]
This completes the proof of the main result.
\end{itemize}
\end{proof}

\begin{cor}
Let $G$ be a cyclic subgroup of $Aut(C)$ where $C$ is a non-singular
plane curve of degree $d\geq4$. Then, $|G|$ must divide one of the
following
$d-1,\,\,d,\,\,d^2-3d+3,\,\,(d-1)^2,\,\,d(d-2),\,\,d(d-1)$. In
particular, $|G|\leq d(d-1)$.
\end{cor}


\section{Characterization of curves where $Aut(C)$ has a cyclic
subgroup of large order}

We study next non-singular plane curves $C$ (up to $K$-isomorphism)
that admits a $\sigma\in Aut(C)$ of order
$d^2-3d+3,\,\,(d-1)^2,\,\,d(d-2),$ $\,d(d-1)$, $\ell(d-1),$ or $\ell
d,$ with $\ell\geq 2$. In particular we are interested in determine
the full automorphism group of the curve and equations of such
curves.

Before a detailed study we recall the following results concerning
$Aut(C)$ which will be useful. Because linear systems $g_2$ are
unique, we always assume $C$ is given by a plane equation
$F(X,Y,Z)=0$ and $Aut(C)$ is a finite subgroup of $PGL_3(K)$ which
fix the equation $F$. Moreover $Aut(C)$ satisfies (see Mitchel
\cite{Mit}) one of the following situations:
\begin{enumerate}
\item fixes a point $P$ and a line $L$ with $P\notin L$ in
$PGL_3(K)$,
\item fixes a triangle, i.e. exists 3 points $S:=\{P_1,P_2,P_3\}$ of
$PGL_3(K)$, such that is fixed as a set,
\item $Aut(C)$ is conjugate of a representation inside $PGL_3(K)$ of
one of the finite primitive group namely, the Klein group
$PSL(2,7)$, the icosahedral group $A_5$, the alternating group
$A_6$, the Hessian groups $Hess_{216}$, $Hess_{72}$ or $Hess_{36}$.
\end{enumerate}

It is classically known that if $G$ a subgroup of automorphisms of a
non-singular plane curve $C$ fixes a point on $C$ then $G$ is cyclic
\cite[Lemma 11.44]{Book}, and recently Harui in \cite[\S2]{Harui}
provided the lacked result in the literature on the type of groups
that could appear in non-singular plane curves. Before introduce
Harui statement we need to define descendent of a plane non-singular
curve. For a non-zero monomial $cX^iY^jZ^k$, $c\in K\setminus\{0\}$,
we define its exponent as $max\{i,j,k\}$. For a homogeneous
polynomial $F$, the core of $F$ is defined as the sum of all terms
of $F$ with the greatest exponent. Let $C_0$ be a smooth plane
curve, a pair $(C,G)$ with $G\leq Aut(C)$ is said to be a descendant
of $C_0$ if $C$ is defined by a homogeneous polynomial whose core is
a defining polynomial of $C_0$ and $G$ acts on $C_0$ under a
suitable coordinate system.

\begin{thm}[Harui] If
$G\preceq\,\,Aut(C)$ where $C$ is a non-singular plane curve of
degree $d\geq4$ then $G$ satisfies one of the following
\begin{enumerate}
  \item $G$ fixes a point on $C$ and then cyclic.
  \item $G$ fixes a point not lying on $C$ and it satisfies a short exact sequence of the form
  $$1\rightarrow N\rightarrow Aut(C)\rightarrow G'\rightarrow 1,$$
with $N$ a cyclic group of order dividing $d$ and $G'$ is conjugate
to a cyclic group $C_m$,  a Diedral group $D_{2m}$, the alternating
groups $A_4$ , $A_5$ or the permutation group $S_4$, where $m$ is an
integer $\leq d-1$. Moreover, if $G'\cong D_{2m}$, then $m|(d-2)$ or
$N$ is trivial.
\item $G$ is conjugate to a subgroup of $Aut(F_d)$ where $F_d$ is the Fermat curve $X^d+Y^d+Z^d$. In particular, $|G|\,|\,6d^2$ and $(G,C)$
is a descendant of $F_d$.
\item $G$ is conjugate to a subgroup of $Aut(K_d)$ where $K_d$ is the Klein curve curve $XY^{d-1}+YZ^{d-1}+ZX^{d-1}$
hence $|G|\,|\,3(d^2-3d+3)$ and $(G,C)$ is a descendant of $K_d$.
\item $G$ is conjugate a finite primitive subgroup of $PGL_3(K)$,
i.e. the Klein group $PSL(2,7)$, the icosahedral group $A_5$, the alternating group $A_6$, the Hessian groups
$Hess_{216}$, $Hess_{72}$ or $Hess_{36}$ inside $PGL_3(K)$.

\end{enumerate}
where
\end{thm}

Next assume as usual $C$ a non-singular plane curve of degree $d\geq
4$ with $\sigma\in Aut(C)$ of exact order $m$ such that acts on
$F(X,Y,Z)=0$ by $(x,y,z)\mapsto (x,\xi_m^a y,\xi_m^bz)$.

%
%

\subsection{Curves with a cyclic automorphism of order $d(d-1)$.}
\mbox{} \\
The following result appears in Harui \cite[\S3]{Harui}.
\begin{prop}[Harui]\label{prop111}
A non-singular projective plane curve $C$ of degree $d\geq 5$ with
$Aut(C)$ cyclic group of order $d(d-1)$ if and only if it is
projectively equivalent to $X^d+Y^d+XZ^{d-1}$.
\end{prop}

%

\begin{prop}\label{prop11}
Let $C$ be a non-singular projective plane curve of degree $d\geq
4$. Then, $Aut(C)$ contains an element of order $d(d-1)$ if and only
if $C$ is projectively equivalent to $\,X^d+Y^d+\alpha XZ^{d-1}$
where $\alpha\neq0$. In particular, $Aut(C)$ is cyclic of order
$d(d-1)$ for $d\geq 5$.
\end{prop}
\begin{rem}\label{rem3.2} Recall that for $d=4$ the automorphism group of $\,\,X^d+Y^d+\alpha
XZ^{d-1}$ is $C_4\circledcirc A_4$ the element of $Ext^1(A_4,C_4)$
given by $\{(\delta,g)\in \mu_{12}\times H:\delta^4=\chi(g)\}/{\pm
1},$ where $\mu_n$ is the set of n-th roots of unity, $H$ is the
group $A_4$ and let take $S,T$ a generators of $H$ of order 2 and 3
respectively with the representation $H=<S,T|S^2=1,T^3=1,...>$ and
$\chi$ is the character $\chi:H\rightarrow \mu_3$ defined by
$\chi(S)=1$ and $\chi(T)=\rho$ with $\rho$ a fixed 3-primitive root
of unity, see \cite{He} (or also \cite{Bars}).
\end{rem}

\begin{proof}
$(\leftarrow)$ Indeed, $[X;\zeta_{d(d-1)}^{d-1}Y;\zeta_{d(d-1)}^{d}Z]$ is an element of $Aut(C)$ of order $d(d-1).$\\
$(\rightarrow)$ It is clear that $d(d-1)$ does not divide any of the
integers $d-1,\,\,d,\,\,d^2-3d+3,\,\,d(d-2),\,\,(d-1)^2$ therefore,
by Theorem \ref{thm20} $(5)$, $C$ is projectively equivalent to a
type $d(d-1),\,\,(a,b)$ of the following form
\begin{eqnarray*}
\,X^d&+&Y^d+\sum_{j=2}^{d-2}\,\,\big(X^{d-j}\sum_{i\in S(2)^{j,X}_{m,(a,b)}}\beta_{ji}Y^iZ^{j-i}\big)+\sum_{i\in S_1^{d,X}\,\,{m,\,(a,b)}}\beta_{di}Y^iZ^{d-i}+\\
&+&X\big(\alpha Z^{d-1}+\sum_{i\in S_1^{d-1,X}\,\,{m,\,(a,b)}}\beta_{(d-1)i}Y^iZ^{d-1-i}\big)=0
\end{eqnarray*}
where $(a,b)\in\Gamma_{d(d-1)}$ such that $da\equiv0\,\,mod\,d(d-1)$ and $(d-1)b\equiv0\,\,mod\,d(d-1)$.\\\\
Claim: $d(d-1),\,\,\,(d-1,d)$ is a generator of the set of solutions\\
Indeed, every solution is of the form $a=(d-1)k$ and $b=dk'$ and we have
\[
[X;\zeta_{d(d-1)}^{d-1}Y;\zeta_{d(d-1)}^{d}Z]^{d(k'-k)+k}=[X;\zeta_{d(d-1)}^{(d-1)k}Y;\zeta_{d(d-1)}^{dk'}Z]
\]
Claim: $S(2)^{j,X}_{m,(a,b)}=\phi$ for all $j=2,3,...,d-2$.\\
\begin{eqnarray*}
S(2)^{j,X}_{m,(a,b)}&:=&\{i:\,\,0\leq i\leq j\,\,\text{and}\,\,(d-1)i+(j-i)d=0\,mod\,\,d(d-1)\}\\
&=&\{i:\,\,0\leq i\leq j\,\,\text{and}\,\,d(d-1)\,\,|(dj-i)\}\\
&=&\phi\,\,\,\,\forall j=2,...,d-2.
\end{eqnarray*}
because $0<dj-i\leq d(d-2)<d(d-1)\,\,\forall j=2,...,d-2.$\\\\
Claim: $S_1^{d-1,X}\,\,{m, (a,b)}=\phi$.
\begin{eqnarray*}
S_1^{d-1,X}\,\,{m,(a,b)}&:=&\{i:\,\,1\leq i\leq d-1\,\,\text{and}\,\,(d-1)i+(d-1-i)d=0\,mod\,\,d(d-1)\}\\
&=&\{i:\,\,1\leq i\leq d-1\,\,\text{and}\,\,d(d-1)\,\,|i\}\\
&=&\phi.
\end{eqnarray*}
Claim: $S_1^{d,X}\,\,{m,(a,b)}=\phi.$
\begin{eqnarray*}
S_1^{d,X}\,\,{m,(a,b)}&:=&\{i:\,\,1\leq i\leq d-1\,\,\text{and}\,\,(d-1)i+(d-i)d=0\,mod\,\,d(d-1)\}\\
&=&\{i:\,\,1\leq i\leq d-1\,\,\text{and}\,\,d(d-1)\,\,|d-i\}\\
&=&\phi.
\end{eqnarray*}
because $0<d-i\leq d-1<d(d-1)$.\\
Consequently, $C$ is projectively equivalent to $X^d+Y^d+\alpha
XZ^{d-1}$ where $\alpha\neq0$. The last part for which $Aut(C)$ is
cyclic of order $d(d-1)$ is followed by Harui \ref{prop111}.
\end{proof}

\subsection{Plane curves with a cyclic automorphism of order $(d-1)^2$}
\begin{prop}\label{prop12}
The automorphism group of the non-singular projective plane curve of
degree $d\geq 5$ defined by \[C: \,\,X^d+Y^{d-1}Z+\alpha XZ^{d-1}\]
such that $\alpha\neq0$ is cyclic of order $(d-1)^2$.
\end{prop}
\begin{rem}\label{rem4} The same result holds for $d=4$, see \cite{He} (or
\cite{Bars}).
\end{rem}

\begin{proof}
We have $Aut(C)$ contains a homology of order $d-1\geq 4$ namely, $[X;\zeta_{d-1}Y;Z]$ therefore (see Mitchell \cite{Mit} $\S 5$), $Aut(C)$ should fix a point, a line or a triangle. Moreover, $(d-1)^2\,\,|\,\,|Aut(C)|$ because the transformation $[X;\zeta_{(d-1)^2}Y;\zeta_{(d-1)^2}^{(d-1)(d-2)}Z]\in Aut(C)$ is of order $(d-1)^2$.
\begin{enumerate}
  \item If $Aut(C)$ fixes a triangle and neither a point nor a line is fixed then, it follows by Harui $\S 4$, that $C$ is either a descendent of the Fermat curve $F_d$ or the Klein curve $K_d$. But, non of them contains an element of order $(d-1)^2$ because $Aut(F_d)$ (resp.\,\,$Aut(K_d)$) has elements of order at most $2d$ (resp. $d^2-3d+3$) $<(d-1)^2$.
  \item Assume that $Aut(C)$ fixes a point not lying on $C$ and satisfies a short exact sequence of the form:
$$1\rightarrow N\rightarrow Aut(C)\rightarrow G'\rightarrow 1,$$
with $N$ a cyclic group of order dividing $d$ and $G'$ is conjugate to a cyclic group $C_m$,  a Diedral group $D_{2m}$, $A_4$ , $A_5$ or $S_4$, where $m$ is an integer $\leq d-1$. Moreover, if $G'\cong
D_{2m}$, then $m|(d-2)$ or $N$ is trivial. Since $|N|$ and $(d-1)^2$ are coprime, then $G'$ should contain an element of order $(d-1)^2$ a contradiction because each of these groups have elements of order at most $max\,\,\{5,d-1\}<(d-1)^2$.
\end{enumerate}
Consequently, $Aut(C)$ fixes a point on $C$, thus it is cyclic of order divisible by $(d-1)^2$ and $\leq d(d-1)$. Therefore, $|Aut(C)|=(d-1)^2$.

\end{proof}

As analogue of Proposition \ref{prop11} we prove the following result.

\begin{prop}\label{prop13}
A non-singular projective plane curve $C$ of degree $d\geq 4$ has an
automorphism of order $(d-1)^2$ if and only if it is isomorphic to
$\,\,X^d+Y^{d-1}Z+\alpha XZ^{d-1}$ such that $\alpha\neq0$. In
particular, $Aut(C)$ is cyclic of order $(d-1)^2$ moreover, if $G$
is a non-cyclic automorphism group with $(d-1)^2$ dividing $|G|$
then $G$ does not contain any element of this order.
\end{prop}

\begin{proof}
It suffices to prove the first part because the other part follows directly by Proposition \ref{prop12} and the remark \ref{rem4}.\\
$(\leftarrow)$\,\, Clearly, $[X;\zeta_{(d-1)^2}Y;\zeta_{(d-1)^2}^{-(d-1)}Z]$ is an automorphism that has order $(d-1)^2$.\\
$(\rightarrow)$\,\, One can easily check that $(d-1)^2$ does not
divide any of the integers
$d-1,\,\,d,\,\,d^2-3d+3,\,\,d(d-2),\,\,d(d-1)$. Therefore, by
Theorem \ref{thm20}, $C$ is isomorphic to type $(d-1)^2, (a,b)$ of
the form
\begin{eqnarray*}
X^d&+&\sum_{j=2}^{d-2}\,\,X^{d-j}\big(\sum_{i\in S(2)^{j,X}_{(d-1)^2,(a,b)}}\beta_{ji}Y^iZ^{j-i}\big)+X\big(\alpha Z^{d-1}+\sum_{i\in S_1^{d-1,X}\,\,{(d-1)^2,(a,b)}}\beta_{(d-1)i}Y^iZ^{d-1-i}\big)+\\
&+&\big(Y^{d-1}Z+\sum_{i\in S_2^{d,X}\,\,{(d-1)^2,(a,b)}}\beta_{di}Y^iZ^{d-i}\big)=0
\end{eqnarray*}
for some $(a,b)\in\Gamma_{(d-1)^2}$ such that $(d-1)a+b\equiv0\, mod\,(d-1)^2$ and $(d-1)b\equiv0\, mod\,(d-1)^2$. Every solution is of the form $a=(d-1)k-k'$ and $b=(d-1)k'$ then we can take a generator $a=1$ and $b=(d-1)(d-2)$ because $[X;\zeta_{(d-1)^2}Y;\zeta_{(d-1)^2}^{(d-1)(d-2)}Z]^{(d-1)k-k'}=[X;\zeta_{(d-1)^2}^{(d-1)k-k'}Y;\zeta_{(d-1)^2}^{(d-1)k'}Z]$. Now,
\begin{eqnarray*}
S(2)^{j,X}_{m,(a,b)}&:=&\{i:\,\,0\leq i\leq j\,\,\text{and}\,\,i+(j-i)(d-1)(d-2)=0\,mod\,\,(d-1)^2\}\\
&=&\{i:\,\,0\leq i\leq j\,\,\text{and}\,\,(d-1)^2\,\,|i-(j-i)(d-1)\}\\
&=&\{i:\,\,0\leq i\leq j\,\,\text{and}\,\,(d-1)^2\,\,|j(d-1)-di\}\\
&=&\phi\,\,\,\,\forall j=2,...,d-2.
\end{eqnarray*}
because if $(d-1)^2\,\,|j(d-1)-di$ then $d|j+1$ a contradiction since $0<j+1<d$.

\begin{eqnarray*}
S_1^{d-1,X}\,\,{m,(a,b)}&:=&\{i:\,\,1\leq i\leq d-1\,\,\text{and}\,\,i+(d-1-i)(d-1)(d-2)=0\,mod\,\,(d-1)^2\}\\
&=&\{i:\,\,1\leq i\leq d-1\,\,\text{and}\,\,(d-1)^2\,\,|di\}\\
&=&\{i:\,\,1\leq i\leq d-1\,\,\text{and}\,\,(d-1)^2\,\,|i\}\\
&=&\phi.
\end{eqnarray*}
because $0<i<(d-1)^2$ so $(d-1)^2\,\,\nmid i$.

\begin{eqnarray*}
S_2^{d,X}\,\,{m,(a,b)}&:=&\{i:\,\,2\leq i\leq d-2\,\,\text{and}\,\,i+(d-i)(d-1)(d-2)=0\,mod\,\,(d-1)^2\}\\
&\subseteq&\{i:\,\,2\leq i\leq d-2\,\,\text{and}\,\,(d-1)^2\,\,|di-(d-1)\}\\
&=&\phi.
\end{eqnarray*}
because $0<di-(d-1)\leq d^2-3d+1<(d-1)^2$. Thus, $C$ is isomorphic
to $\,\,X^d+\alpha XZ^{d-1}+Y^{d-1}Z=0$.
\end{proof}

\subsection{Plane curves with an automorphism of order $d(d-2)$}

\begin{prop}\label{prop15}
The full automorphism group of $C:\,\,X^d+Y^{d-1}Z+\alpha YZ^{d-1}$ of degree $d\geq4$ such that  $\alpha\neq0$ is classified as follows.
\begin{enumerate}
  \item if $d\neq4,6$, then it is the central extension \[<\sigma,\,\tau|\,\,\,\sigma^2=\tau^{d(d-2)}=1\,\,\,and\,\,\,\sigma\tau\sigma=\tau^{-(d-1)}>\] of $D_{2(d-2)}$ by $\mathbb{Z}_d$. In particular, $|Aut(C)|=2d(d-2)$.
  \item if $d=6$, then it is a central extension of $S_4$ by $\mathbb{Z}_6$. Thus, $|Aut(C)|=144$.
  \item if $d=4$, then $C$ is isomorphic to the Fermat quartic $F_4$ thus  $Aut(C)\simeq \mathbb{Z}_4^2\rtimes S_3$.
\end{enumerate}

\end{prop}

\begin{proof}
Let $\mu\in K$ such that $\mu^{d(d-2)}=\alpha$, then $C$ is
projectively equivalent, through the transformation $[X;\mu
Y;\mu^{-(d-1)}Z]$, to the curve $C':\,\,\,X^d+Y^{d-1}Z+YZ^{d-1}$ and
therefore by Harui \cite{Harui} $\S 6$, $Aut(C)$ is a central
extension of $D_{2(d-2)}$ by $\mathbb{Z}_d$ ($d\neq4,6$), a central
extension of $S_4$ by $\mathbb{Z}_d$ ($d=6$) or $\simeq
\mathbb{Z}_4^2\rtimes S_3$ ($d=4$). To prove the remaining part, it
suffices to notice that $Aut(C')$ is generated by $\sigma:=[X;Z;Y]$
and $\tau:=[X;\zeta_{d(d-2)}Y;\zeta_{d(d-2)}^{-(d-1)}Z]$.
\end{proof}

As an analogue of Propositions \ref{prop11} and \ref{prop13}, we have the following result.
\begin{prop}\label{prop17}
A non-singular projective plane curve $C$ of degree $d\geq 4$ admits
an automorphism of order $d(d-2)$ if and only if it is isomorphic to
$\,\,X^d+Y^{d-1}Z+YZ^{d-1}$. In particular, $Aut(C)$ is a central
extension of $D_{2(d-2)}$ by $\mathbb{Z}/d$ ($d\neq4,6$), a central
extension of $S_4$ by $\mathbb{Z}/d$ ($d=6$) or $\simeq
(\mathbb{Z}/4)^2\rtimes S_3$ ($d=4$).
\end{prop}

\begin{proof}
It suffices to prove the first part and the second part is an immediate consequence of Proposition \ref{prop15}. \\
$(\leftarrow)$ Trivial, since $[X;\zeta_{d(d-2)}Y;\zeta_{d(d-2)}^{-(d-1)}Z]$ is an automorphism of order $d(d-2)$.\\
$(\rightarrow)$ Clearly, $d(d-2)$ does not divide any of the
integers $d-1,\,\,d,\,\,d^2-3d+3,\,\,(d-1)^2,\,\,d(d-1)$. Therefore,
by Theorem \ref{thm20}, $C$ is isomorphic to type $d(d-2), (a,b)$ of
the form
\[
X^d+\big(\sum_{j=2}^{d-1}\,\,X^{d-j}\sum_{i\in S(2)^{j,X}_{m,(a,b)}}\beta_{ji}Y^iZ^{j-i}\big)
+\big(Y^{d-1}Z+\alpha YZ^{d-1}+\sum_{i\in S_2^{d,X}\,\,{m,(a,b)}}\beta_{di}Y^iZ^{d-i}\big)=0,
\]
fore some $(a,b)\in\Gamma_{d(d-2)}$ such that $(d-1)a+b\equiv0\,\,mod\,d(d-2)$ and $a+(d-1)b\equiv0\,\,mod\,d(d-2).$ Every solution is of the form $(k,dk'+k)$ such that $k$ and $dk'+k$ are coprime and $d-2|k+k',\,\,k+(d-1)k'$. We can take a generator $k=1$ and $k'=d-3$ because $[X;\zeta_{d(d-2)}Y;\zeta_{d(d-2)}^{d(d-3)+1}Z]^k=[X;\zeta_{d(d-2)}^kY;\zeta_{d(d-2)}^{dk'+k}Z]$ (keep in mind that $d-2|k+k'$). Therefore, we have
\begin{eqnarray*}
S(2)^{j,X}_{m,(a,b)}&:=&\{i:\,\,0\leq i\leq j\,\,\text{and}\,\,i+(j-i)\big(d(d-3)+1\big)=0\,mod\,\,d(d-2)\}\\
&=&\{i:\,\,0\leq i\leq j\,\,\text{and}\,\,d(d-2)\,\,|j(d-1)-di\}\\
&=&\phi\,\,\,\,\forall j=2,...,d-2.
\end{eqnarray*}
because if $d(d-2)\,\,|j(d-1)-di$ then $d|j$ a contradiction since $0<j<d$.
\begin{eqnarray*}
S_2^{d,X}\,\,{m,(a,b)}&:=&\{i:\,\,2\leq i\leq d-2\,\,\text{and}\,\,i+(d-i)\big(d(d-3)+1\big)=0\,mod\,\,d(d-2)\}\\
&\subseteq&\{i:\,\,2\leq i\leq d-2\,\,\text{and}\,\,d-2\,\,|d-1-i\}\\
&=&\phi.
\end{eqnarray*}
Thus, $C$ is isomorphic to $\,\,X^d+Y^{d-1}Z+\alpha YZ^{d-1}$ such
that $\alpha\neq0$.
\end{proof}

\subsection{Plane curves with an automorphism of order $d^2-3d+3$.}

\begin{prop} \label{prop16}
The full automorphism group of $C:\,\,X^{d-1}Y+Y^{d-1}Z+\alpha Z^{d-1}X$ of degree $d\geq5$ such that
 $\alpha\neq0$ the semidirect product of $\mathbb{Z}/3$ by $\mathbb{Z}/{d^2-3d+3}$
given by
\[<\tau,\,\sigma|\,\,\tau^{d^2-3d+3}=\sigma^3=1\,\,\,and\,\,\,\tau\sigma=\sigma\tau^{-(d-1)}\ >.\] In particular, $|Aut(C)|=3(d^2-3d+3)$.
\end{prop}

\begin{proof}
Indeed, $C$ can be transformed to the Klein curve $K_d$ through the
transformation $[X;\mu Y;\mu^{-(d-2)}Z]$ where $\mu\in K$ is defined
by the equation $\alpha=\mu^{d^2-3d+3}$. Now, it follows by Harui
\cite{Harui} $\S 3$ that $Aut(C)$ is a semidirect product of
$\mathbb{Z}/3$ acting on $\mathbb{Z}/{d^2-3d+3}$. To prove the last
part, it suffices to notice that
$\tau:=[X;\zeta_{d^2-3d+3}Y;\zeta_{d^2-3d+3}^{-(d-2)}Z]$ and
$[Z;X;Y]$ are in $Aut(K_d)$.
\end{proof}
\begin{rem} For $d=4$ recall that the Klein curve $K_4$ the order of
the automorphism group is 168, isomorphic to $PSL_2(\mathbb{F}_7)$
\cite{He}.
\end{rem}

\begin{prop}\label{prop14}
Any non-singular projective plane curve $C$ of degree $d\geq 4$
which has an automorphism of order $d^2-3d+3$ is isomorphic to the
Klein curve $K_d:\,\,X^{d-1}Y+Y^{d-1}Z+Z^{d-1}X$. In particular, for
$d\geq 5$, $Aut(C)$ is isomorphic to
\[<\tau,\,\sigma|\,\,\tau^{d^2-3d+3}=\sigma^3=1\,\,\,and\,\,\,\tau\sigma=\sigma\tau^{-(d-1)}\>\].
\end{prop}

\begin{proof}
Since $d^2-3d+3\nmid\,\,d-1,\,\,d,\,\,d(d-1),\,\,d(d-2),\,\,(d-1)^2$
for every $d\geq5$ then, $C$ is projectively equivalent to a plane
curve of type $d^2-3d+3,\,\,(a,b)$ of the form
 \begin{eqnarray*}
X^{d-1}Y&+&Y^{d-1}Z+\alpha Z^{d-1}X+\\
&+&\sum_{j=2}^{\lfloor\frac{d}{2}\rfloor}\,\,X^{d-j}\big(\sum_{i\in S(1)^{j,X}_{m,(a,b)}}\beta_{ji}Y^iZ^{j-i}\big)+Y^{d-j}\big(\sum_{i\in S^{j,Y}_{m,(a,b)}}\alpha_{ji}Z^iX^{j-i}\big)+Z^{d-j}\big(\sum_{i\in S^{j,Z}_{m,(a,b)}}\gamma_{ji}X^{j-i}Y^i\big),
\end{eqnarray*}
for some $(a,b)\in\Gamma_{d^2-3d+3}$ such that $a=(d-1)a+b=(d-1)b\,(mod\,\,d^2-3d+3)$. We can take a generator $a=1$ and $b=d^2-4d+5$ because every solution is of the form $(k,\,(d^2-3d+3)k'-(d-2)k)$ and we have $[X;\zeta_{d^2-3d+3}Y;\zeta_{d^2-3d+3}^{d^2-4d+5}]^{k}=[X;\zeta_{d^2-3d+3}^kY;\zeta_{d^2-3d+3}^{(d^2-3d+3)k'-(d-2)k}]$. Consequently, \begin{eqnarray*}
S(1)^{j,X}_{m,(a,b)}&:=&\{i:\,\,0\leq i\leq j\,\,\text{and}\,\,i+(j-i)(d^2-4d+5)=1\,mod\,\,(d^2-3d+3)\}\\
&=&\{i:\,\,0\leq i\leq j\,\,\text{and}\,\,(d^2-3d+3)\,\,|\big(j(d-2)-i(d-1)+1\big)\}\\
&=&\phi\,\,\,\,\forall j=2,...,\lfloor\frac{d}{2}\rfloor.
\end{eqnarray*}
because $\underbrace{-(d^2-3d+3)} <-(\frac{d}{2}-1)\leq-j+1\leq \underbrace{j(d-2)-i(d-1)+1}\leq j(d-2)+1\leq\frac{d(d-2)}{2}+1<\underbrace{d^2-3d+3}$ $\forall j=2,...,\lfloor\frac{d}{2}\rfloor.$ Consequently, the only possibility is that $j(d-2)-i(d-1)+1=0$ which in turns implies that $d|2j-i-1$ a contradiction since $0<2j-i-1<d$.
\begin{eqnarray*}
S^{j,Y}_{m,(a,b)}&:=&\{i:\,\,0\leq i\leq j\,\,\text{and}\,\,(d^2-4d+5)i+(d-j)=1\,mod\,\,(d^2-3d+3)\}\\
&=&\{i:\,\,0\leq i\leq j\,\,\text{and}\,\,(d^2-3d+3)\,\,|\big((d-j)-(d-2)i-1\big)\}\\
&=&\phi.
\end{eqnarray*}
because $\underbrace{-(d^2-3d+3)}<-\frac{d^2-3d+2}{2}\leq\underbrace{(d-j)-(d-2)i-1}\leq (d-2)-1=d-3<\underbrace{d^2-3d+3}$ $\forall j=2,...,\lfloor\frac{d}{2}\rfloor.$ Therefore, the only possibility is that $(d-j)-(d-2)i-1=0$ in particular $d-2$ divides $j-1$ a contradiction since $0<j-1<d-2$.
\begin{eqnarray*}
S^{j,Z}_{m,(a,b)}&:=&\{i:\,\,0\leq i\leq j\,\,\text{and}\,\,i+(d-j)(d^2-4d+5)=1\,mod\,\,(d^2-3d+3)\}\\
&=&\{i:\,\,0\leq i\leq j\,\,\text{and}\,\,(d^2-3d+3)\,\,|\big((d-j)(d-2)-i+1\big)\}\\
&=&\phi.
\end{eqnarray*}
because $\underbrace{0}<\frac{d(d-3)}{2}+1\leq\underbrace{(d-j)(d-2)-i+1}\leq (d-2)^2+1<\underbrace{d^2-3d+3}$ \\\\
Consequently, $C$ is isomorphic to $X^{d-1}Y+Y^{d-1}Z+\alpha
Z^{d-1}X$ such that $\alpha\neq0$. The second part of the result
follows by Proposition \ref{prop16}.

\end{proof}


\subsection{Plane curves with automorphisms of orders $\ell d$ and $\ell
(d-1)$}\mbox{}\\

 We are interested in non-singular plane curves $C$ of
arbitrary but fixed degree $d\geq5$ whose automorphism groups
contain homologies $\sigma$ of period $d$ (resp.\,\,$d-1$), recall
that an homology is an automorphism that under a change of variables
is of type $m, (a,b)$ with $ab=0$. In these situations, the genus of
$C/<\sigma>$ is zero and $C$ has an unique outer (resp.\,inner)
Galois point $P$ if $d\geq 5$ (see \cite[Lemma 3.7]{Harui} and
\cite{Yoshihara} for the uniqueness of the outer (resp.\, inner)
Galois point when $d\geq 5$), and we refer to \cite{Yoshihara} for
the definition of inner or outer Galois point. Moreover, for a given
non-singular plane curve $C$ of degree $d\geq 5$, if exists $P$ an
outer (resp. inner) Galois point, because it is unique, one deduces
that $P$ is fixed by the full group $Aut(C)$, otherwise $\tau(P)$
will give another outer point for $\tau\in Aut(C)$ (similarly for
the inner Galois point).

\subsubsection{Plane curves with automorphism of orders $\ell(d-1)$.}

\begin{lem}
If a non-singular plane curve $C$ of degree $d\geq5$ has an
automorphism of order $\ell(d-1)$ ($2\leq\ell\leq d$) then $\ell$
should be a divisor of $d$ or $d-1$. In particular,
$d\equiv0\,\,(mod\,\,\ell)$ or $d\equiv1\,\,(mod\,\,\ell)$.

\end{lem}

\begin{proof}
Clear, because $\ell(d-1)$ does not divide
$d-1,\,\,d,\,\,d^2-3d+3,\,\,d(d-2)$.
\end{proof}

\begin{prop}
Let $d\geq5$ where $d\equiv0\,\,(mod\,\,\ell)$ and consider the
non-singular plane curve
\[
C:\,\,X^d+Y^d+\alpha XZ^{d-1}+\sum_{2\leq j=\ell k\leq
d-2}\,\,\beta_{ji}X^{d-j}Y^j
\]
Then, $Aut(C)$ is cyclic of order divisible by $\ell(d-1)$.
\end{prop}

\begin{proof}
Since $\sigma:=[X;\zeta_{\ell (d-1)}^{d-1}Y;\zeta_{\ell
(d-1)}^{\ell}Z]\in Aut(C_3)$ of order $\ell (d-1)$ then, $C_3$ is
not a descendant of the Klein curve $K_d$ (because
$\ell(d-1)\nmid3(d^2-3d+3)$) and is not a descendant of the Fermat
curve $F_d$ (because $\ell(d-1)>2d$ for all $\ell\geq3$ and for
$\ell=2$ we have $2(d-1)\nmid6d^2$). Moreover, $Aut(C_3)$ contains a
homology of period $d-1\geq4$ namely, $[X;Y;\zeta_{\ell
(d-1)}^{\ell^2}Z]=[X;\zeta_{\ell (d-1)}Y;\zeta_{\ell
(d-1)}^{(\ell-1)(d-1)}Z]^{\ell}$ therefore, it must fix a line and a
point off that line. Now, since $[X;\zeta_{\ell (d-1)}^{\ell}Y;Z]\in
Aut(C_3)$ is a homology of period $d-1$ with center $(0;0;1)$ then
by Harui \cite{Harui} $\S 3$, the point $(0;0;1)$ is an inner Galois
point of $C_3$. Moreover, it is the unique such point by Yoshihara
\cite{Yoshihara} $\S 2$ [Theorem $4'$] hence should be fixed by
$Aut(C_3)$. Consequently, $Aut(C_3)$ is cyclic of order divisible by
$\ell(d-1)$.
\end{proof}

\begin{prop}\label{prop121}
Let $C$ be a non-singular plane curve of degree $d\geq5$ where
$d\equiv0\,\,(mod\,\,\ell)$. Assume that $Aut(C)$ contains an
element of order $\ell (d-1)$ with $\ell\geq 2$. Then $C$ is
isomorphic to
\[
\,\,X^d+Y^d+\alpha XZ^{d-1}+\sum_{2\leq j=\ell k\leq
d-2}\,\,\beta_{ji}X^{d-j}Y^j
\]
In particular, $Aut(C)$ is cyclic.
\end{prop}

\begin{proof}
Clearly, $\ell (d-1)$ is not a divisor of
$d-1,\,\,d,\,\,d^2-3d+3,\,\,(d-1)^2$ or $d(d-2)$. Therefore, $C$ is
projectively equivalent to type $\ell (d-1), ((d-1) k,\ell k')$ of
Theorem \ref{thm20} $(5)$ ( where $\ell k'$ $(d-1)k$ are coprime
$<\ell (d-1)$) of the form
\begin{eqnarray*}
\,\,X^d&+&Y^d+\sum_{j=2}^{d-2}\,\,\big(X^{d-j}\sum_{i\in S(2)^{j,X}_{m,(a,b)}}\beta_{ji}Y^iZ^{j-i}\big)+\sum_{i\in S_1^{d,X}\,\,{m,\,(a,b)}}\beta_{di}Y^iZ^{d-i}+\\
&+&X\big(\alpha Z^{d-1}+\sum_{i\in
S_1^{d-1,X}\,\,{m,\,(a,b)}}\beta_{(d-1)i}Y^iZ^{d-1-i}\big)=0
\end{eqnarray*}
Claim 1: $\ell (d-1), (d-1,\ell)$ is a generator of any type $\ell (d-1), ((d-1) k,\ell k')$.\\
Let $k\equiv m\,\,(mod\,\,\ell)$, then $[X;\zeta_{\ell (d-1)}^{d-1}Y;\zeta_{\ell (d-1)}^{\ell}Z]^{(k'-m)(d-1)+k'}=[X;\zeta_{\ell (d-1)}^{m(d-1)}Y;\zeta_{\ell (d-1)}^{\ell k'}Z]=[X;\zeta_{\ell (d-1)}^{k(d-1)}Y;\zeta_{\ell (d-1)}^{\ell k'}Z]$.\\\\
Now, we have
\begin{eqnarray*}
S_1^{d,X}\,\,{m,(a,b)}&:=&\{i:\,\,1\leq i\leq d-1\,\,\text{and}\,\,(d-1) i+(d-i)\ell=0\,mod\,\,\ell (d-1)\}\\
&\subseteq&\{i:\,\,1\leq i\leq d-1\,\,\text{and}\,\,(d-1)\,\,|(i-1)\}\\
&=&\{1\}
\end{eqnarray*}
But, since $(d-1) i+(d-i)\ell=(d-1)+(d-1)\ell\neq0\,mod\,\,\ell
(d-1)$ then $S_1^{d,X}\,\,{m,(a,b)}=\phi.$

\begin{eqnarray*}
S_1^{d-1,X}\,\,{m,(a,b)}&:=&\{i:\,\,1\leq i\leq d-1\,\,\text{and}\,\,(d-1) i+(d-1-i)\ell=0\,mod\,\,\ell (d-1)\}\\
&\subseteq&\{i:\,\,1\leq i\leq d-1\,\,\text{and}\,\,(d-1)\,\,|i\}\\
&=&\{d-1\}
\end{eqnarray*}
But since $(d-1) i+(d-1-i)\ell=(d-1)^2\neq0\,mod\,\,\ell (d-1)$ then
$S_1^{d-1,X}\,\,{m,(a,b)}=\phi.$
\begin{eqnarray*}
S(2)^{j,X}_{m,(a,b)}&:=&\{i:\,\,0\leq i\leq j\,\,\text{and}\,\,(d-1)i+(j-i)\ell=0\,mod\,\,\ell(d-1)\}\\
&\subseteq&\{i:\,\,0\leq i\leq j\,\,\text{and}\,\,(d-1)\,\,|j-i\}\\
\end{eqnarray*}
Since $0\leq j-i\leq d-2<d-1$ then $j=i$ which in turns implies that
$i=j=\ell k$ thus $S(2)^{j,X}_{m,(a,b)}=\phi$ if $\ell\nmid j$ and
$\{j\}$ otherwise.

\end{proof}

A similar argument as above is done when $d\equiv 1\ (mod\ \ell)$:

\begin{prop}
Suppose that $d\geq5$ where $d\equiv1\,\,(mod\ell)$ with $\ell\geq
2$ and let $C$ be the non-singular plane curve defined by
\[
X^d+Y^{d-1}Z+\alpha XZ^{d-1}+\sum_{2\leq j=\ell k\leq
d-2}\,\,\beta_{j0}X^{d-j}Z^{j}
\]
Then, $Aut(C)$ is cyclic of order divisible by $\ell(d-1)$.
\end{prop}

\begin{proof}
Since $\sigma:=[X;\zeta_{\ell (d-1)}Y;\zeta_{\ell
(d-1)}^{(\ell-1)(d-1)}Z]\in Aut(C_3)$ of order $\ell (d-1)$ then,
$C_3$ is not a descendant of the Klein curve $K_d$ (because
$\ell(d-1)\nmid3(d^2-3d+3)$) and is not a descendant of the Fermat
curve $F_d$ (because $\ell(d-1)>2d$ for all $\ell\geq3$ and for
$\ell=2$ we have $2(d-1)\nmid6d^2$). Moreover, $Aut(C_3)$ contains a
homology of period $d-1\geq4$ namely, $[X;\zeta_{\ell
(d-1)}^{\ell}Y;Z]=[X;\zeta_{\ell (d-1)}Y;\zeta_{\ell
(d-1)}^{(\ell-1)(d-1)}Z]^{\ell}$ therefore, it must fix a line and a
point off that line. Now, since $[X;\zeta_{\ell (d-1)}^{\ell}Y;Z]\in
Aut(C_3)$ is a homology of period $d-1$ with center $(0;1;0)$ then
by Harui \cite{Harui} $\S 3$, the point $(0;1;0)$ is an inner Galois
point of $C_3$. Moreover, it is the unique such point by Yoshihara
\cite{Yoshihara} $\S 2$ [Theorem $4'$] hence should be fixed by
$Aut(C_3)$. In particular, $Aut(C_3)$ is cyclic.
\end{proof}

And with a similar proof as Proposition \ref{prop121} we obtain,

\begin{prop}
Let $C$ be a non-singular plane curve of degree $d\geq5$ where
$d\equiv1\,\,(mod\, \ell)$ with $2\leq \ell\leq d-1$. Assume that
$Aut(C)$ contains an element of order $\ell (d-1)$. Then $C$ is
isomorphic to
\[
X^d+Y^{d-1}Z+\alpha XZ^{d-1}+\sum_{2\leq j=\ell k\leq
d-2}\,\,\beta_{j0}X^{d-j}Z^{j}
\]
Consequently, $Aut(C)$ is cyclic.
\end{prop}

In particular

\begin{cor}
Let $C$ be a non-singular plane curve of degree $d\geq5$ such that
$Aut(C)$ contains an element of order $\ell(d-1)$ with
$2\leq\ell\leq d-1$. Then, $Aut(C)$ is cyclic and contains a
homology of period $d-1$. In particular, $C$ has an unique inner
Galois point.
\end{cor}

\begin{rem}\label{rem33} If $C$ is a non-singular plane curve of degree $d\geq 5$
such that $Aut(C)$ contains an element $\sigma$ of order $d-1$ then
if $\sigma$ is an homology, then by \cite[Lemma 3.7]{Harui} $C$ has
an inert Galois point $P\in C$, and because it is fixed for $Aut(C)$
because it is unique the inert Galois point by \cite[Theorem
4]{Yoshihara}, then $P\in C$ is fixed by the full $Aut(C)$, and
therefore $Aut(C)$ is a cyclic group by \cite[Lemma 11.44]{Book}. In
the case that $\sigma$ is not an homology (say is of type $m,(a,b)$
with $ab\neq 0$) then $Aut(C)$ is not necessarely a cyclic group and
will depend of the parameters of the defining equation for the given
type $m,(a,b)$.

In this subsection we proved if $C$ non-singular plane curve of
degree $d\geq 5$ such that $Aut(C)$ contains an element of order
$\ell(d-1)$ with $\ell\geq 2$, then $Aut(C)$ has always an homology
of order $d-1$.
\end{rem}

\subsubsection{Plane curves with automorphism of orders $\ell d$.}

\begin{lem}
If a non-singular plane curve $C$ of degree $d\geq5$ has an
automorphism of order $\ell d$ ($2\leq\ell\leq d-1$) then $\ell$
should be a divisor of $d-1$ or $d-2$. In particular,
$d\equiv1\,\,(mod\,\,\ell)$ or $d\equiv2\,\,(mod\,\,\ell)$.

\end{lem}

\begin{proof}
Obvious, since $\ell d$ does not divide $d-1,\,\,d,\,\,d^2-3d+3,\,\,(d-1)^2$.

\end{proof}

\begin{prop}
Suppose that $d\geq 5$ such that $d\equiv1\,\,(mod\,\,\ell)$ for
some $3\leq\ell\leq d-1$. Let $\tilde{C}$ be the non-singular plane
curve of degree $d$ defined by
\[\tilde{C}:\,\,\,X^d+Y^d+\alpha XZ^{d-1}+\sum_{2\leq j=\ell m\leq d-2}\,\,\beta_{j0}X^{d-j}Z^{j}\]
where $\alpha\neq0$. Then, $Aut(\tilde{C})$ should fix a line and a
point off that line.
\end{prop}

\begin{proof}
Since $[X;\zeta_{\ell d}^{\ell}Y;\zeta_{\ell d}^{d}Z]\in
Aut(\tilde´{C})$ of order $\ell d$ then, $Aut(\tilde{C})$ contains a
homology of period $>4$
 namely, $[X;\zeta_{\ell d}^{\ell^2}Y;Z]=[X;\zeta_{\ell d}^{\ell}Y;\zeta_{\ell d}^{d}Z]^{\ell}$ therefore,
  it must fix a line and a point off that line or it fixes a triangle. If it fixes a triangle and neither a point nor line is leaved invariant,
  then $\tilde{C}$ must be a descendant of the Klein curve $K_d$ or the Fermat curve $F_d$. But, since $|Aut(K_d)|=3(d^2-3d+3)$ is not divisible by $\ell d$
   for all $\ell=2,3,...,d-3$ then $\tilde{C}$ is not be a descendent of $K_d$. On the other hand, $Aut(F_d)$ has elements of order at most $2d$
   then $\tilde{C}$ can not be a descendant of the Fermat curve for all $\ell\geq3$.
Therefore, $Aut(C_1)$ must fix a line and a point off that line.
\end{proof}
\begin{rem} The above proposition when $\ell=2$ with a similar proof
concludes that then $Aut(\tilde{C})$ fix a line and a point off that
line or $(\tilde{C},Aut(\tilde{C}))$ is a descendent of Fermat
curve.
\end{rem}

\begin{lem}
Automorphisms of $\tilde{C}$ are of the form
$[\alpha_1X+\alpha_3Z;Y;\gamma_1X+\gamma_3Z]$.
\end{lem}

\begin{proof}
Since $[X;\zeta_{\ell d}^{\ell^2}Y;Z]\in Aut(\tilde{C})$ is a
homology of period $d$ and center $(0;1;0)$ then by Harui
\cite{Harui} $\S 3$, the point $(0;1;0)$ is an outer Galois point of
$\tilde{C}$. Moreover, $\tilde{C}$ is not isomorphic to the Fermat
curve $F_d$ therefore by Yoshihara \cite{Yoshihara} $\S 2$ [Theorem
$4'$], this point is the unique outer Galois point hence should be
fixed by $Aut(\tilde{C})$. In particular, automorphisms of
$\tilde{C}$ are of the form
$[\alpha_1X+\alpha_3Z;\beta_1X+\beta_2Y+\beta_3Z;\gamma_1X+\gamma_3Z]$
which in turns implies that $\beta_1=0=\beta_3$ (Coefficients of
$XY^{d-1}$ and $Y^{d-1}Z$).
\end{proof}

\begin{prop}
Let $C$ be a non-singular plane curve of degree $d\geq5$ where
$d\equiv1\,\,(mod\,\,\ell)$. Assume that $Aut(C)$ contains a cyclic
group of order $\ell d$. Then $C$ is isomorphic to
\[
\tilde{C}:\,\,X^d+Y^d+\alpha XZ^{d-1}+\sum_{2\leq j=\ell m\leq
d-2}\,\,\beta_{j0}X^{d-j}Z^{j}
\]
In particular, $Aut(C)$ consists of elements of the form
$[\alpha_1X+\alpha_3Z;Y;\gamma_1X+\gamma_2Z]$.
\end{prop}
\begin{rem} Unfortunately as in the previous subsections there are
different groups that may appear in $Aut(C)$ depending of the
parameters $\beta_{j0}$.
\end{rem}

\begin{proof}
Clearly, $\ell d$ is not a divisor of $d-1,\,\,d,\,\,d^2-3d+3,\,\,(d-1)^2$ or $d(d-2)$.
Therefore, $C$ is projectively equivalent to type $\ell d, (\ell n,md)$ of Theorem \ref{thm20} $(5)$ ( where $\ell n$ $md$ are coprime $<\ell d$) of the form
\begin{eqnarray*}
\,\,X^d&+&Y^d+\sum_{j=2}^{d-2}\,\,\big(X^{d-j}\sum_{i\in S(2)^{j,X}_{m,(a,b)}}\beta_{ji}Y^iZ^{j-i}\big)+\sum_{i\in S_1^{d,X}\,\,{m,\,(a,b)}}\beta_{di}Y^iZ^{d-i}+\\
&+&X\big(\alpha Z^{d-1}+\sum_{i\in S_1^{d-1,X}\,\,{m,\,(a,b)}}\beta_{(d-1)i}Y^iZ^{d-1-i}\big)=0
\end{eqnarray*}
Claim 1: $\ell d, (\ell,d)$ is a generator of any type $\ell d, (\ell n,md)$.\\
If $n\equiv k'\,\,(mod\,\,\ell)$, then $[X;\zeta_{\ell d}^{\ell}Y;\zeta_{\ell d}^{d}Z]^{(m-k')d+n}=[X;\zeta_{\ell d}^{\ell n}Y;\zeta_{\ell d}^{md}Z]$.\\\\
Claim 2: $S_1^{d,X}\,\,{\ell d,(\ell,d)}=\phi=S_1^{d-1,X}\,\,{\ell d,(\ell ,d)}$\\
Indeed, we have
\begin{eqnarray*}
S_1^{d,X}\,\,{\ell d,(\ell,d)}&:=&\{i:\,\,1\leq i\leq d-1\,\,\text{and}\,\,\ell i+(d-i)d=0\,mod\,\,\ell d\}\\
&\subseteq&\{i:\,\,1\leq i\leq d-1\,\,\text{and}\,\,\ell d\,\,|i(d-\ell)-d\}\\
&=&\phi
\end{eqnarray*}
The last equality because if $i(d-\ell)-d=\ell dk$ then $d|i$ a contradiction.

\begin{eqnarray*}
S_1^{d-1,X}\,\,{\ell d,(\ell ,d)}&:=&\{i:\,\,1\leq i\leq d-1\,\,\text{and}\,\,\ell i+(d-1-i)d=0\,mod\,\,\ell d\}\\
&\subseteq&\{i:\,\,1\leq i\leq d-1\,\,\text{and}\,\,\ell d\,\,|i(d-\ell)\}\\
&=&\phi
\end{eqnarray*}
The last equality because if $i(d-\ell)=\ell dk$ then $d|\ell i$ but because $(d,\ell)=1$ thus $d|i$ a contradiction.\\\\
Finally, $i\in S(2)^{j,X}_{\ell d,(\ell ,d)}$ then $\ell i-(j-i)d=\ell dk$ from which we get $d|i$ thus $i=0$ then $j=\ell k$ hence $j$ is divisible by $\ell$.

\end{proof}

There are also similar statements when $d=\ell k+2\geq 5$ of the
previous results with similar techniques, we state only the result.
\begin{prop}
Let $C$ be a non-singular plane curve of degree $d=\ell k+2$
($\geq5$) such that $Aut(C)$ contains a cyclic group of order $\ell
d$ with $3\leq\ell\leq d-2$. Then $C$ is isomorphic to
\[
\,\,X^d+Y^{d-1}Z+\alpha YZ^{d-1}+\sum_{2\leq i=\ell m+1\leq
d-2}\beta_{di}Y^iZ^{d-i}
\]
and $Aut(C)$ has only elements of the form
$[X;\beta_2Y+\beta_3Z;\gamma_2Y+\gamma_3Z]$.
\end{prop}
\begin{rem} The above proposition with $\ell=2$ we deduce that
$Aut(C)$ fixes a point and a line off (as the previous case) or
$(C,Aut(C))$ is a decendent of the Fermat curve.
\end{rem}
%
%

In this situation where $Aut(C)$ has an element of order $\ell d$
and with \cite[Lemma 3.7]{Harui} we obtain,

\begin{cor}\label{cor44}
Let $C$ be a non-singular plane curve of degree $d\geq5$. Then,
there exists $G\preceq Aut(C)$ cyclic of order $\ell d$ for some
$\ell\geq1$ if and only if $Aut(C)$ contains an homology of order
$d$. In particular, $C$ has a unique outer Galois point $P\notin C$
that should be fixed by $Aut(C)$.

\end{cor}



\section{Tables of types of cyclic groups}

In this section, we introduce tables for lower degree of type of
cyclic groups and equations that are obtained for result of \S 2. It might happen that two types $m, (a,b)$ and $m, (a',b')$ are isomorphic through a permutation of the variables or $F(X;Y;Z)$ decomposes into a product $X.G(X;Y;Z)$ then after removing such cases we get the following tables:

\section{Tables of Type $m(a,b)$ for degree $d\leq 9$}

The following tables are obtained by running the Sage programm
concerning the first theorem in this note, see the programm in
http://mat.uab.cat/$\sim$eslam/CAGPC.sagews

\begin{center}
\begin{table}[!th]
  \renewcommand{\arraystretch}{1.3}
  \caption{Quartics\,\,\,}\label{table:Cyclic Auto42.}
  \vspace{4mm} 
  \centering
\begin{tabular}{|c|c|}
  \hline
  Type: $m, (a,b)$ & $F(X;Y;Z)$ \\\hline\hline
  $12,(3,4)$& $X^4+Y^4+\alpha XZ^3$ \\\hline
  $9,(1,6)$& $X^4+Y^3Z+\alpha XZ^3$ \\\hline
  $8,(1,5)$& $X^4+Y^3Z+\alpha YZ^3$ \\\hline
  $7,(1,5)$& $X^3Y+Y^3Z+\alpha Z^3X$ \\\hline
$6,(2,3)$& $X^4+Z^4+\alpha XY^3+\beta_{2,2}X^2Z^2$ \\\hline
 $4,(1,2)$& $X^4+Y^4+Z^4+\beta_{2,0}X^2Z^2+\beta_{3,2}XY^2Z $ \\\hline
  $4,(0,1)$& $Z^4+L_{4,Z}$ \\\hline
$3,(1,2)$& $X^4+X\big(Z^3+\alpha Y^3\big)+\beta_{2,1}X^2YZ+\beta_{4,2}Y^2Z^2$ \\\hline
  $3,(0,1)$& $Z^3L_{1,Z}+L_{4,Z}$ \\\hline
  $2,(0,1)$& $Z^4+Z^2L_{2,Z}+L_{4,Z}$ \\\hline
  \end{tabular}
\end{table}
\end{center}

\begin{center}
\begin{table}[!th]
  \renewcommand{\arraystretch}{1.3}
  \caption{Quintics\,\,\,}\label{table:Cyclic Auto52.}
  \vspace{4mm} 
  \centering
\begin{tabular}{|c|c|}
  \hline
  Type: $m, (a,\,b)$ & $F(X;Y;Z)$ \\\hline\hline
  $20,(4,5)$& $X^5+Y^5+\alpha XZ^4$ \\\hline
  $16,(1,12)$& $X^5+Y^4Z+\alpha XZ^4$\\\hline
  $15,(1,11)$& $X^5+Y^4Z+\alpha YZ^4$    \\\hline
$13,(1,10)$& $X^4Y+Y^4Z+\alpha Z^4X$    \\\hline
 $10,(2,5)$& $X^5+Y^5+\alpha XZ^4+\beta_{2,0}X^3Z^2$ \\\hline
  $8,(1,4)$& $X^5+Y^4Z+\alpha XZ^4+\beta_{2,0}X^3Z^2$ \\\hline
 $5,(1,2)$& $X^5+Y^5+Z^5+\beta_{3,1}X^2YZ^2+\beta_{4,3}XY^3Z$    \\\hline
  $5,(0,1)$& $Z^5+L_{5,Z}$    \\\hline
%
%
 $4,(1,2)$& $X^5+X\big(Z^4+\alpha Y^4\big)+\beta_{2,0}X^3Z^2+\beta_{3,2}X^2Y^2Z+\beta_{5,2}Y^2Z^3$\\\hline
  $4,(0,1)$& $Z^4L_{1,Z}+L_{5,Z}$    \\\hline
  $3,(1,2)$& $X^5+Y^4Z+\alpha YZ^4+\beta_{2,1}X^3YZ+X^2\big(\beta_{3,0}Z^3+\beta_{3,3}Y^3\big)+\beta_{4,2}XY^2Z^2$ \\\hline
 $2,(0,1)$& $Z^4L_{1,Z}+Z^2L_{3,Z}+L_{5,Z}$    \\\hline
  \end{tabular}
\end{table}

\end{center}


\begin{center}
\begin{table}[!th]
  \renewcommand{\arraystretch}{1.3}
  \caption{Sextics\,\,\,}\label{table:Cyclic Auto62.}
  \vspace{4mm} 
  \centering
\begin{tabular}{|c|c|}
  \hline
  Type: $m, (a,\,b)$ & $F(X;Y;Z)$ \\\hline\hline
$30,(5,6)$& $X^{6} + Y^{6}+\alpha X Z^{5} $ \\\hline
$25,(1,20)$& $X^{6} + Y^{5} Z+\alpha X Z^{5} $ \\\hline
$24,(1,19)$& $X^{6} + Y^{5} Z +\alpha Y Z^{5}$ \\\hline
  $21,(1,17)$& $X^{5} Y + Y^{5} Z+\alpha X Z^{5} $ \\\hline
$15,(5,3)$& $X^{6} + Y^{6}+\alpha X Z^{5}+{\beta_{3,3}} X^{3} Y^{3} $ \\\hline
 $12,(1,7)$& $X^{6} + Y^{5} Z +\alpha Y Z^{5} +{\beta_{6,3}}Y^{3} Z^{5} $ \\\hline
$10,(5,2)$& $X^{6} + Y^{6}+\alpha X Z^{5}+ {\beta_{2,2}}X^{4} Y^{2} + {\beta_{4,4}}X^{2} Y^{4}  $ \\\hline
 $8,(1,3)$& $X^{6} + Y^{5} Z +\alpha Y Z^{5} + {\beta_{4,2}}X^{2} Y^{2} Z^{2} $ \\\hline
$6,(1,2)$& $X^{6} + Y^{6} + Z^{6}+{\beta_{3,0}}X^{3} Z^{3}  + {\beta_{4,2}}X^{2} Y^{2} Z^{2}  + {\beta_{5,4}}X Y^{4} Z $ \\\hline
$6,(1,3)$& $X^{6} + Y^{6} + Z^{6}+{\beta_{2,0}}X^{4} Z^{2}  + {\beta_{6,3}} Y^{3} Z^{3} +  X^{2}{\left({\beta_{4,0}}Z^{4}  + {\beta_{4,3}}Y^{3} Z \right)}  $ \\\hline
 $6,(0,1)$& $Z^{6} + {L_{6,Z}}$ \\\hline
$5,(4,3)$& $X^{6} + X Z^{5}+\alpha X Y^{5}  + {\beta_{3,1}}X^{3} Y Z^{2}  + {\beta_{4,3}}X^{2} Y^{3} Z  +{\beta_{6,2}} Y^{2} Z^{4} $ \\\hline
$5,(4,1)$& $X^{6} + X Z^{5}+\alpha X Y^{5}+{\beta_{2,1}}X^{4} Y Z  + {\beta_{4,2}}X^{2} Y^{2} Z^{2}  + {\beta_{6,3}}Y^{3} Z^{3} $ \\\hline
  $5,(0,1)$& $Z^{5}{L_{1,Z}}  + {L_{6,Z}}$ \\\hline
$4,(1,3)$& $X^{6} + Y^{5} Z +\alpha Y Z^{5} +{\beta_{6,3}}Y^{3}
Z^{5} + {\beta_{2,1}}X^{4} Y Z  +  X^{2}{\left({\beta_{4,0}}Z^{4}  +
{\beta_{4,2}}Y^{2} Z^{2} + {\beta_{4,4}}Y^{4} \right)} $ \\\hline
$3,(0,1)$& $Z^{6} + Z^{3}{L_{3,Z}}  + {L_{6,Z}}$ \\\hline $2,(0,1)$&
$Z^{6} + Z^{4}{L_{2,Z}}  + Z^{2}{L_{4,Z}}+{L_{6,Z}} $
\\\hline
   \end{tabular}
\end{table}

\end{center}


\newpage

\begin{small}
\begin{center}
\begin{table}[!th]
  \renewcommand{\arraystretch}{1.3}
  \caption{degree 7\,\,\,}\label{table:Cyclic Auto72.}
  \vspace{4mm} 
  \centering
\begin{tabular}{|c|c|}
  \hline
  Type: $m, (a,\,b)$ & $F(X;Y;Z)$ \\\hline\hline
  $42,(6,7)$& $X^7+ Y^7+\alpha XZ^6$ \\\hline
$36,(1,30)$& $X^{7} + Y^{6} Z+\alpha X Z^{6}$ \\\hline
  $35,(1,29)$& $X^7+Y^6Z+\alpha YZ^6$ \\\hline
  $31,(1,26)$& $X^{6} Y + Y^{6} Z+\alpha X Z^{6} $ \\\hline
  $21,(3,7)$& $X^7 + Y^7+\alpha XZ^6 + \beta_{3,0}X^4Z^3$ \\\hline
  $18,(1,12)$& $X^{7} + Y^{6} Z+\alpha X Z^{6}+ {\beta_{3,0}}X^{4} Z^{3} $ \\\hline
 $14,(2,7)$& $X^7 + Y^7+\alpha XZ^6 + \beta_{2,0}X^5Z^2 +\beta_{4,0}X^3Z^4$ \\\hline
  $12,(1,6)$& $X^{7} + Y^{6} Z+\alpha X Z^{6}+ {\beta_{2,0}}X^{5} Z^{2}+ {\beta_{4,0}}X^{3} Z^{4}$ \\\hline
  $9,(1,3)$& $X^{7} + Y^{6} Z+\alpha X Z^{6}+ {\beta_{3,0}}X^{4} Z^{3}  + {\beta_{5,3}}X^{2} Y^{3} Z^{2} $ \\\hline
  $7,(1,2)$& $X^{7} + Y^{7} + Z^{7}+{\beta_{4,1}}X^{3} Y Z^{3}  + {\beta_{5,3}}X^{2} Y^{3} Z^{2}  + {\beta_{6,5}}X Y^{5} Z $ \\\hline
 $7,(1,3)$& $X^{7} + Y^{7} + Z^{7}+{\beta_{3,1}}X^{4} Y Z^{2}  + {\beta_{5,4}}X^{2} Y^{4} Z  + {\beta_{6,2}}X Y^{2} Z^{4}  $ \\\hline
 $7,(0,1)$& $Z^{7}+L_{7,Z}$ \\\hline
  $6,(5,4)$& $X^{7} + X Z^{6}+\alpha X Y^{6}+ {\beta_{3,0}}X^{4} Z^{3}+ {\beta_{4,2}}X^{3} Y^{2} Z^{2}+ {\beta_{5,4}}X^{2} Y^{4} Z+ {\beta_{7,2}}Y^{2} Z^{5}$ \\\hline
  $6,(4,3)$& $X^{7} + X Z^{6}+\alpha X Y^{6}+ {\beta_{2,0}}X^{5} Z^{2}  + {\beta_{3,3}}X^{4} Y^{3}  + {\beta_{4,0}}X^{3} Z^{4}  + X^{2} {\beta_{5,3}}Y^{3} Z^{2}+ {\beta_{7,3}}Y^{3} Z^{4}$ \\\hline
  $6,(0,1)$& $Z^{6}L_{1,Z} + L_{7,Z}$ \\\hline
 $5,(1,4)$& $X^7+Y^6Z+\alpha YZ^6+\beta_{2,1}X^{5}YZ+ \beta_{4,2}X^{3} Y^{2} Z^{2}+ \beta_{6,3}X Y^{3} Z^{3}+X^{2}{\left({\beta_{5,0}}Z^{5} + {\beta_{5,5}}Y^{5} \right)}  $ \\\hline
 $4,(1,2)$& $ X^{7} + Y^{6} Z+\alpha X Z^{6}  + {\beta_{2,0}}X^{5} Z^{2}  + {\beta_{3,2}}X^{4} Y^{2} Z  + {\beta_{5,2}}X^{2} Y^{2} Z^{3}  + {\beta_{6,4}}X Y^{4} Z^{2} + {\beta_{7,2}}Y^{2} Z^{5}+ $ \\
 & $+ X^{3}{\left({\beta_{4,0}}Z^{4}  + {\beta_{4,4}}Y^{4} \right)}$ \\\hline
 $3,(1,2)$& $X^{7} + X Z^{6}+\alpha X Y^{6}  + {\beta_{2,1}}X^{5} Y Z  + {\beta_{4,2}}X^{3} Y^{2} Z^{2}  + {\beta_{6,3}}X Y^{3} Z^{3}  + {\beta_{7,2}}Y^{2} Z^{5} + {\beta_{7,5}}Y^{5} Z^{2}+$ \\
 & $X^{4}{{\beta_{3,0}}\left(Z^{3}  + {\beta_{3,3}}Y^{3} \right)}  + X^{2}{\left({\beta_{5,1}}Y Z^{4} + {\beta_{5,4}}Y^{4} Z \right)}$ \\\hline
 $3,(0,1)$& $Z^{6}L_{1,Z}+Z^3L_{4,Z} + L_{7,Z}$ \\\hline
 $2,(0,1)$& $Z^{6}L_{1,Z}+Z^4L_{3,Z}+Z^2L_{5,Z} + L_{7,Z}$ \\\hline

%
%

  \end{tabular}
\end{table}

\end{center}

\end{small}

\newpage

\begin{small}
\begin{center}
\begin{table}[!th]
  \renewcommand{\arraystretch}{1.3}
  \caption{degree 8\,\,\,}\label{table:Cyclic Auto82.}
  \vspace{4mm} 
  \centering
\begin{tabular}{|c|c|}
  \hline
  Type: $m, (a,\,b)$ & $F(X;Y;Z)$ \\\hline\hline
  $56,(7,8)$& $X^{8} + Y^{8}+\alpha X Z^{7}  $ \\\hline
  $49,(1,42)$& $ X^{8} + Y^{7} Z+\alpha X Z^{7}  $ \\\hline
$48,(1,41)$& $X^{8} + Y^{7} Z+\alpha Y Z^{7} $ \\\hline $43,(1,37)$&
$ X^{7} Y + Y^{7} Z+\alpha X Z^{7} $ \\\hline $28,(7,4)$& $X^{8} +
Y^{8}+\alpha X Z^{7}  + {\beta_{4,4}}X^{4} Y^{4} $ \\\hline
$24,(1,17)$& $X^{8} + Y^{7} Z+\alpha Y Z^{7}+{\beta_{8,4}}Y^{4}
Z^{7} $ \\\hline $16,(1,9)$& $X^{8} + Y^{7} Z+\alpha Y
Z^{7}+{\beta_{8,5}}Y^{5} Z^{7}  + {\beta_{8,3}}Y^{3} Z^{7} $
\\\hline $14,(7,2)$& $X^{8} + Y^{8}+\alpha X Z^{7}  +
{\beta_{2,2}}X^{6} Y^{2}  + {\beta_{4,4}}X^{4} Y^{4}  +
{\beta_{6,6}}X^{2} Y^{6} $ \\\hline $12,(1,5)$& $X^{8} + Y^{7}
Z+\alpha Y Z^{7} +{\beta_{8,4}}Y^{4} Z^{7}  + {\beta_{4,2}}X^{4}
Y^{2} Z^{2}$ \\\hline $8,(1,2)$& $X^{8} + Y^{8} +
Z^{8}+{\beta_{4,0}} X^{4} Z^{4} + {\beta_{5,2}}X^{3} Y^{2}
Z^{3}+{\beta_{6,4}}X^{2} Y^{4} Z^{2}+{\beta_{7,6}}X Y^{6} Z $
\\\hline $8,(1,3)$& $X^{8} + Y^{8} + Z^{8} +{\beta_{4,2}}X^{4} Y^{2}
Z^{2}  + {\beta_{8,4}}Y^{4} Z^{4}  +  X^{2}{\left({\beta_{6,1}}Y
Z^{5}  + {\beta_{6,5}}Y^{5} Z \right)} $ \\\hline $8,(1,4)$& $X^{8}
+ Y^{8} + Z^{8}+{\beta_{2,0}}X^{6} Z^{2}  + {\beta_{4,0}}X^{4} Z^{4}
+ {\beta_{5,4}}X^{3} Y^{4} Z  + {\beta_{6,0}}X^{2} Z^{6} +
{\beta_{7,4}}X Y^{4} Z^{3} $ \\\hline $8,(0,1)$& $Z^{8} + {L_{8,Z}}$
\\\hline $7,(6,5)$& $X^{8} + X Z^{7}+\alpha X Y^{7}  +
{\beta_{4,1}}X^{4} Y Z^{3}  + {\beta_{5,3}}X^{3} Y^{3} Z^{2}  +
{\beta_{6,5}}X^{2} Y^{5} Z  + {\beta_{8,2}} Y^{2} Z^{6}$ \\\hline
$7,(6,1)$& $X^{8} + X Z^{7}+\alpha X Y^{7}  + {\beta_{3,1}}X^{5} Y
Z^{2}  + {\beta_{5,4}}X^{3} Y^{4} Z  + {\beta_{6,2}}X^{2} Y^{2}
Z^{4}  + {\beta_{8,5}}Y^{5} Z^{3}  + $ \\\hline $7,(0,1)$&
$Z^{7}{L_{1,Z}}  + {L_{8,Z}}$ \\\hline
$6,(1,5)$& $X^{8} + Y^{7} Z +\alpha Y Z^{7}  + {\beta_{2,1}} X^{6} Y Z  + {\beta_{4,2}}X^{4} Y^{2} Z^{2} +{\beta_{8,4}}Y^{4} Z^{7} $ \\
& $+X^{2}{\left({\beta_{6,0}}Z^{6}  + {\beta_{6,3}}Y^{3} Z^{3}  +
{\beta_{6,6}}Y^{6} \right)}  $ \\\hline $4,(0,1)$& $Z^{8} +
Z^{4}{L_{4,Z}}  + {L_{8,Z}}$ \\\hline
$3,(1,2)$& $X^{8} + Y^{7} Z +\alpha Y Z^{7}  +{\beta_{8,4}}Y^{4} Z^{7}  +  {\beta_{2,1}}X^{6} Y Z  + {\beta_{4,2}}X^{4} Y^{2} Z^{2}  +  X^{5}{\left({\beta_{3,0}}Z^{3}  + {\beta_{3,3}}Y^{3} \right)} + $ \\
& $+ X^{3}{\left({\beta_{5,1}}Y Z^{4}  + {\beta_{5,4}}Y^{4} Z
\right)}  + X^{2}{\left({\beta_{6,0}}Z^{6}  + {\beta_{6,3}}Y^{3}
Z^{3}  + {\beta_{6,6}}Y^{6} \right)}  + X{\left({\beta_{7,2}}Y^{2}
Z^{5}  + {\beta_{7,5}}Y^{5} Z^{2} \right)} $ \\\hline $2,(0,1)$&
$Z^{8} + Z^{6}{L_{2,Z}}  + Z^{4}{L_{4,Z}}  + Z^{2}{L_{6,Z}}  +
{L_{8,Z}}$ \\\hline

%
%

  \end{tabular}
\end{table}

\end{center}

\end{small}

\newpage

\begin{small}
\begin{center}
\begin{table}[!th]
  \renewcommand{\arraystretch}{1.3}
  \caption{degree 9\,\,\,}\label{table:Cyclic Auto92.}
  \vspace{4mm} 
  \centering
\begin{tabular}{|c|c|}
  \hline
  Type: $m, (a,\,b)$ & $F(X;Y;Z)$ \\\hline\hline
$72,(8,9)$& $X^{9} + Y^{9}+\alpha X Z^{8}$ \\\hline $64,(1,56)$&
$X^{9} + Y^{8} Z+\alpha X Z^{8}$ \\\hline $63,(1,55)$& $X^{9} +
Y^{8} Z+\alpha Y Z^{8}$ \\\hline $57,(1,50)$& $X^{8} Y + Y^{8}
Z+\alpha X Z^{8} $ \\\hline $36,(4,9)$& $X^{9} + Y^{9}+\alpha X
Z^{8}+ {\beta_{4,0}}X^{5} Z^{4} $ \\\hline $32,(1,24)$& $X^{9} +
Y^{8} Z+\alpha X Z^{8}+ {\beta_{4,0}}X^{5} Z^{4} $ \\\hline
$24,(8,3)$& $X^{9} + Y^{9}+\alpha X Z^{8}+ {\beta_{3,3}}X^{6} Y^{3}
+ {\beta_{6,6}}X^{3} Y^{6} $ \\\hline $21,(1,13)$& $X^{9} + Y^{8}
Z+\alpha Y Z^{8}  + {\beta_{6,3}}X^{3} Y^{3} Z^{3} $ \\\hline
$18,(2,9)$& $X^{9} + Y^{9}+\alpha X Z^{8}+ {\beta_{2,0}}X^{7} Z^{2}
+ {\beta_{4,0}}X^{5} Z^{4}  + {\beta_{6,0}}X^{3} Z^{6} $ \\\hline
$16,(1,8)$& $X^{9} + Y^{8} Z+\alpha X Z^{8}+ {\beta_{2,0}}X^{7}
Z^{2}  + {\beta_{4,0}}X^{5} Z^{4}  + {\beta_{6,0}}X^{3} Z^{6}  $
\\\hline $12,(4,3)$& $X^{9} + Y^{9}+\alpha X Z^{8}  +
{\beta_{3,3}}X^{6} Y^{3}  + {\beta_{4,0}}X^{5} Z^{4}  +
{\beta_{6,6}}X^{3} Y^{6}  + {\beta_{7,3}}X^{2} Y^{3} Z^{4} $
\\\hline $9,(1,2)$& $X^{9} + Y^{9} + Z^{9}+{\beta_{5,1}}X^{4} Y
Z^{4}  + {\beta_{6,3}}X^{3} Y^{3} Z^{3}  + {\beta_{7,5}}X^{2} Y^{5}
Z^{2}  + {\beta_{8,7}}X Y^{7} Z $ \\\hline $9,(1,3)$& $X^{9} + Y^{9}
+ Z^{9}+{\beta_{3,0}}X^{6} Z^{3}  + {\beta_{5,3}}X^{4} Y^{3} Z^{2}
+ {\beta_{6,0}}X^{3} Z^{6}  + {\beta_{7,6}}X^{2} Y^{6} Z +
{\beta_{8,3}}X Y^{3} Z^{5}  $ \\\hline $9,(0,1)$& $Z^{9} +
{L_{9,Z}}$ \\\hline $8,(7,6)$& $X^{9} + X Z^{8}+\alpha X Y^{8}  +
{\beta_{4,0}}X^{5} Z^{4}  + {\beta_{5,2}}X^{4} Y^{2} Z^{3}  +
{\beta_{6,4}}X^{3} Y^{4} Z^{2}  + {\beta_{7,6}}X^{2} Y^{6} Z  +
{\beta_{9,2}}Y^{2} Z^{7}$ \\\hline $8,(7,4)$& $X^{9} + X
Z^{8}+\alpha X Y^{8} + {\beta_{2,0}}X^{7} Z^{2}  +
{\beta_{4,0}}X^{5} Z^{4}  + {\beta_{5,4}}X^{4} Y^{4} Z  +
{\beta_{6,0}}X^{3} Z^{6}  + {\beta_{7,4}}X^{2} Y^{4} Z^{3}  +
{\beta_{9,4}}Y^{4} Z^{5} $ \\\hline $8,(7,2)$& $X^{9} + X
Z^{8}+\alpha X Y^{8} +{\beta_{3,2}}X^{6} Y^{2} Z  +
{\beta_{4,0}}X^{5} Z^{4}  +  {\beta_{6,4}}X^{3} Y^{4} Z^{2} +
{\beta_{7,2}}X^{2} Y^{2} Z^{5} +  {\beta_{9,6}}Y^{6} Z^{3} $
\\\hline $8,(0,1)$& $Z^{8}{L_{1,Z}}+{L_{9,Z}}$ \\\hline
$7,(1,6)$& $X^{9} + Y^{8} Z +\alpha Y Z^{8}  + {\beta_{2,1}}X^{7} Y Z  + {\beta_{4,2}}X^{5} Y^{2} Z^{2}  + {\beta_{6,3}}X^{3} Y^{3} Z^{3}  + $ \\
& $+{\beta_{8,4}}X Y^{4} Z^{4}+ X^{2}{{\beta_{7,0}}\left(Z^{7}  + {\beta_{7,7}}Y^{7} \right)} $ \\\hline
$6,(2,3)$& $X^{9} + Y^{9} +\alpha X Z^{8}  + {\beta_{2,0}}X^{7} Z^{2}  + {\beta_{3,3}}X^{6} Y^{3}  + {\beta_{4,0}}X^{5} Z^{4}  + {\beta_{5,3}}X^{4} Y^{3} Z^{2}+ $ \\
& $+{\beta_{7,3}}X^{2} Y^{3} Z^{4}  + {\beta_{7,3}}Y^{3} Z^{6}  +  {\beta_{8,6}}Y^{6} Z^{2}  + X^{3}{\left({\beta_{6,0}}Z^{6}  + {\beta_{6,6}}Y^{6} \right)} $ \\\hline
$4,(3,2)$& $X^{9} + X Z^{8} +\alpha X Y^{8}  + {\beta_{2,0}}X^{7} Z^{2}  + {\beta_{3,2}}X^{6} Y^{2} Z  + {\beta_{5,2}}X^{4} Y^{2} Z^{3}  + {\beta_{8,4}}X Y^{4} Z^{4}  + $ \\
& ${\beta_{9,2}}Y^{2} Z^{7}  + {\beta_{9,6}}Y^{6} Z^{3} +
X^{5}{\left({\beta_{4,0}}Z^{4}  + {\beta_{4,4}}Y^{4} \right)}  +
X^{3}{\left({\beta_{6,0}}Z^{6}  + {\beta_{6,4}}Y^{4} Z^{2} \right)}
+ X^{2}{\left({\beta_{7,2}}Y^{2} Z^{5}  + {\beta_{7,6}}Y^{6} Z
\right)} $ \\\hline $4,(0,1)$& $Z^{8}{L_{1,Z}}+ Z^{4}{L_{5,Z}}+
{L_{9,Z}}$ \\\hline $3,(0,1)$& $Z^{9} + Z^{6}{L_{3,Z}}  +
Z^{3}{L_{6,Z}}  + {L_{9,Z}}$ \\\hline $2,(0,1)$& $Z^{8}{L_{1,Z}} +
Z^{6}{L_{3,Z}}  + Z^{4}{L_{5,Z}}  + Z^{2}{L_{7,Z}}  + {L_{9,Z}}$
\\\hline

%
%

  \end{tabular}
\end{table}

\end{center}

\end{small}

\chapter{Automorphism group for non-singular plane curves with degree
5}

\section{Abstract} Henn in \cite{He} obtains the exact list of groups
that appears as automorphism group of a plane non-singular curve of
degree 4 in an algebraic closed field $K$ of zero characteristic,
and also give an equation for such groups. In this chapter we
present the analog of Henn's result for degree 5 non-singular plane
curves. Similar arguments can be applied to deal with higher degree.

\section{Cyclic subgroups for degree 5 non-singular plane curve}

Fix $C$ a non-singular curve of degree 5 over an algebraic closed
field of zero characteristic $K$, and assume that $Aut(C)$ is not
trivial, and write by $F(X,Y,Z)=0$ the curve in $\mathbb{P}^2(K)$.
Also assume $\sigma\in Aut(C)$ is an element of exact order $m$ in
$Aut(C)$ and we may assume, by a change of variables in $K$, that
$\sigma$ maps $(x:y:z)\mapsto (x:\xi_m^a y:\xi_m^b z)$ where $\xi_m$
is a primitive $m$-th root of unity in $K$ and $a,b$ naturals such
that $0\leq a\neq b\leq m-1$ with $a\leq b$ and $gcd(a,b)$ coprime
with $m$ if $ab\neq 0$, (and we can reduce to $gcd(a,b)=1$) and
$gcd(b,m)=1$ if $a=0$, we call type $m,(a,b)$ such automorphism. In
this note with $m\geq 2$, $C_m$ denote also the cyclic group of
order $m$.

Then by a change of variables we may have one of the following
situations by \cite{BaBacyc} (which follow the argument of Dolgachev
did in degree 4 \cite{Dol}) where $L_{i,*}$ means an homogenious
degree $i$ polynomial in the variables $\{X,Y,Z\}$ such that the
variable $*$ does not appears, $\alpha$ is always a non-zero element
(which by a change of variables can be always assumed equal to 1),
and $\beta_{i,j}\in K$:

\begin{center}
\begin{table}[!th]
  \renewcommand{\arraystretch}{1.3}
  \caption{Quintics\,\,\,}\label{table:Cyclic Auto.5}
  \vspace{4mm} 
  \centering
\begin{tabular}{|c|c|}
  \hline
  Type: $m, (a,\,b)$ & $F(X;Y;Z)$ \\\hline\hline
  $20,(4,5)$& $X^5+Y^5+\alpha XZ^4$ \\\hline
  $16,(1,12)$& $X^5+Y^4Z+\alpha XZ^4$\\\hline
  $15,(1,11)$& $X^5+Y^4Z+\alpha YZ^4$    \\\hline
$13,(1,10)$& $X^4Y+Y^4Z+\alpha Z^4X$    \\\hline
 $10,(2,5)$& $X^5+Y^5+\alpha XZ^4+\beta_{2,0}X^3Z^2$ \\\hline
  $8,(1,4)$& $X^5+Y^4Z+\alpha XZ^4+\beta_{2,0}X^3Z^2$ \\\hline
 $5,(1,2)$& $X^5+Y^5+Z^5+\beta_{3,1}X^2YZ^2+\beta_{4,3}XY^3Z$    \\\hline
  $5,(0,1)$& $Z^5+L_{5,Z}$    \\\hline
%
%
  $4,(1,3)$& $X^5+X\big(Z^4+\alpha Y^4+\beta_{4,2}Y^2Z^2\big)+\beta_{2,1}X^3YZ$\\\hline
 $4,(1,2)$& $X^5+X\big(Z^4+\alpha Y^4\big)+\beta_{2,0}X^3Z^2+\beta_{3,2}X^2Y^2Z+\beta_{5,2}Y^2Z^3$\\\hline
  $4,(0,1)$& $Z^4L_{1,Z}+L_{5,Z}$    \\\hline
  $3,(1,2)$& $X^5+Y^4Z+\alpha YZ^4+\beta_{2,1}X^3YZ+X^2\big(\beta_{3,0}Z^3+\beta_{3,3}Y^3\big)+\beta_{4,2}XY^2Z^2$ \\\hline
 $2,(0,1)$& $Z^4L_{1,Z}+Z^2L_{3,Z}+L_{5,Z}$    \\\hline
  \end{tabular}
\end{table}
\end{center}

\section{General properties of the full automorphism group}
Before a detailed study of the automorphism groups for degree 5 we
recall the following results concerning $Aut(C)$ which will be
useful. Because linear systems $g_2$ are unique, we always assume
$C$ is given by a plane equation $F(X,Y,Z)=0$ and $Aut(C)$ is a
finite subgroup of $PGL_3(K)$ which fix the equation $F$. Moreover
$Aut(C)$ satisfies (see Mitchel \cite{Mit}) one of the following
situations:
\begin{enumerate}
\item fixes a point $P$ and a line $L$ with $P\notin L$ in
$PGL_3(K)$,
\item fixes a triangle, i.e. exists 3 points $S:=\{P_1,P_2,P_3\}$ of
$PGL_3(K)$, such that is fixed as a set,
\item $Aut(C)$ is conjugate of a representation inside $PGL_3(K)$ of
one of the finite primitive group namely, the Klein group
$PSL(2,7)$, the icosahedral group $A_5$, the alternating group
$A_6$, the Hessian groups $Hess_{216}$, $Hess_{72}$ or $Hess_{36}$.
\end{enumerate}

It is classically known that if $G$ a subgroup of automorphisms of a
non-singular plane curve $C$ fixes a point on $C$ then $G$ is cyclic
\cite[Lemma 11.44]{Book}, and recently Harui in \cite[\S2]{Harui}
provided the lacked result in the literature on the type of groups
that could appear in non-singular plane curves. Before introduce
Harui statement we need to define descendent of a plane non-singular
curve. For a non-zero monomial $cX^iY^jZ^k$, $c\in K\setminus\{0\}$,
we define its exponent as $max\{i,j,k\}$. For a homogeneous
polynomial $F$, the core of $F$ is defined as the sum of all terms
of $F$ with the greatest exponent. Let $C_0$ be a smooth plane
curve, a pair $(C,G)$ with $G\leq Aut(C)$ is said to be a descendant
of $C_0$ if $C$ is defined by a homogeneous polynomial whose core is
a defining polynomial of $C_0$ and $G$ acts on $C_0$ under a
suitable coordinate system.

\begin{thm}[Harui] \label{teoHarui} If
$G\preceq\,\,Aut(C)$ where $C$ is a non-singular plane curve of
degree $d\geq4$ then $G$ satisfies one of the following
\begin{enumerate}
  \item $G$ fixes a point on $C$ and then cyclic.
  \item $G$ fixes a point not lying on $C$ and it satisfies a short exact sequence of the form
  $$1\rightarrow N\rightarrow G\rightarrow G'\rightarrow 1,$$
with $N$ a cyclic group of order dividing $d$ and $G'$ (which is a
subgroup of $PGL_2(K)$) is conjugate to a cyclic group $C_m$,  a
Dihedral group $D_{2m}$, the alternating groups $A_4$ , $A_5$ or the
permutation group $S_4$, where $m$ is an integer $\leq d-1$.
Moreover, if $G'\cong D_{2m}$, then $m|(d-2)$ or $N$ is trivial.
\item $G$ is conjugate to a subgroup of $Aut(F_d)$ where $F_d$ is the Fermat curve $X^d+Y^d+Z^d$. In particular, $|G|\,|\,6d^2$ and $(G,C)$
is a descendant of $F_d$.
\item $G$ is conjugate to a subgroup of $Aut(K_d)$ where $K_d$ is the Klein curve curve $XY^{d-1}+YZ^{d-1}+ZX^{d-1}$
hence $|G|\,|\,3(d^2-3d+3)$ and $(G,C)$ is a descendant of $K_d$.
\item $G$ is conjugate a finite primitive subgroup of $PGL_3(K)$,
i.e. the Klein group $PSL(2,7)$, the icosahedral group $A_5$, the
alternating group $A_6$, the Hessian groups $Hess_{216}$,
$Hess_{72}$ or $Hess_{36}$ inside $PGL_3(K)$.

\end{enumerate}
\end{thm}

Recall also the following statement \cite[Theorem 2.3]{Harui},
\begin{thm} Given $C$ non-singular plane curve of degree $d\neq4,6$ we have
$$|Aut(C)|\leq 6 d^2.$$
\end{thm}

\begin{cor} Given $C$ non-singular plane curve of degree 5, then
$Aut(C)$ is not conjugate to the Hessian group $Hess_{216}$, the
Klein group $PSL(2,7)$ or the alternating group $A_6.$
\end{cor}

And the following from \cite{BaBacyc}
\begin{prop} Given $C$ a non-singular plane curve of degree $d$ and
let $m$ the order or an element of $Aut(C)$ then $m$ divides one of
the following naturals:
$d-1,\,\,d,\,\,(d-1)^2,\,\,d(d-2),\,\,d(d-1)$ or $d^2-3d+3$.
\end{prop}

Now we recall in the statement different results proved for cyclic
subgroups in $Aut(C)$ obtained by Badr and Bars in \cite{BaBacyc}:
\begin{thm}\label{BB} Let $C$ a non-singular plane curve of degree
$d$, and $\sigma\in Aut(C)$. Then,
\begin{enumerate}
\item if $\sigma$ has order $d(d-1)$, then $Aut(C)=<\sigma>$ and $C$
isomorphic to $X^d+Y^d+\alpha XZ^{d-1}=0$ with $\alpha\neq 0$.
\item if $\sigma$ has order $(d-1)^2$, then $Aut(C)=<\sigma>$ and
$C$ is isomorphic to $X^d+Y^{d-1}Z+\alpha XZ^{d-1}=0$ with
$\alpha\neq 0$.
\item if $\sigma$ has order $d(d-2)$ then $C$ is isomorphic to
$X^d+Y^{d-1}Z+\alpha YZ^{d-1}=0$ with $\alpha\neq 0$ and for $d\neq
4,6$ we have $Aut(C)=<\sigma,\tau|\tau^2=\sigma^{d(d-2)}=1,\, and\,
\tau\sigma\tau=\sigma^{-(d-1)}>.$
\item if $\sigma$ has order $d^2-3d+3$ then $C$ is isomorphic to
$K_d$ and for $d\geq 5$ we have
$$Aut(C)=<\sigma,\tau|\sigma^{d^2-3d+3}=\tau^3=1\, and\,
\sigma\tau=\tau\sigma^{-(d-1)}>.$$
\item if $\sigma$ has order $\ell (d-1)$ with $\ell\geq 2$ then
$Aut(C)$ is cyclic of order $\ell'(d-1)$ with $\ell|\ell'$. If
$\ell=1$, the same conclusion if $\sigma$ is an homology.
\item if $\sigma$ has order $\ell d$ with $\ell\geq 3$ then $Aut(C)$
fix a line and a point off that line with the point not in $C$, and
$Aut(C)$ is an exterior group as in Theorem \ref{teoHarui} (2) with
$N$ of order $d$. When $\ell=2$ may be a decendent of the Fermat
curve or $Aut(C)$ is an exterior group as in Theorem \ref{teoHarui}
(2) with $N$ of order $d$.
\end{enumerate}
\end{thm}

Now, assume as usual $C$ a non-singular plane curve of degree $d=5$
with $\sigma\in Aut(C)$ of exact order $m$ such that acts on
$F(X,Y,Z)=0$ by $(x,y,z)\mapsto (x,\xi_m^a y,\xi_m^bz)$ and we
assume that $m$ is maximal.

The following result determines the the full automorphism groups of such curves with elements of large orders in $Aut(C)$.

\begin{cor} For non-singular plane curves of degree $5$
over an algebraic closed field of zero characteristic we have:
\begin{enumerate}
\item The cyclic group $C_{20}$ appears as $Aut(C)$ inside $PGL_3(K)$ generated by
the transformation $(x,y,z)\mapsto (x,\xi_{20}^4 y,\xi_{20}^5 z)$ up
to conjugation by $P\in PGL_3(K)$, and $C$ is isomorphic (through
$P$) to the plane non-singular curve $X^5+Y^5+\alpha XZ^4$.
\item The cyclic group $C_{16}$ appears as $Aut(C)$ inside
$PGL_3(K)$ generated by the transformation $(x,y,z)\mapsto
(x,\xi_{16} y,\xi_{16}^{12} z)$ up to conjugation by $P\in
PGL_3(K)$, and $C$ is isomorphic (through $P$) to the plane
non-singular curve $X^5+Y^4 Z+\alpha XZ^4$.
\item The group $SmallGroup(30,1)\,\,\cong\,\,<\sigma,\tau|\tau^2=\sigma^{15}=
(\tau\sigma)^2\sigma^{3}=1>$ of order 30 appears as $Aut(C)$ inside $PGL_3(K)$ given by
$\sigma:(x,y,z)\mapsto (x,\xi_{15} y,\xi_{15}^{11} z)$ and
$\tau:(x,y,z)\mapsto (x,\mu z,\mu^{-1}y)$  up to conjugation by $P\in PGL_3(K)$,
and $C$ is isomorphic (through $P$) to the curve $X^5+Y^4 Z+\alpha YZ^4=0$ where $\mu^3=\alpha$.
\item The group $SmallGroup(39,1)\,\,\cong\,\,<\tau,\sigma|\sigma^{13}=\tau^3=1,\,
\sigma\tau=\tau\sigma^{3}>$ of order 39 appears as $Aut(C)$ inside $PGL_3(K)$ given by
$\sigma:(x,y,z)\mapsto (x,\xi_{13} y,\xi_{13}^{10} z)$ and
$\tau:(x,y,z)\mapsto (y,z,x)$ up to conjugation by $P\in PGL_3(K)$,
and $C$ is isomorphic (through $P$) to the curve $K_5:X^4 Y+Y^4 Z+
Z^4 X=0$.
\item The cyclic group $C_{8}$ appears as $Aut(C)$ inside $PGL_3(K)$ generated by
the transformation $(x,y,z)\mapsto (x,\xi_{8} y,\xi_{8}^4 z)$ up to
conjugation by $P\in PGL_3(K)$, and $C$ is isomorphic (through $P$)
to the plane non-singular curve $X^5+Y^4Z+\alpha
XZ^4+\beta_{2,0}X^3Z^2$, with $\alpha\beta_{2,0}\neq 0$.
\end{enumerate}
\end{cor}
\begin{proof} Straightforward the first statements because corresponds to apply theorem
\ref{BB} when the curve $C$ has a cyclic automorphism of order:
$d(d-1)$, $(d-1)^2$, $d(d-2)$ or $d^2-3d+3$ respectively, with $d=5$
by use the table in \S2. The last statement where $Aut(C)$ has a
cyclic automorphism of order $\ell (d-1)$ with $\ell=2$, by apply
Theorem \ref{BB} and table in \S2, we only need to observe that if
$Aut(C)$ should be bigger, then always is cyclic and should be the
group of order 16, therefore the only restriction to impose is
$\beta_{20}\neq 0$ to ensure that the curve has exact automorphism
group $C_8$.
\end{proof}

\section{Determination of the automorphism group with small cyclic
subgroups}

Observe that remains to study $Aut(C)$ for curves $C$ of degree
$d=5$ where its larger order for any element in $Aut(C)$ has order
$2 d$, or $\leq d$ by table in \S2 and the results of the previous
section from Theorem \ref{teoHarui}.

\begin{prop} Suppose that $C$ of degree 5 admits $\sigma\in Aut(C)$
of order 10 as an element of highest order in $Aut(C)$, then we reduce
after conjugation by certain $P\in PGL_3(K)$ that $\sigma$ acts on
$X^5+Y^5+\alpha XZ^4+\beta_{2,0}X^3Z^2$ where $\alpha\beta_{2,0}\neq 0$ as $\sigma:(x,y,z)\mapsto (x,\xi_{10}^2 y,\xi_{10}^5
z)$ and one of the following situations happens:
\begin{enumerate}
\item If $\alpha=5$ and $\beta_{2,0}=10$ then $C$ is isomorphic to the Fermat quintic $F_5: \,\,\,X^5+Y^5+Z^5$ and $Aut(C)$ is conjugate to SmallGroup$(150,5)$. In particular, it is generated by $\eta_1,\eta_2,\eta_3,\eta_4$ of orders $2,3,5,5$ respectively such that $$(\eta_1\eta_2)^2=(\eta_1\eta_3)(\eta_3\eta_1)^{-1}=(\eta_3\eta_4)(\eta_4\eta_3)^{-1}=\eta_1\eta_4^2\eta_1(\eta_3\eta_4)^{-3}=\eta_2\eta_3\eta_2(\eta_4\eta_3)^{-1}=1.$$
\item If $(\alpha,\beta_{2,0})\neq(5,10)$ then $Aut(C)$ is cyclic of order $10$. Moreover, in such case if $1+\alpha+\beta_{2,0}\neq0$ then $C$ is a descendant of the Fermat curve defined by $$PC:\,\,X^5+Y^5+Z^5+\frac{5-3\alpha+\beta_{2,0}}{1+\alpha+\beta_{2,0}}(X^4Z+XZ^4)+\frac{2(5+\alpha-\beta_{2,0})}{1+\alpha+\beta_{2,0}}(X^3Z^2+X^2Z^3).$$
\end{enumerate}
\end{prop}

\begin{proof}
Recall that the condition $\alpha\neq 0$, we can change the variables with a $P$ and
fix a concrete value for $\alpha$ and because an element of $Aut(C)$
where 10 divides the order should be of order 20 from previous
result in last section, and because we fix here the highest order
should be 10 we can assume that $\beta_{2,0}\neq 0$. Thus we reduce
that $C$ has an equation (up to $K$-isomorphism) of the form
$X^5+Y^5+\alpha XZ^4+\beta_{2,0}X^3Z^2=0$
with $\alpha\beta_{2,0}\neq 0$. This curve admits a homology $\sigma^2$ of order $5>3$ therefore $Aut(C)$ fixes a line and a point off that line or $C$ is a descendent of Fermat curve (by Theorem \ref{BB}). Moreover, the center $(0;1;0)$ of this homology is an outer Galois point (by Lemma 3.7 \cite{Harui}) and if $C$ is not isomorphic to the Fermat curve $F_5: X^5+Y^5+Z^5$ then it is unique (Theorem 4' \cite{Yoshihara}) hence should be fixed by $Aut(C)$.
\par Assume first that $Aut(C)$ fixes a line and a point off that line and $C$ is not the Fermat quintic.
In that case, $Aut(C)\subseteq PGL_3(K)$ consists of elements of the form
$[\alpha_1X+\alpha_3Z;Y;\gamma_1X+\gamma_2Z]$ and $Aut(C)$ satisfies
a short exact sequence $1\rightarrow C_5\rightarrow
Aut(C)\rightarrow G'\rightarrow 1,$ where $C_5$ generated by
$\sigma':=[X;\zeta_5^2 Y;Z]$ and $G'$ contains an element of order $2$ from the
projection of the $Aut(C)$ given by $\delta:=[X;Y;\zeta_{10}^5Z]$
and hence $G'$ is conjugate to $C_2, C_4, S_3, A_4, S_4$ or $A_5$. Now, there is no group of order $30$ or $60$ which contains elements of order $10$ and no higher orders therefore $G'$ is not conjugate to $S_3$ or $A_4$. Moreover, assume that $G'$ is conjugate to $S_4$ then $Aut(C)$ is conjugate to $SmallGroup(120,5)$ or $SmallGroup(120,35)$ (because they are the only groups of order $120$ with elements of order $10$ and no higher orders) but there are no elements $\tau\in Aut(C)$ of order $3$ or $10$ that commutes with $\delta$ thus $Aut(C)$ is not conjugate to any of these two groups hence $G'$ can not be conjugate to $S_4$.

On the other hand, groups of order $20$ that contain elements of order $10$ and no higher orders are; $SmallGroup(20,m)$ where $m=1,3,4,5$ and since there is no element $\tau\in Aut(C)$ of order $4$ such that $\sigma'\tau\sigma'=\tau$ or $\sigma'\tau\sigma'^{-1}=\tau\sigma'$ then $m\neq1,3$ moreover there is no element $\tau$ of order 2 in $Aut(C)$ which commutes with $\delta$ therefore $m\neq4,5$ that is $G'$ is not conjugate to $C_4$. Furthermore, groups of order $300$ that contain elements of order $10$ and no higher orders are; $SmallGroup(300,m)$ where $m=25,26,27,41,43$. If $m=43$ or $41$ then $Aut(C)$ contains exactly $3$ element of order $2$ a contradiction (because $Aut(C)$ should have at least $15$ such elements). On the other hand, $m\neq25,26$ or $27$ because there are no elements of order 2 in $Aut(C)$ such that $\tau\delta=\delta\tau$ consequently $G'$ is not conjugate to $A_5$.
\par Finally, we assume that $C$ is a descendent of the Fermat curve. This happens through $P\in PGL_{3}(K)$ such that $P\sigma P^{-1}\in Aut(F_5)$ is of order 10 therefore $P\sigma P^{-1}$ has one of the forms $[X;\zeta_{10}^{2b}Z;\zeta_{10}^{2a}Y],\,\,[\zeta_{10}^{2a}Y;X;\zeta_{10}^{2b}Z]$ or $[\zeta_{10}^{2b}Z:\zeta_{10}^{2a}Y;X]$ with $5\nmid(a+b)$. In what follows, we treat each of these cases.
\par If $P\sigma P^{-1}=\lambda[X;\zeta_{10}^{2b}Z;\zeta_{10}^{2a}Y]$ then $\lambda=\zeta_{10}^{2},\,\,5|a+b+2$ and $P=[\alpha_2Y;\beta_1X+\beta_3Z;\zeta_{10}^{2a+2}\beta_1X+\zeta_{10}^{2a-3}\beta_3Z]$. This transforms $C$ into $F_5'$, a descendant of the Fermat quintic, where $Y^2Z^3$ and $Y^3Z^2$ have coefficients 0 and $\zeta_{10}^{4a-6}\beta_{2,0}  \alpha _2^3 \beta _3^2\neq0$ respectively which is a contradiction (since $[X;\zeta_{10}^{2b}Z;\zeta_{10}^{2a}Y]\in Aut(F_5')$). Therefore, this case should be excluded.
\par If $P\sigma P^{-1}=\lambda[\zeta_{10}^{2a}Y;X;\zeta_{10}^{2b}Z]$ then $\lambda=\zeta_{10}^{2-2b},\,\,5|a-2b+2,a-2b-3$ and
$P=[\alpha_1X+\alpha_3Z;\zeta_{10}^{2-2b}\alpha_1X+\zeta_{10}^{2-2b-5}\alpha_3Z;\gamma_2Y]$. Therefore, $C$ is transformed through $P$ to a descendant of $F_5$ where the monomial $X^5$ appears with coefficient $2\alpha_1^5\neq0$ and $Y^5$ does not appear a contradiction.
\par If $P\sigma P^{-1}=\lambda[\zeta_{10}^{2b}Z;\zeta_{10}^{2a}Y;X]$ then $\lambda=\zeta_{10}^{2-2a},\,\,5|b-2a+2,b-2a-3$ and
$P=[\alpha_1X+\alpha_3Z;Y;\zeta_{10}^{2-2a}\alpha_1X+\zeta_{10}^{-3-2a}\alpha_3Z]$. It suffices to consider the case $a=1$ and $b=0$\footnote{We have $(a,b)=(0,3), (1,0), (2,2), (3,4)$ and the corresponding automorphisms are all conjugate inside $Aut(F_5)$. Indeed, $[X;\zeta_{10}^{2a'}Y;\zeta_{10}^{2b'}Z][Z;\zeta_{10}^{2}Y;X][X;\zeta_{10}^{-2a'}Y;\zeta_{10}^{-2b'}Z]=[\zeta_{10}^{-4b'}Z;\zeta_{10}^{2-2b'}Y;X] $ with $b'=1,4,3$.} that is, $P=[\alpha_1X+\alpha_3Z;Y;\alpha_1X-\alpha_3Z]$. In particular, $C$ is transformed through $P$ to a descendant of the Fermat quintic curve only if $1+\alpha+\beta_{2,0}\neq0$ (The coefficients of $X^5$ and $Y^5$ are $(1+\alpha+\beta_{2,0})\alpha_1^5=1$ and $(1+\alpha+\beta_{2,0})\alpha_3^5=1$ respectively) moreover $\alpha_3=b\alpha_1$ where $b$ is a $5$-th root of unity. Comparing the coefficients of $X^4Z,\,\, XZ^4$ and $X^2Z^3,\,\,X^3Z^2$ we can assume without loss of generality that $b=1$\footnote {otherwise, $5-3\alpha+\beta_{2,0}=0=5+\alpha-\beta_{2,0}$ thus $\alpha=5$ and $\beta_{2,0}=10$ therefore $C$ is transformed into the Fermat quintic itself} and hence we obtain the following curve
$$PC:\,\,X^5+Y^5+Z^5+\frac{5-3\alpha+\beta_{2,0}}{1+\alpha+\beta_{2,0}}(X^4Z+XZ^4)+\frac{2(5+\alpha-\beta_{2,0})}{1+\alpha+\beta_{2,0}}(X^3Z^2+X^2Z^3)$$
Now, if $\alpha\neq5$ or $\beta_{2,0}\neq10$ then $[X;\zeta_{10}^{2b}Z;\zeta_{10}^{2a}Y],\,\,[\zeta_{10}^{2a}Y;X;\zeta_{10}^{2b}Z],\,\,[\zeta_{10}^{2a}Y;\zeta_{10}^{2b}Z;X],\,\,[\zeta_{10}^{2b}Z;X;\zeta_{10}^{2a}Y]\notin Aut(PC).$ Furthermore, $[\zeta_{10}^{2b}Z;\zeta_{10}^{2a}Y;X]$ or $[X;\zeta_{10}^{2a}Y;\zeta_{10}^{2b}Z]\in Aut(PC)$ iff $b=0$ that is $Aut(PC)$ is cyclic of order 10.\\
This completes the proof.

\end{proof}

The following lemma, is very useful to discard all the groups with a
subgroup isomorphic to $C_2\times C_2$ in non-singular curves of
degree 5, in particular help to simplify also the results above:
\begin{lem} \label{noafive} Suppose that $C$ is a non-singular plane curve of degree 5 with
$C_2\times C_2$ as a subgroup of $Aut(C)$. Then such $C$ does not
exists, in particular there is no non-singular plane curve of degree
5 with full automorphism group isomorphic to the groups: $C_2\times
C_2$, $A_4$, $S_4$ and $A_5$.
\end{lem}
\begin{proof}
By Mitchel [Mi] and Harui [Ha] the group $C_2\times C_2$ inside
$PGL_3(K)$ giving invariant a non-singular plane curve $C$ of degree
$d$ should fix a point not belonging to the curve, or is a decendant
of Fermat or Klein curve. Now for $d=5$ could not be a decendant of
Fermat of Klein curve because $4$ does not divide the order of
automorphism of the Fermat of Klein curve of degree 5. Therefore the
subgroup $C_2\times C_2$ fixes a point not in $C$  and because $2$
does not divide the degree $d=5$ by Harui main theorem [Ha] we have
that we can think the elements of $C_2\times C_2$ in a short exact
sequence:

$$1\rightarrow N=1\rightarrow G\rightarrow G\rightarrow 1$$
where $G$ is isomorphic to $C_2\times C_2$ but $N$ is the subgroup
that acts on $X$, thus the subgroup $G$ fixes the variable $X$ in
the equation of $C$ and we can reduce by [Ha] that $G$ only acts in
the variables $Y,Z$ as a matrix in $PGL_2(K)$. Let now
$\sigma,\tau\in G\subseteq GL_2(K)$ of order two, by a
transformation $P\in GL_2(K)$ we can assume that $\sigma=diag(1,-1)$
and $\tau=[[a,b],[c,-a]]\neq\sigma$, (and recall that
$\sigma\tau=\tau\sigma$) and therefore we may assume that the curve
$C$ has a model of type 2(0,1). Now all possible $\tau$ does not
retain invariant the equation of type 2(0,1) for any choose of the
free parameters, obtaining the result, because for commuting with
$\sigma$ the automorphism $\tau=diag(-1,1)$ and therefore $C$ has
the equation: $Z^4 L_{1,Z}+Z^2 L_{3,Z}+L_{5,Z}$ and simultaneously
$Y^4L_{1,Y}+Y^2L_{3,Y}+L_{5,Z}$ which is impossible, or $\tau$ is
the transformation $[X,Y,Z]\mapsto [X,bZ,cY]$ with $bc\neq 0$, but
we obtain that $C$ has simultaneously also the above two equations,
convenient multiplied for some constants which is not possible.
\end{proof}

\begin{prop} Suppose that $C$ of degree 5 admits $\sigma\in Aut(C)$
of order 4 as element of higher order in $Aut(C)$, then we reduce
after conjugation by certain $P\in PGL_3(K)$ one of the following
situations:
\begin{enumerate}
\item $Aut(C)$ is isomorphic to cyclic element of order 4
where $Aut(C)=<\sigma>$ with $\sigma(x,y,z)=(x,\xi_4 y,\xi_4^2 z)$
(up to conjugation by $P\in PGL_3(K)$) and the curve $C$ satisfies
$F(X,Y,Z)=0$ (up to $P$, i.e. up to $K$-isomorphism) given by
$X^5+X(Z^4+\alpha
Y^4)+\beta_{2,0}X^3Z^2)+\beta_{3,2}X^2Y^2Z+\beta_{5,2}Y^2Z^3=0$,
with $\alpha\beta_{5,2}\neq0$.
\item $Aut(C)$ is isomorphic to cyclic element of order 4
where $Aut(C)=<\sigma>$ with $\sigma(x,y,z)=(x, y,\xi_4 z)$
(up to conjugation by $P\in PGL_3(K)$) and the curve $C$ satisfies
$F(X,Y,Z)=0$ (up to $P$, i.e. up to $K$-isomorphism) given by $Z^4
Y+ L_{5,Z}(X,Y)=0$ such that $L_{5,Z}'(X,\zeta_mY')\neq\zeta^r_m
                L_{5,Z}'(X,Y')$ where $(m,r)=(8,1),(16,1)$ or $(20,4)$.
\end{enumerate}
\end{prop}

\begin{proof}
We consider the situations in \S2 concerning types $4, (a,b)$.

We observe first that
$C$ can not be a descendant of the Fermat curve $F_5$ or the Klein curve $K_5$ because $|Aut(F_5)|=150$ and
$|Aut(K_5)|=39$ and $4\nmid|Aut(F_5)|$ or $|Aut(K_5)|$ and $Aut(C)$ is not conjugate to $A_5$ since there are no elements of order 4.
Consequently, $Aut(C)$ is conjugate to $Hess_{36}, Hess_{72}$ or it should fix a line and a point off that line by the
result of Harui. Moreover, for the last case, we need to consider the situation of a short exact sequence of the form
  $$1\rightarrow N=1\rightarrow Aut(C)\rightarrow G'\rightarrow 1,$$
where $G'$ should contain an element of order $4$. That is, $G'$ is
conjugate to a cyclic group $C_4$ or a Dihedral group $D_{8}$ (by
use of the previous lemma \ref{noafive}).
\begin{enumerate}
        \item Type $4, (1,3)$;
here the automorphism group is bigger and also has a subgroup of
order 32 when we impose the coefficient of $X^3YZ$ to be trivial
which would give a contradiction on Harui result. But in Type
$4(1,3)$ the equation is $F(X,Y,Z)=X \cdot G(X,Y,Z)$ and therefore
not of degree 5 plane non-singular curve (which is connected). Thus
this situation can not happen.
\item Type $4, (0,1)$; This curve admits a homology of order $d-1$ with center $(0;0;1)$ then it follows by Harui \cite{Harui}
that this point is an inner Galois point of $C$ and moreover it is unique by Yoshihara \cite{Yoshihara}. Therefore, this point should be fixed
by $Aut(C)$ consequently, $Aut(C)$ is cyclic. It follows by the assumption that
                $C$ is not conjugate to any of the above (In particular, types $8, (1,4),\,\,16, (1,12)$ or $20, (4,5)$)
                then $Aut(C)$ is cyclic of order $4$. To be more precise, for the type $4(0,1)$ we can write it as
                $Z^4(aX+bY)+L_{5,Z}(X,Y)=0$ and with $Y':=aX+bY$ we
                can rewrite it with the same type as
                $Z^4Y'+L_{5,Z}'(X,Y')=0$. Now, it is necessary to
                impose the condition that $L_{5,Z}'(X,\zeta_mY')\neq\zeta^r_m
                L_{5,Z}'(X,Y')$ where $(m,r)=(8,1),(16,1)$ or $(20,4)$ (otherwise; we get a bigger automorphism group
                conjugate to those of types $8,(1,4),\,\,16, (1,12)$ or $20, (4,5)$).
            \item Type $4, (1,2)$; First from the same reason as type $4(1,3)$ we need to assume that $\beta_{5,2}\neq 0$.
            First, we'll show that $Aut(C)$ is not conjugate to any of the Hessian subgroups $Hess_{36}$ or $Hess_{72}.$ Both groups contains reflections but no four groups hence all reflections in the group will be conjugate (\cite{Mit} Theorem $11$). Therefore, it suffices to consider the case $P\,4, (1,2)^2\,P^{-1}=\lambda [Z;Y;X]$ this in turns gives solutions non of them transform $\tilde{\tilde{C}}$ to a smooth curve $P\tilde{\tilde{C}}$ with $\{[X;Z;Y],\,\,[Y;X;Z],\,\,[Z;Y;X]\}\subseteq Aut(P\tilde{\tilde{C}})$. Indeed, $P$ has one of the forms $[\alpha_1X+\alpha_2Y+\alpha_3Z;\beta_1X+\beta_3Z;\alpha_1X-\alpha_2Y+\alpha_3Z]$ or $[\alpha_1X+\alpha_2Y+\alpha_3Z;\beta_2Y;-\alpha_1X+\alpha_2Y-\alpha_3Z]$. For both cases, we must have $\alpha_1=\alpha_3$ (coefficients of $XY^4$ and $Y^4Z$) (in particular, the second case does not occur) moreover, from the coefficients of $X^3Y^2$ and $Y^2Z^3$, we get $\gamma_1=\gamma_2$ a contradiction. Consequently, the claim follows and $Aut(C)$ should fix a line and a point off that line.
             \par Now, if $C$ admits a bigger non-cyclic automorphism group then it should non-commutative by Harui and contain an element
            of order 2 with does not commute with the cyclic element given by the type and we can reduce to a subgroup conjugate to the dihedral
            group $D_8$ and moreover $Aut(C)$ fixes the point $(1;0;0)$ and the line $X=0$. In particular, automorphisms of $C$ are of the form
            $[X;vY+wZ;sY+tZ]$. Since there is no element $\tau\in Aut(C)$ of order 2 such that $\tau\sigma\tau=\sigma^{-1}$ then $Aut(C)$ is not conjugate to $D_8$ or $S_4$. In particular, it is cyclic of order $4$. To be more precise, if $Aut(C)$ is cyclic of order $4k>4$ then $k=2,4$ or $5$. If $k=5$ that is $C$ is projectively equivalent to type $20, (4,5)$ and hence $\sigma$ is projectively equivalent to a homology of order 4 a contradiction (Indeed, $Aut(C')$
                 contains exactly two elements of order $4$ and both are homologies) and similarly, if $k=2$ or $4$.
                  That is, $C$ is not isomorphic to any of the above with the full automorphism group is cyclic of order $4k.$

      \end{enumerate}

\end{proof}

Now, we deal with quintic curves with a cyclic element of order 5
as an element in $Aut(C)$ of highest order.
\begin{prop} Suppose that $C$ has an automorphism of type $5,(1,2)$ and we write $C$ as
 $X^5+Y^5+Z^5+\beta_{3,1}X^2YZ^2+\beta_{4,3}XY^3Z=0$ such that
$\beta_{3,1}\neq0$ or $\beta_{4,3}\neq0$. Suppose the highest order
of an element in $Aut(5)$ is 5, then $C$ is a descendent of the
degree 5 Fermat curve with $Aut(C)$ conjugate to $D_{10}$.
\end{prop}

\begin{proof}
$Aut(C)$ is not conjugate to any of the Hessian groups $Hess_{36}$
or $Hess_{72}$ and is not conjugate to a subgroup of $Aut(K_5)$
since there are no elements of order $5$. On the other hand, always
$C$ admits a bigger automorphism group isomorphic to $D_{10}$
through the transformation $[Z;Y,X]$ (in particular, $Aut(C)$ is not
cyclic). Consequently, $C$ is a descendant of the Fermat quintic or
$Aut(C)$ fixes a line and a point off that line or $Aut(C)$ is
conjugate to $A_5$ (as a finite primitive subgroups of $PGL_3(K)$)
but this last situation is not possible by previous lemma
\ref{noafive}.
\begin{enumerate}
\item $Aut(C)$ fixes a line and a point off that line\\
This line should be $Y=0$ and the point is $(0;1;0)$ because
$<\sigma,\tau>\subset Aut(C)$ with $\sigma(x,y,z)=(x,\xi_5y,\xi_5^2
z)$ and $\tau(x,y,z)=(z,y,x)$.

Hence elements of $Aut(C)$ are of the form $[\alpha_1X+\alpha_3Z;Y;\gamma_1X+\gamma_3Z]$. From the coefficients of $Y^3Z^2$ and $Y^3X^2$ we get $\alpha_1=0=\gamma_3$ or $\alpha_3=0=\gamma_1$. Because $\beta_{3,1}\neq0$ or $\beta_{4,3}\neq 0$ (otherwise; $C$ is the Fermat quintic and hence admits automorphisms of order $10>5$ a contradiction) therefore, $\alpha_{3,1}^5=\gamma_{1,3}^5=1$ and $\alpha_{3,1}\gamma_{1,3}=1$ or $(\alpha_{3,1}\gamma_{1,3})^2=1$ thus $Aut(C)$ has order 10. In particular, $Aut(C)$ is conjugate to $D_{10}$.

\item If $C$ is a descendant of the Fermat curve $F_5$ and neither a line nor a point is leaved invariant.
 Then, through a transformation $P\,\,5,(1,2)\,\,P^{-1}=5, (1,2)$. Indeed, elements of order $5$ in $Aut(F_5)$ are of the form $5, (a,b)$ and if $P\,\,5,(1,2)\,\,P^{-1}=\lambda\,\,5, (a,b)$ then $(a,b)\in\{(1,2),(2,1),(3,4),(4,3),(1,4),(4,1)\}$ but all are conjugate in $Aut(F_5)$. Now, it is straightforward to verify that there are no more automorphisms in $Aut(F_5)\cap Aut(C)$ that is, in such case $Aut(C)$ is conjugate to $D_{10}$.
\end{enumerate}
\end{proof}

\begin{prop} Suppose that $C$ has an automorphism of type $5,(0,1)$ and $C$ has an expression as $Z^5+L_{5,Z}$. Assume that the
highest order element of $Aut(C)$ is 5, then $C$ has $Aut(C)$ cyclic
of order 5.
\end{prop}

\begin{proof}
This curve has a homology $\sigma$ of order $d$ with center $(0;0;1)$ and axis $Z=0$ then (by Harui) this point is an outer Galois point of $C$. Moreover, because automorphisms of $C$ has order $\leq5$ then it is not isomorphic to the Fermat curve thus (by Yoshihara) this Galois point is unique hence should be fixed by $Aut(C)$. In particular, $Aut(C)$ fixes a line ($Z=0$) and a point off that line $(0;0;1)$ and therefore, elements of $Aut(C)$ have the form $[\alpha_1X+\alpha_2Y;\beta_1X+\beta_2Y;Z]$. Furthermore, $Aut(C)$ satisfies a short exact sequence  $$1\rightarrow N\rightarrow Aut(C)\rightarrow G'\rightarrow 1,$$
with $N$ is cyclic of order dividing $5$ and $G'$ is conjugate to $C_m, D_{2m}, A_4, S_4$ or $A_5$ where $m\leq4$ and in case of $D_{2m}$
we have $m|3$ or $N$ is trivial.
\par If $N$ is trivial then $G'$ should be conjugate to $A_5$ (because non of the other groups contains elements of order 5) in particular, $D_{10}$ is a subgroup of $Aut(C)$ but one can easily verify that there exists no elements $\tau\in Aut(C)$ of order $2$ such that $\tau\sigma\tau=\sigma^{-1}$ hence $N$ can not be trivial.
\par If $N$ has order $5$ then for any value of $G'$ (except possibly the trivial group, $C_2$, $C_4$ or $A_4$ such that $Aut(C)$ is conjugate to $D_{10}$, $SmallGroup(20,3)$ or $A_5$) there are elements of order $>5$ in $Aut(C)$ a contradiction. As above $D_{10}$ is not conjugate to a subgroup of $Aut(C)$ therefore $G'$ is conjugate to $C_2$ or $A_4$. On the other hand, there are no elements $\tau\in Aut(C)$ of order $4$ such that $(\tau\sigma)^2=1$ $\sigma\tau\sigma^{-1}=\tau\sigma$ thus $G^{'}$ is not conjugate to $C_4$. Consequently, $Aut(C)$ is cyclic of order 5.

\end{proof}

Remains yet the study of curves where its large automorphism element
has order at most 3. In particular, it is not conjugate to $A_5, Hess_{36}$ or $Hess_{72}$ because each of them contain elements of order $>3$. Therefore, in such case $Aut(C)$ should fix a line and a point off that line or it is conjugate to a subgroup of $Aut(F_5)$ or $Aut(K_5)$.

\begin{prop} Consider a curve of type $3(1,2)$, and we can consider
that
$C:X^5+Y^4Z+YZ^4+\beta_{2,1}X^3YZ+X^2(\beta_{3,0}Z^3+\beta_{3,3}Y^3)+\beta_{4,2}XY^2Z^2=0$,
and assume that the highest order of an element in $Aut(C)$ is 3.
Then $Aut(C)$ is cyclic of order 3 (if $\beta_{3,0}\neq\beta_{3,3}$)
or conjugate to $S_3$ (otherwise).
\end{prop}

\begin{proof}
\item If $C$ is a descendant of the Klein curve then $Aut(C)$ is conjugate to a subgroup of $Aut(K_5)$ and hence can not be of order $>3$.
Indeed, because $|Aut(K_5)|=3^2.7$ then any subgroup of order $>3$ inside $Aut(K_5)$ which does not contain elements of order $>3$ are isomorphic to $C_3\times C_3$ therefore there exist $\tau\in Aut(C)$ of order $3$ that commutes with $\sigma:=[X;\xi_3Y,\xi_3^2Z]$. This gives $\tau$ of the forms $[X;\alpha Y;\mu Z], [Y;\alpha Z;\mu X]$ or $[Z;\alpha X; \mu Y]$ but since the variable $X$ can not be permuted with $Y$ or $Z$ then $\tau=[X;\alpha Y;\mu Z]$ with $\alpha=\omega$ (resp. $\omega^2$) and $\mu=\omega^2$ (resp. $\omega$) a contradiction.
\item If $C$ is a descendant of the Fermat curve then $Aut(C)$ is cyclic of order $3$ or conjugate to $S_3$ inside $Aut(F_5)$. Indeed, $|Aut(F_5)|=2.3.5^2$ hence any subgroup of order $>3$ is conjugate to $S_3$ (note that $Aut(F_5)$ contains no elements of order $6$) or it contains elements of order $5>3$. Now, if $Aut(C)$ is conjugate to $S_3$ then there exists $\tau\in Aut(C)$ of order $2$ such that $\tau\,\,3,(1,2)\tau=3,(2,1)$ which in turns implies that $\tau$ has the form $[X;\beta Z;\beta^{-1}Y]$. But, $[X;\beta Z;\beta^{-1}Y]\in Aut(C)$ iff $\beta^3=1$ and $\beta_{3,0}=\beta_{3,3}.$
\item If $Aut(C)$ fix a point then this point should be one of the reference points $P_1:=(1;0;0),\,\,P_2:=(0;1;0)$ or $P_3:=(0;0;1)$ (because these are the only points which are fixed by $3,(1,2)$). If the fixed points is $P_2$ or $P_3$ then $Aut(C)$ is cyclic of order $3$ since both points belong to $C$. If the fixed point is $P_1$ then the line that is leaved invariant should be $X=0$ hence automorphisms of $C$ have the form $[X;\beta_2Y+\beta_3Z;\gamma_2Y+\gamma_3Z]$. Moreover, it follows by Harui and the assumption that there are no elements in $Aut(C)$ of order $>3$ that $Aut(C)$ satisfies a short exact sequence of the form $$1\rightarrow N\rightarrow Aut(C)\rightarrow G'\rightarrow 1,$$ where $N$ is trivial and $G'$ is conjugate to $C_3, S_3$ or $A_4$. Now, let $\tau:=[X;\beta_2Y+\beta_3Z;\gamma_2Y+\gamma_3Z]$ be an element of order $2$ then it has one of the following forms
    $$\left(
        \begin{array}{ccc}
          1 & 0 & 0 \\
          0 & 1 & 0 \\
          0 & 0 & -1 \\
        \end{array}
      \right),\,\,\,\left(
        \begin{array}{ccc}
          1 & 0 & 0 \\
          0 & -1 & 0 \\
          0 & 0 & 1 \\
        \end{array}
      \right),\,\,\,\left(
        \begin{array}{ccc}
          1 & 0 & 0 \\
          0 & -1 & 0 \\
          0 & 0 & -1 \\
        \end{array}
      \right),\,\,\,\left(
        \begin{array}{ccc}
          1 & 0 & 0 \\
          0 & \beta_2 & \beta_3 \\
          0 & \gamma_2 & -\beta_2 \\
        \end{array}
      \right)
    $$
where $\beta_2^2+\beta_3\gamma_2=1$ and  it suffices to consider the last case because non of the first three transformations retain $C$. If $Aut(C)$ is conjugate to $A_4$ then we can suppose that $(3,(1,2)\,\,\tau)^3=Id$  therefore,
$\left(
  \begin{array}{ccc}
    1 & 0 & 0 \\
    0 & \omega\beta_2 & \omega\beta_3 \\
    0 & \omega^2\gamma_2 & -\omega^2\beta_2 \\
  \end{array}
\right)$ has order 3. That is,
\begin{eqnarray*}
\left(
  \begin{array}{ccc}
    1 & 0 & 0 \\
    0 & \omega\beta_2 & \omega\beta_3 \\
    0 & \omega^2\gamma_2 & -\omega^2\beta_2 \\
  \end{array}
\right)\left(
  \begin{array}{ccc}
    1 & 0 & 0 \\
    0 & \omega\beta_2 & \omega\beta_3 \\
    0 & \omega^2\gamma_2 & -\omega^2\beta_2 \\
  \end{array}
\right)&=&\left(
  \begin{array}{ccc}
    1 & 0 & 0 \\
    0 & \omega^2\beta_2^2+\gamma_2\beta_3 & (\omega^2-1)\beta_2\beta_3 \\
    0 & (1-\omega)\gamma_2\beta_2 & \omega\beta_2^2+\gamma_2\beta_3 \\
  \end{array}
\right)\\
&=& \left(\begin{array}{ccc}
    1 & 0 & 0 \\
    0 & \omega^2+(1-\omega^2)\gamma_2\beta_3 & (\omega^2-1)\beta_2\beta_3 \\
    0 & (1-\omega)\gamma_2\beta_2 & \omega+(1-\omega)\gamma_2\beta_3 \\
  \end{array}
\right)=\left(
  \begin{array}{ccc}
    1 & 0 & 0 \\
    0 & \omega^2\beta_2 & \omega\beta_3 \\
    0 & \omega^2\gamma_2 & -\omega\beta_2 \\
  \end{array}
\right)
\end{eqnarray*}
From which we get $\beta_3=0$ or $\beta_2=\frac{\omega}{\omega^2-1}$. If $\beta_3=0$ then $\beta_2=-1$ and $\gamma_2=0$ (yields no automorphism) and the second possibility implies that $1=\gamma_2\beta_3=\frac{-2\omega}{3}$ a contradiction. Consequently, $Aut(C)$ can not be conjugate to $A_4$.\\
This proves the result.
\end{proof}

\begin{prop} Consider a curve of type $2, (0,1)$ and we can consider
that $C:Z^4L_{1,Z}+Z^2L_{3,Z}+L_{5,Z}=0$. Assume that the highest
order of any element of $Aut(C)$ is two. Then the automorphism group
of $C$ is cyclic of order 2.
\end{prop}

\begin{proof}
$C$ is not a descendant of the Klein curve because $2\nmid |Aut(K_5)|(=63)$.
Also, if $C$ is a descendant of the Fermat curve then $Aut(C)$ can not be conjugate to a bigger subgroup of $Aut(F_5)$
because $|Aut(F_5)|=2.3.5^2$ thus subgroups of order $>2$ should contains elements of order $3$ or $5$ a contradiction.
Finally, if $Aut(C)$ fix a line and a point off that line then by Harui and our assumption that there are no automorphisms of order $>2$ we get that
 $Aut(C)$ satisfies a short exact sequence of the form  $$1\rightarrow N=1\rightarrow Aut(C)\rightarrow G'\rightarrow 1,$$  where $G'$ contain an
  element of order $2$ and no higher orders thus $Aut(C)$ should be conjugate to $C_2$ or $C_2\times C_2$, thus by lemma \ref{noafive} we conclude.
\end{proof}
Now in the previous results we need to ensure that there is an
equation $C$ which the highest order automorphism is of exact order
$m$ in the cases of $m\leq 5$.
\begin{lem} Take $C$ a plane non-singular curves of degree 5 with
the equation given by Type $m,(a,b)$ with $m\leq 5$. Then exist an
equation $C$ of Type $m,(a,b)$ with certain specification value at
the parameters such that the element inside $Aut(C)$ of highest
order is $m$.
\end{lem}
\begin{proof}
The proof is only a computation that the plane equations with bigger
automorphism group that the one obtained imposes restriction in the
locus and does not cover all the specialization of the parameters
which gives a non-singular plane curve. For example for $C$ of Type
5(1,2), they have at least $D_{10}$ in the automorphism group, and
the equation with parameters satisfies such situation have freedom
of two parameters, and the locus of curves with $D_{10}$ as a
strictly subgroup are giving only for a point, and there are
specialization of such parameters given non-singular plane curve
which does not corresponds the point corresponing to the Fermat of
degree 5,
 therefore
$D_{10}$ is realized. A similar computation can be done for the
other situations. The case $m=4$ one can read in \cite{BaBa1}.

\end{proof}

%

Therefore we can conclude as follows, when we list the exact groups
that appears as $Aut(C)$ for $C$ non-plane singular curve of degree
5 (some of them we give in GAP notation), an equation (up to
$K$-isomorphism) with effective action by the group (and some of
them will have exact full automorphism group the predicted by the
group but remain to introduce the parameter restrictions in order to
be the equation geometrically irreducible and do not have a bigger
automorphism group), and the third column corresponds to the
presentation of the group inside $PGL_3(K)$ that the automorphism
group happens:

\begin{center}
\begin{table}[!th]
  \renewcommand{\arraystretch}{1.3}
  \caption{Full Automorphism of quintics}\label{table:Cyclic Auto55.}
  \vspace{4mm} 
  \centering
\begin{tabular}{|c|c|c|}
  \hline
  $Aut(C)$ & $F(X;Y;Z)$ & Generators \\\hline\hline
  $(150,5)$& $X^5+Y^5+Z^5$  & $[\xi_{5}X;Y;Z],[X;\xi_{5}Y;Z]$ \\
           &   & $[X;Z;Y],\,[Y;Z;X]$ \\\hline
  $(39,1)$& $X^4Y+Y^4Z+Z^4X$  & $[X;\xi_{13} Y;\xi_{13}^{10} Z],[Y;Z;X]$ \\\hline
  $(30,1)$ & $X^5+Y^4Z+\alpha YZ^4$  & $[X;\xi_{15} Y;\xi_{15}^{11}Z],[X;\mu Z;\mu^{-1}Y]$   \\\hline
  $C_{20}$& $X^5+Y^5+\alpha XZ^4$ & $[X;\xi_{20}^4 Y;\xi_{20}^5 Z]$   \\\hline
  $C_{16}$& $X^5+Y^4Z+\alpha XZ^4$& $[X;\xi_{16} Y;\xi_{16}^{12} Z]$    \\\hline
$C_{10}$& $X^5+Y^5+\alpha XZ^4+\beta_{2,0}X^3Z^2$& $[X;\xi_{10}^2Y;\xi_{10}^5Z]$\\
& $\beta_{2,0}\neq0$ and $(\alpha,\beta_{2,0})\neq(5,10)$&
\\\hline

 \end{tabular}
\end{table}


\begin{table}[!th]
  \caption{Full Automorphism of quintics}
  \centering

\begin{tabular}{|c|c|c|}


\hline
$D_{10}$& $X^5+Y^5+Z^5+\beta_{3,1}X^2YZ^2+\beta_{4,3}XY^3Z$   & $[X;\xi_5 Y;\xi_5^2 Z],\, [Z,Y,X]$\\
& $(\beta_{3,1},\beta_{4,3})\neq(0,0)$  & \\\hline $C_8$&
$X^5+Y^4Z+\alpha XZ^4+\beta_{2,0}X^3Z^2$\,\,\,$(\beta_{2,0}\neq0)$
& $[X;\xi_{8} Y;\xi_{8}^{4} Z]$ \\\hline
 $S_3$& $X^5+Y^4Z+YZ^4+\beta_{2,1}X^3YZ+X^2\big(Z^3+Y^3\big)+$ & $[X;\xi_3Y;\xi_3^2Z]$ \\
  &$+\beta_{4,2}XY^2Z^2$ (not above) & $[X;Z;Y]$ \\\hline

$C_5$& $Z^5+L_{5,Z}$ (not above)  & $[X;Y;\xi_5Z]$\\\hline

$C_4$& $X^5+X\big(Z^4+\alpha Y^4\big)+\beta_{2,0}X^3Z^2+\beta_{3,2}X^2Y^2Z+\beta_{5,2}Y^2Z^3$ &  $[X;\xi_4Y;\xi_4^2Z]$   \\
& $(\beta_{5,2}\neq0)$\,  (not above) &   \\\hline $C_4$&
$Z^4L_{1,Z}+L_{5,Z}$\,\,(not above)   &  $[X;Y;\xi_4Z]$   \\\hline
 $C_3$& $X^5+Y^4Z+\alpha YZ^4+\beta_{2,1}X^3YZ+$ & $[X;\xi_3Y;\xi_3^2Z]$ \\
&
$+X^2\big(\beta_{3,0}Z^3+\beta_{3,3}Y^3\big)+\beta_{4,2}XY^2Z^2$\,\,$(\beta_{3,0}\neq\beta_{3,3})$
(not above) &  \\\hline
 $C_2$& $Z^4L_{1,Z}+Z^2L_{3,Z}+L_{5,Z}$\,\,(not above)   &  $[X;Y;\xi_2Z]$  \\\hline
  \end{tabular}
\end{table}
\end{center}
 where $\alpha\neq0$ which can take always 1 for a change of variables,
 and $\mu$ satisfies $\mu^3=\alpha$.

\end{document}